\documentclass[preprint]{elsarticle}
\usepackage[utf8]{inputenc} 
\usepackage[T1]{fontenc} 
\usepackage[english,activeacute]{babel} 
\usepackage{ragged2e} 
\usepackage{fancyhdr}
\usepackage[Lenny]{fncychap}
\usepackage{makeidx}
\usepackage{enumitem}
\makeindex

\usepackage{calc}
\usepackage[reqno]{amsmath}
\usepackage{amssymb}
\usepackage{stmaryrd}
\usepackage{amsthm}
\usepackage{txfonts}
\usepackage{mathrsfs}
\usepackage{empheq}

\usepackage{array} 
\usepackage{graphicx} 
\usepackage{epstopdf}
\usepackage{tabularx}
\usepackage{float}

\usepackage{subfig}

\usepackage{caption} 
\usepackage{hyperref} 

\usepackage{color} 
\usepackage[toc,title]{appendix} 

\usepackage{longtable}
\usepackage{bigints}
\usepackage{tikz}
\usetikzlibrary{positioning}
\usetikzlibrary{shapes.geometric, arrows,fit,calc,backgrounds}

\usepackage{mathptmx} 
\usepackage{latexsym}
\usepackage{color,soul}

\usepackage{multicol}
\usepackage{cleveref}

\usepackage{algorithm} 
\usepackage{algorithmic}  
\usepackage[algo2e]{algorithm2e}

\def\r{\mathbb{R}}
\def\n{\mathbb{N}}

\def\qnn{\mathbf{q}}
\def\xnn{\mathbf{x}}
\def\ynn{\mathbf{y}}
\def\wnn{\mathbf{w}}

\def\enn{\mathbf{e}}
\def\trian{\mathfrak{T}}
\def\dt{\Delta t}
\def\bphi{\overline{\phi}}

\def\dim{\texttt{d}}
\def\lstab{\mathcal{L}_{\text{coup}}}
\def\llin{\mathcal{L}_{\text{lin}}}
\def\aef{\varmathbb{A}}
\def\kef{\varmathbb{K}}
\def\ttime{\mathrm{T}}
\def\men{\!\scriptscriptstyle -}
\def\mas{\!\scriptscriptstyle +}
\def\scomp{\texttt{s}}
\def\rcomp{\texttt{r}}

\newtheorem{proposition}{Proposition}
\newtheorem{theorem}{Theorem}
\newtheorem{lemma}{Lemma}
\newtheorem*{remark}{Remark}
\newtheorem{assumption}{Assumption}

\newtheorem{definition}{Definition}

\makeatletter
\newcommand{\leqnomode}{\tagsleft@true\let\veqno\@@leqno}
\newcommand{\reqnomode}{\tagsleft@false\let\veqno\@@eqno}
\makeatother

\newcommand*\patchAmsMathEnvironmentForLineno[1]{%
	\expandafter\let\csname old#1\expandafter\endcsname\csname #1\endcsname
	\expandafter\let\csname oldend#1\expandafter\endcsname\csname end#1\endcsname
	\renewenvironment{#1}{\linenomath\csname old#1\endcsname}{\csname oldend#1\endcsname\endlinenomath}}%
\newcommand*\patchBothAmsMathEnvironmentsForLineno[1]{%
	\patchAmsMathEnvironmentForLineno{#1}%
	\patchAmsMathEnvironmentForLineno{#1*}}%
\patchBothAmsMathEnvironmentsForLineno{equation}%
\patchBothAmsMathEnvironmentsForLineno{align}%
\patchBothAmsMathEnvironmentsForLineno{flalign}%
\patchBothAmsMathEnvironmentsForLineno{alignat}%
\patchBothAmsMathEnvironmentsForLineno{gather}%
\patchBothAmsMathEnvironmentsForLineno{multline}%

\makeatletter
\def\ps@pprintTitle{%
	\let\@oddhead\@empty
	\let\@evenhead\@empty
	\def\@oddfoot{}%
	\let\@evenfoot\@oddfoot}
\makeatother

\usepackage{lineno}
\modulolinenumbers[1]

\begin{document}

\begin{frontmatter}

\title{An adaptive multi-scale iterative scheme for a phase-field model for precipitation and dissolution in porous media}

\author[Hass]{Manuela Bastidas\corref{mycorrespondingauthor}}
\cortext[mycorrespondingauthor]{Corresponding author}
\ead{manuela.bastidas@uhasselt.be}

\author[Stut]{Carina Bringedal}
\author[Hass]{Iuliu Sorin Pop}

\address[Hass]{Faculty of Sciences, UHasselt - Hasselt University. Diepenbeek, Belgium.}
\address[Stut]{Institute for Modelling Hydraulic and Environmental Systems, University of Stuttgart. Stuttgart, Germany.}

\begin{abstract}
	Mineral precipitation and dissolution processes in a porous medium can alter the structure of the medium at the scale of pores. Such changes make numerical simulations a challenging task as the geometry of the pores changes in time in an apriori unknown manner. To deal with such aspects, we here adopt a two-scale phase-field model, and propose a robust scheme for the numerical approximation of the solution. The scheme takes into account both the scale separation in the model, as well as the non-linear character of the model. After proving the convergence of the scheme, an adaptive two-scale strategy is incorporated, which improves the efficiency of the simulations. Numerical tests are presented, showing the efficiency and accuracy of the scheme in the presence of anisotropies and heterogeneities.
\end{abstract}


\end{frontmatter}


\section{Introduction}
\label{intro}

Processes involving precipitation and dissolution in porous media are encountered in many real-life applications. Notable examples in this sense appear in environmental engineering (the management of freshwater in the subsurface), geothermal energy, and agriculture (soil salinization). Particularly challenging for the mathematical modeling and numerical simulations are the situations when the chemistry is affecting the micro structure of the medium, in the sense that the pore geometry and even morphology is altered by dissolution or precipitation. In other words, at the scale of pores (from now on the micro scale), the geometry changes due to chemistry and this also impacts the averaged model behavior at the Darcy-scale (from now on the macro scale).

Mathematical models for dissolution and precipitation in porous media have been extensively discussed in the past decades. In this sense, we mention the model proposed in  \cite{knabner1995analysis}, in which the possibility of having an under- or oversaturated regime is expressed in rigorous mathematical terms. Various mathematical aspects for such models, like the existence and uniqueness of a (weak) solution, the rigorous derivation of the macro-scale model from a micro-scale one, the numerical approximation, or qualitative properties like traveling waves are studied in \cite{knabner1995analysis, moszkowicz1996diffusion, bouillard2007diffusion, kumar2014convergence, agosti2015analysis, kumar2016homogenization,  hoffmann2017existence}. The models discussed there do not take explicitly into account any evolution of the micro-scale geometry. In those cases one can work with the mineral as a surface concentration and the micro-scale volumetric changes in the mineral phase are neglected (see \cite{van2009FREE, kumar2011effective}). At the macro scale, this implies that the porosity does not depend on the solute concentration. An exception is the macro-scale model proposed in \cite{agosti2015analysis}, including an equation relating the changes in the porosity to the (macro-scale) concentration of the mineral. 

Whenever the changes in the mineral layer thickness are large compared to the typical micro-scale length (the size of pores), the micro-scale changes in porosity and morphology cannot be neglected. This impacts the flow at the micro scale, and implicitly the averaged macro-scale quantities which are of primary interest for real-life applications. In this context, upscaling is a natural way to derive macro-scale models incorporating the micro-scale processes accurately. We recall that, due to the chemical processes mentioned above, the structure of the pores (the micro structure) is changing in time, depending on the concentration of the dissolved components, which is a model unknown. In other words, one deals with free boundaries appearing at the micro scale. The challenges related to such models are two-fold; on the one hand, related to the free boundaries, and on the other hand, to the fact that these appear at the micro scale. 

Several approaches are available to account for the evolution of the pore-scale geometry. In one spatial dimension, a free boundary model for dissolution and precipitation in porous media is proposed in \cite{van2008stefan}. There, the existence and uniqueness of a solution are proved. For closely related results, we mention \cite{muntean2009moving, Kota}, where the existence of solutions for similar, one-dimensional free-boundary problems is proved. For the multi-dimensional case, we mention \cite{kumar2013reactive, mabuza2014conservative, mabuza2014nonlinear} where mathematical models for reactive transport models in moving domains are proposed. Similarly, in \cite{mabuza2016modeling} the existence of a solution for a model describing reactive solute transport in deformable two-dimensional channels with adsorption-desorption at the walls is proved, relying on the techniques in \cite{muha2013}.

Whereas the one-dimensional case is quite direct; there are various ways to deal with the (freely) moving boundaries in multiple spatial dimensions. When dealing with simple geometries, like a radially symmetric channel, a layer thickness function can be defined to locate the free boundary. This approach is adopted in \cite{van2009FREE, kumar2011effective, bringedal2015model}. For more complex situations, a level set approach can be considered, as done in  \cite{van2009crystal2, schulz2017strong, schulz2019beyond, bringedal2016upscaling}. Upscaled models can be derived in both cases. For simple geometries, transversal averaging is sufficient, leading to an upscaled model in which the layer thickness is related to the changes in porosity and permeability. For more complex situations, one can apply homogenization techniques. In this case, the upscaled model components and parameters are determined by solving micro-scale (cell) problems involving moving interfaces.

A third option, which inspired the present work, is the phase-field approach. In this case, a thin, diffuse interface layer approximates the freely moving interfaces separating the fluid from the mineral (the precipitate). Building on the idea of minimizing the free energy (see e.g. \cite{caginalp1988dynamics}) the phase-field indicator $\phi$ is an approximation of the characteristic function that approaches $1$ in the fluid phase and $0$ in the mineral phase. In between, a smooth transition zone of width $\lambda>0$ is encountered (see e.g. \cite{ratz2016diffuse}). This approach was considered in \cite{van2011phase} for describing the dissolution and precipitation processes as encountered at the micro scale. There, two phases are encountered (the mineral and the solvent), both being immobile; the solute concentration changes due to chemistry (precipitation and dissolution) and diffusion. An extension to two fluid phases and the mineral is proposed in \cite{redeker2016upscaling}. There, the Darcy-scale counterpart is derived by homogenization techniques but still for the case without fluid motion. The model in \cite{van2011phase} is further extended in \cite{bringedal2019phase} to incorporate fluid flow at the micro scale, and where a Darcy-scale counterpart is derived. In this context, we also mention \cite{redeker2016pod} where model order reduction techniques are employed to build an efficient multi-scale algorithm applicable to the phase-field model proposed in \cite{redeker2016upscaling}. 

Here we focus on the two-scale model in \cite{bringedal2019phase}, in which the so-called cell problems defined at the micro scale are solved for determining the effective parameters appearing in the macro-scale equations modeling the flow and the chemical processes. In other words, we compute effective parameters such as the effective diffusion and the effective permeability tensors to resolve the homogenized problem. These macro-scale quantities are found through local micro-scale problems that depend on the evolution of the phase field at the micro scale. 
The main goal of this paper is to develop a robust multi-scale iterative scheme accounting for both the scale separation and the non-linearities in the model. Although motivated by \cite{bringedal2019phase}, this approach can be applied to other two-scale models resulting from homogenization. Unlike classical multi-scale schemes, e.g., \cite{efendiev2009multiscale}, where one has the same type of equations at both the macro and micro scales, the scheme proposed here allows for different equations at the micro and the macro scale. This approach is hence in line with the heterogeneous multi-scale methods in \cite{engquist2007heterogeneous}. In the present context, we mention the similarities with \cite{garttnernumerical, ray2019numerical}, where a multi-scale scheme is developed for reactive flow and transport in porous media where a level-set is employed to track the evolution of the solid-fluid interface at the micro scale. 

The scheme proposed here is a multi-scale iterative one and relies on the backward Euler (BE) method for the time discretization. The general ideas are presented in \cite{Manuela_proced}. Inspired by \cite{brun2019iterative}, an artificial term is included in the (micro-scale) phase-field equation. This parameter has a stabilizing effect in the coupling with the (macro-scale) flow and reactive transport equations. We mention that, compared to \cite{brun2019iterative}, this coupling is bridging here two different scales. In a simplified setting, we give the rigorous convergence proof of the scheme. This result is obtained without specifying any particular spatial discretization.

To guarantee mass conservation, the mixed finite element method (MFEM) is employed for the spatial discretization at both scales. Since effective quantities are needed for each macro-scale element, the finer the macro-scale mesh is, the more micro-scale problems have to be solved numerically. This increases the computational effort significantly. To deal with this aspect, a macro-scale adaptive strategy is included,  inspired by \cite{redeker2013fast}. The main idea is to select at each time step a representative fraction of the macro-scale points (so-called active nodes), for which the micro-scale cell problems are solved and the effective quantities updated. The results are then transferred to the remaining (inactive) nodes, which are assigned to an active node based on a similarity criterion. A similar approach was also applied in \cite{redeker2016upscaling,garttnernumerical}.

Adaptivity is further applied at the micro scale, where it is crucial to have an accurate description of the diffuse transition zone. In such regions, a fine mesh is necessary to capture the phase-field changes at every time. On the other hand, away from such transition zones, in both the mineral and the fluid phases, the phase field is barely varying. There a coarser mesh is sufficient to obtain an accurate numerical solution. Therefore we use an adaptive mesh that follows the movement of the phase-field transition zone. We start with a coarse micro-scale mesh and apply a prediction-correction strategy as described in \cite{heister2015primal} for a phase-field model for fracture propagation. Finally, since the micro-scale cell problems for the phase field are non-linear, we use a fixed-point iterative scheme called L-scheme, as described in \cite{pop2004mixed, list2016study}. Incorporating this linearization scheme in the multi-scale iterative one mentioned above can be made with no effort, as they both involve similar stabilization terms. Moreover, this scheme has the advantage of being convergent regardless of the starting point and the spatial discretization (the method itself, and the mesh size). Finally, as much the spatial discretization allows it, the iterative scheme guarantees the lower and upper bounds for the phase field.

This paper is organized as follows. In \Cref{sec:twoScal}, the two-scale geometry and the two-scale model are presented briefly. In \Cref{sec:Itscheme0}, we present the iterative scheme and in \Cref{sec:micro}, we introduce the non-linear solver used on the micro-scale problems. In \Cref{sec:Analysis}, we prove the convergence of the multi-scale iterative scheme. The micro- and macro-scale adaptive strategies are described in \Cref{sec:adaptivity}. Finally, in \Cref{sec:numeric}, two numerical test cases are applied in which we study in detail the effect of different choices of parameters.

\subsection{Notations}
\label{sec:pre}

In this paper we use common notations from the functional analysis. For a general domain $\mathfrak{D} \subset \r^\dim$ with $\dim=2,3$, we denote by $L^p(\mathfrak{D})$ the space of the $p-$integrable real-valued functions equipped with the usual norm and by $H^1(\mathfrak{D})$  the Sobolev space of $L^2(\mathfrak{D})$ functions having weak derivatives in the same space.

We let $\left\langle \cdot, \cdot \right\rangle_{\mathfrak{D}}$ represent the inner product on $L^2(\mathfrak{D})$ and norm $\|v\|^2_{L^2(\mathfrak{D})} = \|v\|^2_{\mathfrak{D}}:= \left\langle v, v \right\rangle_{\mathfrak{D}}$. For defining a solution in a weak sense we use the spaces $H^1_\#(\mathfrak{D}) = \left\lbrace \right. p \in H^1(\mathfrak{D})\, |  \, p \text{ is } \mathfrak{D}  \text{-periodic} \left. \right\rbrace$ and $H^1_0(\mathfrak{D}) = \left\lbrace \right. p \in H^1(\mathfrak{D})\, |  \, p=0 \text{ on } \partial\mathfrak{D} \left. \right\rbrace$, with $H^{-1}_\#(\mathfrak{D})$ and $H^{-1}_0(\mathfrak{D})$  being the corresponding dual spaces. When the functions are defined over two domains $\mathfrak{D}_1\subset \r^\dim$ and $\mathfrak{D}_2\subset \r^\dim$ we use the Bochner spaces $L^p(\mathfrak{D}_1; L^q(\mathfrak{D}_2))$ for $p,q \in [1, \infty)$, with the usual norm. 
In the case $p=q=2$ we denote the corresponding norm $ \| v \|_{\mathfrak{D}_1\times\mathfrak{D}_2} := \| v \|_{L^2(\mathfrak{D}_1; L^2(\mathfrak{D}_2))}^2$. 

We use the positive and negative cut of a real number $v$, defined as $[v]_{\mas} := \max( v, 0)$ and $[v]_{\men} := \min( v, 0)$.

\section{The two-scale model}
\label{sec:twoScal}

As mentioned before, we consider the upscaled phase-field model in \cite{bringedal2019phase}. This model describes single-phase flow and reactive transport through a porous medium where the fluid-solid interface evolves due to mineral precipitation and dissolution. 
The macro-scale domain is $\Omega$. It should be interpreted as a homogenized porous medium in which the micro-scale complexities (e.g., the alternating solid and void parts) are averaged out. Following the homogenization procedure, to each macro-scale point $\xnn \in \Omega$, a micro-scale domain $Y$ is assigned, representing an idealization of the complex structure at the micro scale. These micro-scale domains are used to define the cell problems, yielding the effective parameters and functions required at the macro scale.

Following \cite{bringedal2019phase}, the model considered here has been derived by homogenization techniques. At the micro scale the geometry consists of solid grains surrounded by void space (pores). The precipitation and dissolution processes are encountered on the boundary of already existing mineral (grains) and not in the interior of the void space. We assume that the mineral never dissolves entirely and that the void space is always connected; thus the porosity is never vanishing. We refer to \cite{schulzdegenerate, schulz2020degenerate} for the analysis of models, including vanishing porosity and to \cite{bringedal2017effective} for a comparison of different approaches used in the context near clogging.

We write the model in non-dimensional form by following the non-dimensionalization in \cite{bringedal2019phase}. In doing so, we use a local unit cell $Y=[-0.5, 0.5]^\dim$ and to identify the variations at the micro scale we define a fast variable $\ynn$. We associate one micro-scale cell $Y$ to every macro-scale location $\xnn \in \Omega$ (see \Cref{fig:multi-scalephase}).

\begin{figure}[htpb!]
	\centering
	\includegraphics[width=0.8\textwidth]{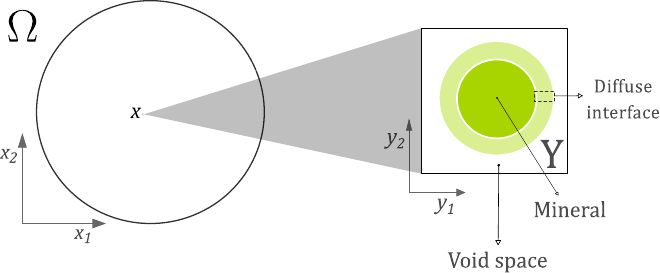}
	\caption{The two-scale domain: the macro scale, homogenized porous medium $\Omega$ (left) and the micro-scale domain $Y$ (right) corresponding to a point $\xnn\in \Omega$.}
	\label{fig:multi-scalephase}
\end{figure}

Following from the upscaling, the unknowns $\qnn(\xnn, t)$, $p(\xnn, t)$ denote the (macro-scale) velocity and pressure in the fluid and $u(\xnn, t)$ is the upscaled solute concentration. At the macro scale the flow is described by
\leqnomode
\begin{equation}\label{eq:macro1}
\tag{\textbf{P}$^\mathrm{M}_p$}
\left\lbrace \begin{aligned}
\nabla\cdot\qnn & =0, & & \text{in } \Omega_{\ttime}:=\Omega \times (0, \ttime], \\
\qnn & = - \kef\nabla p, & & \text{in } \Omega_{\ttime},\\
p &= p_D, & & \text{on } \partial\Omega_{\ttime}:=\partial\Omega \times (0, \ttime], \\
p &= p_I, & & \text{in } \Omega \text{ and }  t=0.
\end{aligned}\right.
\end{equation}
For the solute transport one has
\begin{equation}\label{eq:macro2}
\tag{\textbf{P}$^\mathrm{M}_u$}
\left\lbrace \begin{aligned}
\partial_t(\overline{\phi}(u-u^{\star}))+ \nabla\cdot(\qnn u) &= D\nabla\cdot (\aef\nabla u), & & \text{in } \Omega_{\ttime}, \\
u &= u_D, & & \text{on } \partial\Omega_{\ttime}, \\
u &= u_I, & & \text{in } \Omega \text{ and }  t=0,
\end{aligned}\right.
\end{equation}
where all the spatial derivatives are taken with respect to the macro-scale variable $\xnn$. Here $D$ denotes the pore-scale diffusivity of the solute. In the mineral domain, the mineral is immobile and has a constant concentration $u^{\star}$. Here, $\qnn$ comes from averaging the product of the pore-scale velocity and the phase field $\phi$ over the micro domain $Y$. The variable $\bphi(\xnn, t)$ defines the porosity and the matrices $\aef(\xnn, t)$ and $\kef(\xnn, t)$ are the effective diffusion and permeability, respectively and are determined through local cell problems on $Y$. In other words, the porosity $\bphi$ and the effective parameters $\aef$ and $\kef$ depend on the micro scale in a way that will be explained below.

To derive the macro-scale parameters $\bphi(\xnn, t)$, $\aef(\xnn, t)$ and $\kef(\xnn, t)$, the phase field $\phi(\xnn, \ynn, t)$ is determined for all $\xnn \in \Omega$ by solving the following micro-scale problem
\begin{equation}\label{eq:microphase}
\tag{\textbf{P}$^\mu_{\phi}$}
\left\lbrace \begin{aligned}
\lambda^2 \partial_t \phi + \gamma P'(\phi) & = \gamma \lambda^2 \Delta \phi - 4\lambda\phi(1-\phi)\frac{1}{u^{\star}} f(u), \quad \text{in } Y , \\
\phi \, \, & \text{is } Y\text{-periodic},\\
\phi &= \phi_I, \quad \text{in } Y \text{ and }  t=0,
\end{aligned}\right.
\end{equation}
where all the spatial derivatives are taken with respect to the micro-scale variable $\ynn$. The function $f(u)$ is the reaction rate and $\gamma$ denotes the diffusivity of the interface that separates the fluid and the mineral. Further, $P(\phi)$ denotes the double-well potential, which ensures that the phase field mainly attains values (close to) 0 and 1 for small values of the width of the transition zone $\lambda$. For improving the local conservation of the phase field $\phi$, one may follow \cite{Chen, Carina_proced} and add an additional, $Y$-averaged term in the phase-field equation.

While $\phi$ enters in the micro-scale problems trough the effective parameters defined below, the reverse coupling with the micro scale is given through the reaction rate $f(u)$, with $u$ being constant w.r.t the variable $\ynn \in Y$. The macro-scale porosity in \eqref{eq:macro2} is defined by averaging the phase field
\reqnomode
\begin{equation*}
\bphi(\xnn, t) = \int_Y\phi(\xnn, \ynn, t) d\ynn.
\end{equation*}

In the cell problems, we use a regularized phase field $\phi_\delta := \phi+\delta$ with $\delta>0$ being a small regularization parameter. With this we avoid degeneracies and ensure that the cell problems are well defined. Notice that this regularization parameter only plays a role in the calculation of the effective parameters and does not appear explicitly in \eqref{eq:microphase},\eqref{eq:macro1} and \eqref{eq:macro2}.

The elements of the effective matrices $\aef(\xnn, t)$ and $\kef(\xnn, t)$ are given by
\begin{equation}\label{eq:eff}
\aef_{\rcomp\scomp}(\cdot, t)= \int_{Y} \phi_\delta \left( \delta_{\rcomp\scomp}+\partial_{\rcomp}\omega^\scomp\right) d\ynn \quad \text{and} \quad \kef_{\rcomp\scomp}(\cdot, t)= \int_{Y} \phi_\delta\, \wnn^\scomp_\rcomp d\ynn, 
\end{equation}
for $\rcomp, \scomp\in\{1, \dots, \dim\}$. The functions $\omega^\scomp$ and $\wnn^\scomp=[\wnn^\scomp_1, \dots, \wnn^\scomp_\dim]^t$ solve the following cell problems, defined for each $\xnn\in \Omega$
\leqnomode
\begin{equation}\label{eq:microA}
\tag{\textbf{P}$^{\mu}_{\aef}$}
\left\lbrace
\begin{aligned}
\nabla \cdot (\phi_\delta(\nabla \omega^\scomp + \enn_\scomp)) & = 0, & \text{ in } Y, \\
\omega^\scomp \, \, \text{is } Y\text{-periodic} \quad \text{and }\quad & \int_Y \omega^\scomp d\ynn = 0, 
\end{aligned}\right.
\end{equation}
\begin{equation}\label{eq:microK}
\tag{\textbf{P}$^{\mu}_{\kef}$}
\left\lbrace
\begin{aligned}
\nabla \Pi^\scomp+\enn_\scomp +\mu_f\Delta(\phi_\delta\wnn^\scomp) & =\frac{g(\phi, \lambda)}{\phi_\delta}\wnn^\scomp, & \text{ in } Y, \\
\nabla \cdot (\phi_\delta\wnn^\scomp) & = 0, & \text{ in } Y, \\
\Pi^\scomp \, \, \text{is } Y\text{-periodic} \quad \text{and }\quad & \int_Y \Pi^\scomp d\ynn = 0. \\
\end{aligned}\right.
\end{equation}
Here $\enn_\scomp$ is the $\scomp$-th canonical vector and the function $g(\phi, \lambda)$ ensures that there is zero flow in the mineral phase. This function is such that $g(1, \lambda) = 0$ and $g(0, \lambda)>0$. Here we take $g(\phi, \lambda) := \frac{10K(1-\phi)}{\lambda(\phi+10)}$ with $K=25$ as motivated in \cite{garcke2015numerical}.

Finally, the boundary and initial conditions in \eqref{eq:macro1}, \eqref{eq:macro2} and \eqref{eq:microphase} satisfy the following assumptions
\begin{enumerate}[label={(A\arabic*)}]
	\item The functions $p_D$ and $u_D$ are traces of functions in $H^1(\Omega)$ and $u_D$ is essentially bounded.
	\item The function $u_I\in L^\infty(\Omega)$ is such that $0\leq u_I(\xnn)\leq u^{\star}$ a.e. and the function $\phi_I\in L^\infty(\Omega\times Y)$ is such that $0\leq \phi_I(\xnn,\ynn)\leq 1$ a.e. 
\end{enumerate}
For simplicity, in the following sections we consider homogeneous Dirichlet boundary conditions but the extension to other cases can be done straightforwardly.

\section{The iterative scheme}
\label{sec:Itscheme0}

We propose an iterative scheme to simulate the multi-scale behavior of the phase-field model presented in \Cref{sec:twoScal}. Here we use an artificial coupling parameter between the two scales. In \cite{brun2019iterative, mikelic2013convergence} similar approaches about handling the coupling between scales and non-linear systems of equations can be found.

\subsection{Preliminaries}

For a fixed micro-scale domain $Y$ corresponding to one macro-scale point $\xnn\in \Omega$ the non-linear part of \eqref{eq:microphase}, namely $F:\r\times\r \to \r$, is defined by
\reqnomode
\begin{equation}\label{eq:nonlinearterm}
F(\phi, u) := - \gamma P'(\phi) - 4\lambda\phi(1-\phi)\frac{1}{u^{\star}} f(u).
\end{equation}

Further, we choose the reaction rate $f(u)$ and the double-well potential $P(\phi)$ to be 
\[f(u):= k\left(  \frac{[u]_+^2}{u_{\text{eq}}^2}-1\right) \quad \text{ and } \quad \textbf{} P(\phi):=8\phi^2(1-\phi)^2,\]  
where $u_{\text{eq}}$ is the equilibrium concentration and $k$ is a reaction constant chosen to be $1$. Under this choice and denoting by $\partial_k F$ the partial derivative of $F$ respect to the $k$-th argument, the non-linear term \eqref{eq:nonlinearterm} satisfies the following properties
\begin{enumerate}[label={(F\arabic*)}]
	\item\label{as:a2} For each $u\in \r$, the function $F(\cdot, u)$ is locally Lipschitz continuous with respect to the first argument. There exists a constant $\mathfrak{M}_{F_1} \geq 0$ such that $ |\partial_1 F(s, u) | \leq \mathfrak{M}_{F_1} $ for a.e. $s \in [\alpha, \beta] \subset \r$.
	
	\item\label{as:a3} For each $\phi\in \r$, the function $F(\phi, \cdot)$ is locally Lipschitz continuous with respect to the second argument. There exists a constant $\mathfrak{M}_{F_2} > 0$ such that $|\partial_2 F(\phi, s)| \leq \mathfrak{M}_{F_2}$ for a.e. $s \in [\alpha, \beta] \subset \r$. Moreover, the function $F(\phi, \cdot)$ is such that $F(0,\cdot)= F(1, \cdot)=0$.
	
\end{enumerate}

\begin{proposition}\label{prop:nonlinear}
	For each $u\in \r$, the function $F(\cdot, u)$ is continuous and can be decomposed as $F(\cdot, u) := F_{+}(\cdot, u) + F_{\men}(\cdot, u)$ with $F_{\mas}(\cdot, u)$ denoting the increasing part of $F(\cdot, u)$ and $F_{\men}(\cdot, u)$ the decreasing part of $F(\cdot, u)$
	\begin{align*}
	F_{\mas}(\alpha, u) = \int_0^{\alpha} [ \partial_1 F(z\!, u)]_{\mas} \, dz, \text{ and }
	F_{\men}(\alpha, u) = \int_0^{\alpha} [ \partial_1 F(z\!, u)]_{\men} \, dz.
	\end{align*}
\end{proposition}

In Section \eqref{sec:micro} we propose a micro-scale non-linear solver and there, the splitting of the non-linear term in \Cref{prop:nonlinear} guarantees the convergence. In the following sections we treat $F_{\men}$ implicitly and $F_{\mas}$ explicitly. A similar strategy splitting the non-linearities into their convex and concave components can be found in \cite{frank2018finite}.

\subsection{The multi-scale iterative scheme}
\label{sec:Itscheme}
The multi-scale iterative scheme is inspired by \cite{brun2019iterative}, where a stabilizing term involving the parameter $\lstab > 0$ is added to the micro-scale phase-field equation. In \Cref{sec:Analysis} we show that choosing $\lstab>0$ guarantees the convergence of the scheme.

We let $N\in \n$ be the number of time steps and $\dt=\ttime/N$ be the time step size. For $n\in\{1, \dots, N\}$, we define $t^n=n\dt$ and denote the time-discrete solutions by $\phi^n:= \phi(\cdot, \cdot, t^n)$ and $\nu^n := \nu(\cdot, t^n)$ for $\nu \in \left\lbrace \aef, \kef, p, \qnn, u \right\rbrace$.

Applying the Euler implicit discretization, at each time a fully coupled non-linear system of equations has to be solved. For each $n>0$, the iterative algorithm defines a multi-scale sequence $\left\lbrace \phi^n_{i}, \aef^n_i, \kef^n_i, p^n_i, \qnn^n_i, u^n_i\right\rbrace$ where $i>0$ is the iteration index. Naturally, the initial guesses for $\phi^n_{0}$ and $u^n_0$ are $\phi^{n\!-\!1}$ and $u^{n\!-\!1}$. We call $\phi^{0}$ and $u^0$ the initial conditions of $\phi$ and $u$.

The iterative scheme is defined as follows. First, for a given $u^{n\!-\!1}$, $u^n_{i-1}$, $\bphi^{n\!-\!1}$ and $\phi^n_{i-1}$, one solves the micro-scale phase-field problem
\leqnomode
\begin{equation}\label{eq:microphase_i}
\tag{\textbf{P}$^{\mu, i}_{\phi}$}
\left\lbrace\begin{aligned}
 \phi^n_i - \dt \gamma \Delta \phi^n_i - \frac{\dt}{\lambda^2} F_{\men}(\phi^n_i, u^n_{i\!-\!1}) &+ \lstab \left( \phi^n_i-\phi^n_{i\!-\!1}\right) \\ &= \phi^{n\!-\!1} + \frac{\dt}{\lambda^2} F_{\mas}(\phi^{n\!-\!1}, u^n_{i\!-\!1}), \quad \text{in } Y, \\
 \qquad \phi^n_i \quad &\text{is } Y\text{-periodic},
\end{aligned}\right.
\end{equation}
By using the solution $\phi^n_i$ in \eqref{eq:eff}, \eqref{eq:microA} and \eqref{eq:microK} we calculate the iterative effective parameters $\aef^n_i$ and $\kef^n_i$. Then, one continues with the macro-scale problems
\leqnomode
\begin{equation}\label{eq:macro1_i}
\tag{\textbf{P}$^{\mathrm{M}, i}_{p}$}
\left\lbrace\begin{aligned}
\nabla\cdot\qnn^n_i & =0, & & \text{in } \Omega, \\
\qnn^n_i & = - \kef^n_i\nabla p^n_i, & & \text{in } \Omega, \\
p^n_i &= 0, & & \text{on } \partial\Omega, 
\end{aligned}\right.
\end{equation}
\begin{equation}\label{eq:macro2_i}
\tag{\textbf{P}$^{\mathrm{M}, i}_{u}$}
\left\lbrace\begin{aligned}
\overline{\phi}^n_i(u^n_i-u^{\star}) &+ \dt \nabla\cdot(\qnn^n_iu^n_i) \\ & = \dt D\nabla\cdot (\aef^n_i\nabla u^n_i) + \overline{\phi}^{n\!-\!1}(u^{n\!-\!1}\!-\!u^{\star}), & & \text{ in } \Omega, \\
u^n_i &= 0, & & \text{ on } \partial\Omega.
\end{aligned}\right.
\end{equation}

\paragraph{\textbf{The iterative scheme}}
For $n>0$ and $i > 0$ with given $u^{n\!-\!1}$, $u^n_{i-1}$, $\bphi^{n\!-\!1}$ and $\phi^n_{i-1}$, one performs the following steps
\begin{enumerate}[label={(S\arabic*)}]
	\item\label{step1} For each $\xnn \in \Omega$, find $\phi^n_i$ by solving the phase-field problem \eqref{eq:microphase_i}.

	\item\label{step2} Given $\phi^n_{i}$, find the effective matrices $\aef^n_i$ and $ \kef^n_i$ in \eqref{eq:eff} by solving the cell problems \eqref{eq:microA} and \eqref{eq:microK}.
	
	\item\label{step3} Given $\kef^n_i$ and $\aef^n_i$, find $p^n_i$, $\overline{\qnn^\phi}^n_i$ and $u^n_i$ by solving the macro-scale problems \eqref{eq:macro1_i} and \eqref{eq:macro2_i}.
\end{enumerate}

The multi-scale iterations in steps \ref{step1} - \ref{step3} take place until one reaches a prescribed threshold $\textit{tol}_M>0$ for the following $L^2$-norm
\[\epsilon_M^{n, i}:= \| \bphi^{n}_i - \bphi^{n}_{i\!-1\!} \|_\Omega  \leq \textit{tol}_M. \]
We highlight that this stopping criterion is chosen according to the results in \Cref{te:teorema} in \Cref{sec:Analysis}. Observe that the convergence of the porosity $\bphi^{n}_i$ guarantees the convergence of the macro-scale concentration $u^{n}_i$, so the stopping criterion above is sufficient. However, different stopping criteria can also be used, including e.g., the residuals of the macro-scale concentration and velocity.

Proving the existence and uniqueness of a solution to the coupled system \eqref{eq:macro1}, \eqref{eq:macro2}, \eqref{eq:microphase}, \eqref{eq:microA} and \eqref{eq:microK} is beyond the scope of this paper. Such results are known if each model component is considered apart. For example, when taken individually the problems \eqref{eq:macro1}, \eqref{eq:macro2}, \eqref{eq:microA} and \eqref{eq:microK} are linear and elliptic, while the non-linearity in \eqref{eq:microphase} is monotone and Lipschitz continuous. For such problems the existence and uniqueness of a weak solution is guaranteed by standard arguments. The same holds for \eqref{eq:macro1_i} and \eqref{eq:macro2_i}. For the parabolic counterparts, before applying the time discretization, we refer to \cite{friedman1988combustion, friedman1992transport, muntean2010multiscale, redeker2016upscaling}. There the existence and uniqueness of solutions to similar problems related to phase field modeling or the interaction between scales are addressed.

\section{The micro-scale non-linear solver}
\label{sec:micro}

The multi-scale iterative scheme in steps \ref{step1} - \ref{step3} includes a non-linear problem at the micro scale. At each time and for each $\xnn \in \Omega$, the iterations of the multi-scale scheme require solving the micro-scale non-linear problem \eqref{eq:microphase_i} in the micro-scale domain $Y$.
For this we construct an iterative non-linear solver based on the L-scheme \cite{pop2004mixed, list2016study},  which is a contraction-based approach. The main advantages of the L-scheme are that, unlike the Newton method, this method does not involve the calculation of derivatives and its convergence is guaranteed regardless of the initial approximation, the spatial discretization, or the mesh size.

To be specific, for a fixed $\xnn \in \Omega$ and $n>0$, let $\phi^{n\!-\!1}\in L^2(Y)$ and the concentration $u^n(\xnn)$ be given. The weak solution of the problem \eqref{eq:microphase} is defined as follows
\begin{definition}
	A weak solution to the problem \eqref{eq:microphase} is a function $\phi^{n}\in H_\#^1(Y)$ satisfying
	\reqnomode
	\begin{equation}\label{eq:phasef_w}
	\begin{aligned}
	&\langle \phi^n, \psi \rangle_Y + \dt \gamma \langle \nabla \phi^n, \nabla \psi \rangle_Y - \frac{\dt}{\lambda^2} \langle F_{\men}(\phi^n, u^n), \psi \rangle_Y \\ & \quad = \langle\phi^{n\!-\!1}\! +\! \frac{\dt}{\lambda^2}F_{\mas}(\phi^{n\!-\!1}, u^n), \psi\rangle_Y, 
	\end{aligned}
	\end{equation}
	for all $\psi\in H_\#^1(Y)$.
\end{definition}
Further, let $i>0$ be a multi-scale iteration index and assume $\phi^{n}_{i\!-\!1}\in L^2(Y)$ and $u^{n}_{i-1}(\xnn) \in \r$ known. The weak solution of the problem \eqref{eq:microphase_i} is defined in
\begin{definition}
	A weak solution to the problem \eqref{eq:microphase_i} is a function $\phi^{n}_i\in H_\#^1(Y)$ satisfying
	\reqnomode
	\begin{equation}\label{eq:phasef_wit}
	\begin{aligned}
	& \langle \phi^n_i, \psi \rangle_Y + \dt \gamma \langle \nabla \phi^n_i, \nabla \psi \rangle_Y - \frac{\dt}{\lambda^2} \langle F_{\men}(\phi^n_i, u^n_{i\!-\!1}), \psi \rangle_Y + \left\langle \lstab \left( \phi^n_i\! - \!\phi^n_{i\!-\!1} \right), \psi \right\rangle_Y \\
	& \qquad = \langle\phi^{n\!-\!1}\! +\! \frac{\dt}{\lambda^2}F_{\mas}(\phi^{n\!-\!1}, u^n_{i\!-\!1}), \psi\rangle_Y, 
	\end{aligned}
	\end{equation}
	for all $\psi\in H_\#^1(Y)$.
\end{definition}
Observe that \eqref{eq:microphase_i} is a non-linear problem, which is solved numerically by employing a linear iterative scheme. To this aim we take $\llin\in \r^+$ such that $\llin \geq \mathfrak{M}_{F_1}$ and let $j\in \n$, $j\geq1$ be a micro-scale iteration index. The weak solution of the linear problem associated to \eqref{eq:microphase_i} is defined as
\begin{definition} 
	A weak solution to the linearized version of problem \eqref{eq:microphase_i} is a function $\phi^n_{i, j} \in H^1_\#(Y)$ satisfying
	\reqnomode
	\begin{equation}\label{eq:phasef_L}
	\begin{aligned}
	& \left\langle (1+\lstab)\phi^n_{i, j}, \psi \right\rangle_Y + \dt \gamma \langle \nabla \phi^n_{i, j}, \nabla \psi \rangle_Y - \frac{\dt}{\lambda^2} \langle F_{\men}(\phi^n_{i, j\!-\!1}, u^n_{i\!-\!1}), \psi \rangle_Y \\ & \qquad + \frac{\dt}{\lambda^2}\langle\llin(\phi^n_{i, j}-\phi^n_{i, j\!-\!1}), \psi \rangle_Y = \langle\phi^{n\!-\!1} + \frac{\dt}{\lambda^2} F_{\mas}(\phi^{n\!-\!1}, u^n_{i\!-\!1})+ \lstab\phi^n_{i\!-\!1} , \psi\rangle_Y, 
	\end{aligned}
	\end{equation}
	for all $\psi \in H^1_\#(Y)$.
\end{definition}

The natural choice for the initial micro-scale iteration $\phi^{n}_{i, 0}$ is $\phi^{n}_{i\!-\!1}$. Nevertheless, this choice is not compulsory for the convergence of the micro-scale linear solver as the convergence is independent of the initial guess. The iterations \eqref{eq:phasef_L} are performed until one reaches a prescribed threshold $\textit{tol}_\mu\ll \textit{tol}_M$ for the following $L^2$-norm 
\reqnomode
\begin{equation}\label{eq:tol_mu}
\epsilon_\mu^{n, i, j} := \| \phi^{n}_{i, j}(\xnn,\cdot) - \phi^{n}_{{i, j\!-1\!}}(\xnn,\cdot) \|_Y \leq \textit{tol}_\mu
\end{equation}
where $i>0$ is the iteration index of the multi-scale scheme and $j>0$ indicates the micro-scale iterations of the non-linear solver.

We highlight that in this specific case and due to the strong coupling between the flow, chemistry and the phase field over two scales, an accurate solution of the micro-scale problems is crucial to achieve convergence of the multi-scale iterations. For this reason we solve the micro-scale non-linear problem at every multi-scale iteration and take $\textit{tol}_\mu\ll \textit{tol}_M$. 

Now we show that the solution of the phase-field problem \eqref{eq:microphase_i} at every $\xnn \in \Omega$ remains bounded. For a fixed $\xnn \in \Omega$ and $n>0$, let $\phi^{n\!-\!1}\in L^2(Y)$ and a certain concentration $u^n(\xnn)$ be given. Further, let $i>0$ be a multi-scale iteration index and $\phi^{n}_{i\!-\!1}\in L^2(Y)$ be given. 

\begin{lemma}[Maximum principle for the phase-field]
	\label{lem:maxphase}
	Assume $\phi^{n\!-\!1}$, $\phi^n_{i\!-\!1}$ and $ \phi^n_{i, j\!-\!1}\in L^\infty(Y)$ are all essentially bounded by $0$ and $1$. Then $\phi^n_{i, j}\in H^1_\#(Y)$ solving \eqref{eq:phasef_L} satisfies the same essential bounds.
\end{lemma}
\begin{proof}
	First, we test in \eqref{eq:phasef_L} with $\psi:= [\phi_{i, j}^n]_{\men}$, then
	\reqnomode
	\begin{equation}\label{max2}
	\begin{aligned}
	& \left( 1+\lstab+\frac{\dt}{\lambda^2} \llin \right) \| [\phi_{i, j}^n]_{\men}\|^2_Y + \dt \gamma \| \nabla [\phi_{i, j}^n]_{\men}\|^2_Y \\ & \qquad = \langle\phi^{n\!-\!1} + \frac{\dt}{\lambda^2} F_{\mas}(\phi^{n\!-\!1}, u^n_{i\!-\!1}) + \lstab\phi_{j\!-\!1}^n, [\phi_{i, j}^n]_{\men}\rangle_Y \\ & \qquad \qquad + \frac{\dt}{\lambda^2} \langle F_{\men}(\phi^n_{i, j\!-\!1}, u^n_{i\!-\!1})+\llin\phi^n_{i, j\!-\!1}, [\phi_{i, j}^n]_{\men} \rangle_Y.
	\end{aligned}
	\end{equation}
	Using \ref{as:a3} and the mean value theorem on the right hand side of \eqref{max2} one obtains
	\begin{equation}\label{reml1}
	\begin{aligned}
	& \langle\phi^{n\!-\!1} + \frac{\dt}{\lambda^2}F_{\mas}(\phi^{n\!-\!1}\!, u^n_{i\!-\!1}) + \lstab\phi_{j\!-\!1}^n, [\phi_{i, j}^n]_{\men}\rangle_Y \\
	& \qquad =\langle(1+\frac{\dt}{\lambda^2}\partial_1F_{\mas}(\xi, u^n_{i\!-\!1}))\phi^{n\!-\!1} + \lstab\phi_{j\!-\!1}^n , [\phi_{i, j}^n]_{\men}\rangle_Y,
	\end{aligned}
	\end{equation}
	and
	\begin{equation}\label{reml2}
	\begin{aligned}
	& \frac{\dt}{\lambda^2} \langle F_{\men}(\phi^n_{i, j\!-\!1}\!, u^n_{i\!-\!1})+\llin\phi^n_{i, j\!-\!1}, [\phi_{i, j}^n]_{\men} \rangle_Y \\
	& \qquad = \frac{\dt}{\lambda^2} \langle \left( \partial_1F_{\men}(\eta, u^n_{i\!-\!1})+\llin\right) \phi^n_{i, j\!-\!1}, [\phi_{i, j}^n]_{\men} \rangle_Y,
	\end{aligned}
	\end{equation}
	where $\xi : Y \to \r$ and $\eta:Y \to \r$ are two functions such that $\xi(\ynn) \in (0, \phi^{n\!-\!1}(\ynn))$ and $\eta(\ynn) \in (0, \phi^{n}_{i, j-1}(\ynn))$ for all $\ynn \in Y$. Knowing that $\lstab$, $\partial_1 F_{\mas} \geq 0$ and $ \llin \geq \mathfrak{M}_{F_1}$, we get that the right-hand sides of \eqref{reml1} and \eqref{reml2} are negative. Consequently, 
	\begin{equation*}
	\left( 1+ \lstab + \frac{\dt}{\lambda^2} \llin\right) \| [\phi_{i, j}^n]_{\men}\|^2_Y + \dt \gamma \| \nabla [\phi_{i, j}^n]_{\men}\|^2_Y \leq 0, 
	\end{equation*}
	which implies $\left( 1+\lstab + \frac{\dt}{\lambda^2} \llin\right) \| [\phi_{i, j}^n]_{\men}\|^2_Y = 0$. In conclusion $[\phi_{i, j}^n]_{\men} = 0$ a.e., and with this we obtain the lower bound of $\phi_{i, j}^n$. 
	
	The upper bound follows by testing \eqref{eq:phasef_L} with $[\phi_{i, j}^n\!-\!1]_{\mas}$ and following the same steps. We obtain $\phi^n_{i, j}(\ynn)\leq 1$ a.e.
\end{proof}

As mentioned before, solving the non-linear problem accurately is crucial to guarantee the convergence of the multi-scale iterative scheme. The following theorem ensures the convergence of the micro-scale non-linear iterations under mild restrictions on $\dt$, $\llin$ and $\lstab$.

\begin{theorem}[Convergence of the non-linear solver]
	\label{te:teoremamicro}
	Let $\mathfrak{M}_{F_1}$ be as above, with $\lstab\geq0$ and $\llin\geq \mathfrak{M}_{F_1}$. If $\dt \leq \frac{\lambda^2(1+\lstab)}{\mathfrak{M}_{F_1}}$ then the L-scheme \eqref{eq:phasef_L} is convergent.
\end{theorem}
The proof of \Cref{te:teoremamicro} follows the same steps as the proof in \cite[Lemma~4.1]{kumar2014convergence}. We omit the details here.

\section{Analysis of the multi-scale iterative scheme}
\label{sec:Analysis}

In this section we show the convergence of the multi-scale iterative scheme in steps \ref{step1} - \ref{step3}. We verify a relation between the effective diffusivity and the porosity and prove the convergence of the multi-scale iterative scheme in steps \ref{step1} - \ref{step3}. The main difficulty in the convergence proof is due to the multi-scale characteristics of the scheme and the presence of the non-linear terms. We consider a simplified setting in which the flow component is disregarded.

\begin{proposition}\label{prop}
	For $n>0$ and the multi-scale iteration index $i>0$, the effective diffusion tensors $\aef^n$ and $\aef_i^n$ are symmetric, continuous and positive definite. In other words, the constants $a_m, a_M>0$ exist such that for all $\psi \in \r^\dim$ and $\xnn \in \Omega$
	\begin{equation*}
	a_m \|\psi\|^2 \leq \psi^\mathsf{T} \, \aef^n(\xnn)\, \psi \leq a_M \|\psi\|^2, \quad \text{and} \quad a_m \|\psi\|^2 \leq \psi^\mathsf{T} \, \aef^n_i(\xnn)\, \psi \leq a_M \|\psi\|^2.
	\end{equation*}
\end{proposition}
We refer to \cite[Proposition~6.12]{cioranescu1999introduction} for the proof of the symmetry and positive definiteness of the effective diffusion tensor.

\reqnomode
For $n>0$, let $u^{n\!-\!1}\in L^2(\Omega)$ and $\overline{\phi}^{n} ,\overline{\phi}^{n\!-\!1} \in L^\infty(\Omega)$ be given. In the absence of flow, the weak solution of the problem \eqref{eq:macro2} is defined as follows
\begin{definition}\label{def_macro}
	A weak solution to the problem \eqref{eq:macro2} is a function $u^{n}\in H_0^1(\Omega)$ satisfying
	\begin{equation}\label{eq:macro_w}
	\left\langle \overline{\phi}^{n} (u^{n} - u^{\star}) , v \right\rangle_{\Omega} + \dt D \left\langle \aef^n\nabla u^n, \nabla v \right\rangle_{\Omega} = \left\langle \overline{\phi}^{n\!-\!1} (u^{n\!-\!1}- u^{\star}), v \right\rangle_{\Omega}, 
	\end{equation}
	for all $v\in H_0^1(\Omega)$.
\end{definition}
We let $i\in\n$ denote the multi-scale iteration index. The iterated porosity\\  $\overline{\phi}^{n}_i(\xnn) := \int_Y \phi^n_i(\xnn, \ynn) \, d\ynn$ is given for all $\xnn \in \Omega$ and the diffusivity tensor $\aef_i^n$ depends on $\phi^n_i$ as explained in \eqref{eq:eff}. In the absence of flow, the weak solution of the problem \eqref{eq:macro2_i} is defined as follows
\begin{definition}\label{def_macroit}
	A weak solution to the problem \eqref{eq:macro2_i} is a function $u^{n}_i\in H_0^1(\Omega)$ satisfying
	\begin{equation}\label{eq:macro_wit}
	\left\langle \overline{\phi}^{n}_i (u^{n}_i -u^{\star}), v \right\rangle_{\Omega} + \dt D \left\langle \aef_i^n\nabla u^n_i, \nabla v \right\rangle_{\Omega} = \left\langle \overline{\phi}^{n\!-\!1} (u^{n\!-\!1} - u^{\star}), v \right\rangle_{\Omega}, 
	\end{equation}
	for all $v\in H_0^1(\Omega)$.
\end{definition}

For a fixed time $n>0$ and in order to prove the convergence of the multi-scale iterative scheme, we make three extra assumptions and show an important property of the effective diffusion tensor.
\begin{assumption}\label{as:pro}
	The porosity $\bphi^n$ is bounded away from $0$ and $1$. That is, there exists two constants $\bphi_m$ and $\bphi_M$ such that $0< \bphi_m \leq \bphi^n(\xnn) \leq \bphi_M<1$ a.e.
\end{assumption}
\begin{assumption}\label{as:cons}
	The concentration is such that $\|\nabla u^n\|_{L^\infty(\Omega)}\leq C_u$ for some constant $C_u>0$.
\end{assumption}
\begin{assumption}\label{as:consw}
	For every time step, multi-scale iteration and macro-scale location, the solution of the micro-scale cell problems \eqref{eq:microA} is such that \\
	$\|\nabla \omega^\scomp\|_{L^\infty(\Omega)}\leq C_w$ for some constant $C_w>0$ and for all $\scomp\in\{1, \dots, \dim\}$.
\end{assumption}
Assuming the essential boundedness of the gradients of $u^n$, respectively $\omega^\scomp$ is justified under certain conditions. For example, since $u^n_{i-1}$ is constant in $Y$, the solutions to the micro-scale elliptic problems are bounded uniformly w.r.t. $i$ in $H^1(Y)$, and have a better regularity than $H^1$. Assuming that $\nabla \phi^{n-1}$ is essentially bounded, one obtains bounds for $\nabla \phi^n_{i}$ by using \eqref{eq:microphase_i}. Furthermore, with a fixed $\delta > 0$ and recalling the essential bounds proved in Lemma \ref{lem:maxphase}, the problem \eqref{eq:microA} solved by $\omega^\scomp$ is linear, elliptic, and the coercivity constant is uniformly bounded. In view of the regularity and boundedness of $\phi^n_i$, one obtains that $\nabla \omega^\scomp$ is essentially bounded as well. Finally, for the macro-scale problem \eqref{eq:macro2_i}, assuming the the domain $\Omega$ and the initial data are sufficiently smooth, the essential boundedness of the gradient of $u^n$ can be obtained e.g. as in \cite[Chapter~3.15]{Ladyzhenskaya}.

For proving the convergence of the iterative scheme we start by showing that the changes in the phase field are bounding the variations in the diffusion tensor. We refer to \cite{schulz2019beyond, bringedal2017effective, ray2018old} for numerical studies revealing the relation between diffusivity (and permeability) and porosity.

\begin{lemma}\label{lema:diff}
	Let $n>0$ and $i>0$ be fixed. There exists a constant $C_A>0$ such that 
	\begin{equation}\label{eq:rel}
	\|\aef^n_i - \aef^n\|_{\Omega} \leq C_A\|\phi^n_i -\phi^n\|_{\Omega\times Y}.
	\end{equation}
\end{lemma}
\begin{proof}
	For each $\xnn \in \Omega$ we denote $\omega_{i,n}^\scomp$ and $\omega_{n}^\scomp$ the $\scomp$-component of the solution of the micro-scale cell problems \eqref{eq:microA} that corresponds to $\phi_i^n$ and $\phi^n$. By subtracting those two cell problems we get formally that 
	\begin{equation*}
	\nabla \cdot ((\phi_i^n + \delta)(\nabla (\omega_{i,n}^\scomp - \omega_{n}^\scomp))) = - \nabla \cdot ((\phi_i^n-\phi^n)( \enn_\scomp + \nabla \omega_{n}^\scomp )).
	\end{equation*}
	From this one immediately obtains that 
	\begin{equation}\label{eq:var1}
	| \left\langle (\phi_i^n + \delta)\nabla (\omega_{i,n}^\scomp - \omega_{n}^\scomp), \nabla \psi \right\rangle_Y | = |\left\langle (\phi^n-\phi_i^n)( \enn_\scomp + \nabla \omega_{n}^\scomp ), \nabla\psi \right\rangle_Y|
	\end{equation}
	for all $\psi \in H_\#^1(Y)$. Since $|Y| =1$ and $0\leq \phi_i^n$, by taking $\psi = \omega_{i,n}^\scomp - \omega_{n}^\scomp$ in \eqref{eq:var1}, applying Cauchy-Schwartz and due to Assumption \ref{as:consw} we obtain
	\begin{equation}\label{eq:var2}
	\| \nabla (\omega_{i,n}^\scomp - \omega_{n}^\scomp) \|_{L^1(Y)} \leq \| \nabla (\omega_{i,n}^\scomp - \omega_{n}^\scomp) \|_{L^2(Y)} \leq \frac{1+C_w}{\delta} \| \phi_i^n-\phi^n\|_Y. 
	\end{equation}
	
	On the other hand, for each component $\rcomp\scomp$ of $\aef^n_i(\xnn) - \aef^n(\xnn)$ it is easy to show that 
	\begin{align*}
	|[\aef^n_i(\xnn)]_{\rcomp\scomp} - [\aef^n(\xnn)]_{\rcomp\scomp}| &\leq \int_Y |\phi_i^n-\phi^n| d\ynn + \int_Y |(\phi_i^n+\delta)\partial_\rcomp\omega^\scomp_{i,n}-(\phi^n+\delta)\partial_\rcomp\omega^\scomp_{n}| d\ynn \\
	& \leq \int_Y |\phi_i^n-\phi^n| d\ynn  \\ &\qquad + \int_Y |(\phi_i^n+\delta)\left( \partial_\rcomp\omega^\scomp_{i,n} -\partial_\rcomp\omega^\scomp_{n}\right)|+ |(\phi_i^n-\phi^n)\partial_\rcomp\omega^\scomp_{n}|d\ynn \\
	&\leq (1+C_w) \|\phi_i^n-\phi^n\|_Y + \int_Y |(\phi_i^n+\delta)(\partial_\rcomp\omega^\scomp_{i,n}-\partial_\rcomp\omega^\scomp_{n})| d\ynn.
	\end{align*}
	By using \eqref{eq:var2} and the equivalence of norms in $\r^{\dim\times\dim}$ one gets 
	\begin{equation*}
	C_{f}\|[\aef^n_i(\xnn)] - [\aef^n(\xnn)]\|_{2,\r^{\dim\times\dim}} \leq	\|[\aef^n_i(\xnn)] - [\aef^n(\xnn)]\|_{1,\r^{\dim\times\dim}} \leq \frac{\dim(1+C_w)(1+\delta)}{\delta} \|\phi_i^n-\phi^n\|_Y,
	\end{equation*}
	where $\| \cdot \|_{p,\r^{\dim\times\dim}}$ denotes the $L^p$ matrix norm induced by the $L^p$ vector norm for $p=1,2$. The constant $C_{f}>0$ is coming from the equivalence between the induced $L^1$ and $L^2$ matrix norms. By integrating over $\Omega$, we conclude that
	\begin{equation*}
	\|\aef^n_i - \aef^n\|_{\Omega} \leq  \frac{\dim(1+C_w)(1+\delta)}{C_{f}\delta} \|\phi^n_i -\phi^n\|_{\Omega\times Y},
	\end{equation*}
\end{proof}

\paragraph{The multi-scale error equations}
For a fixed $n>0$ and the iteration index $i>0$, we define $ e_i^{\phi}:= \phi_i^n -\phi^n$, $e_i^{u}:= u_i^n -u^n$ and $e_i^{\bphi}:= \bphi_i^n - \bphi^n $. Subtracting \eqref{eq:phasef_wit} from \eqref{eq:phasef_w} and \eqref{eq:macro_wit} from \eqref{eq:macro_w} shows that
\begin{equation}\label{eq:error_phasef}
\begin{aligned}
& \langle e_i^{\phi}, \psi \rangle_Y + \dt \gamma \langle \nabla e_i^{\phi}, \nabla \psi \rangle_Y + \frac{\dt}{\lambda^2} \lstab \langle( e_i^{\phi}-e^{\phi}_{i\!-\!1}), \psi \rangle_Y \\
& \qquad = \frac{\dt}{\lambda^2} \langle F_{\men}(\phi^n_i, u^n_{i\!-\!1}) - F_{\men}(\phi^n, u^n), \psi \rangle_Y  \\
& \qquad \qquad + \frac{\dt}{\lambda^2} \langle F_{\mas}(\phi^{n\!-\!1}, u^n_{i\!-\!1}) - F_{\mas}(\phi^{n\!-\!1}, u^n), \psi \rangle_Y , 
\end{aligned}
\end{equation}
\begin{equation}\label{eq:error_macro}
\left\langle \overline{\phi}_i^{n} e_i^{u} , v \right\rangle_{\Omega} + \dt D \left( \left\langle \aef_i^n \nabla u^n_i, \nabla v \right\rangle_{\Omega} - \left\langle \aef^n \nabla u^n, \nabla v \right\rangle_{\Omega}\right) = \left\langle (u^n - u^{\star}) e_i^{\bphi} , v \right\rangle_{\Omega}, 
\end{equation}
for all $\psi\in H_{\#}^1(Y)$ and $v\in H_{0}^1(\Omega)$. We use this to prove the convergence of the multi-scale scheme. 
\begin{theorem}[Convergence of the multi-scale scheme]
	\label{te:teorema}
	Let $n>0$ be fixed and $\bphi^{n\!-\!1}$ be given. Under the Assumptions \ref{as:pro} - \ref{as:consw} and \ref{as:a2} - \ref{as:a3}, with $\mathfrak{M}:= \max \left( \mathfrak{M}_{F_1}, \mathfrak{M}_{F_2}\right)$ and $\lstab > 6\mathfrak{M}$, if the time step is small enough (i.e. satisfying \eqref{eq:cond_dt} below), the multi-scale iterative scheme in steps \ref{step1} - \ref{step3} is convergent.
\end{theorem}

\begin{proof}
	For a fixed macro-scale point $\xnn\in \Omega$ and the iteration index $i>0$, we consider the error equation \eqref{eq:error_phasef} and take the test function $\psi=e^\phi_i$.
	By the mean value theorem and \ref{as:a2} - \ref{as:a3}, one gets
	\begin{equation*}\label{eq:1theo}
	\begin{aligned}
	& \| e_i^\phi\|^2_Y + \dt \gamma \| \nabla e_i^\phi\|_Y^2 + \lstab\frac{\dt}{\lambda^2} \| e_i^\phi\|^2_Y \leq \lstab \frac{\dt}{\lambda^2} \langle e_{i\!-\!1}^\phi, e_i^\phi\rangle_Y \\
	& \qquad +\frac{\dt}{\lambda^2} \langle 2\mathfrak{M} e_{i\!-\!1}^u, e_i^\phi\rangle_Y + \frac{\dt}{\lambda^2} \langle \mathfrak{M} e_{i}^\phi, e_i^\phi\rangle_Y.
	\end{aligned}
	\end{equation*}
	Using Young's inequality on the first two terms on the right hand side, with $\delta_1, \delta_2 >0$ one obtains
	\begin{equation*}
	\begin{aligned}
	& \left( 1 + \frac{\dt}{\lambda^2} \left( \lstab-\mathfrak{M}\right) \right) \| e_i^\phi\|^2_Y + \dt \gamma \| \nabla e_i^\phi\|_Y^2  \\
	& \qquad \leq \lstab\frac{\dt}{\lambda^2 } \frac{\delta_1}{2} \| e_{i\!-\!1}^\phi\|^2_Y + \lstab\frac{\dt}{\lambda^2 } \frac{1}{2\delta_1} \| e_{i}^\phi\|^2_Y + \mathfrak{M}\frac{\dt \delta_2}{\lambda^2 } | e_{i\!-\!1}^u|_Y^2 + \mathfrak{M}\frac{\dt}{\lambda^2} \frac{1}{\delta_2} \| e_{i}^\phi\|^2_Y.
	\end{aligned}
	\end{equation*}
	By taking $\delta_1=1$ and $\delta_2=\frac{1}{2}$, we get
	\begin{equation*}
	\begin{aligned}
	& \left( 1 + \frac{\dt}{\lambda^2}\left( \frac{\lstab}{2}-3\mathfrak{M}\right) \right) \| e_i^\phi\|_Y^2 \leq \lstab\frac{\dt}{2\lambda^2 } \| e_{i\!-\!1}^\phi\|_Y^2 + \mathfrak{M}\frac{\dt}{2\lambda^2} | e_{i\!-\!1}^u|_Y^2.
	\end{aligned}
	\end{equation*}
	Integrating over the macro-scale domain $\Omega$, we conclude that
	\begin{equation}\label{eq:proof_micro}
	\left( 1 + \frac{\dt}{\lambda^2}\left( \frac{\lstab}{2}-3\mathfrak{M}\right) \right) \| e_i^\phi\|_{\Omega\times Y}^2 \leq \lstab\frac{\dt}{2\lambda^2 } \| e_{i\!-\!1}^\phi\|_{\Omega\times Y}^2 + \mathfrak{M}\frac{\dt}{2\lambda^2 } \| e_{i\!-\!1}^u\|_{\Omega}^2.
	\end{equation}
	On the other hand, taking the test function $v= e_i^u$ on the macro-scale error equation \eqref{eq:error_macro} and using the Assumption \ref{as:pro} and the Proposition \ref{prop}, we have
	\begin{equation*}
	\bphi_m \| e_i^u\|^2_\Omega + \dt D a_m \| \nabla e_i^u \|^2_\Omega \leq \dt D \langle (\aef_i^n-\aef^n) \nabla u^n, \nabla e_i^u \rangle_\Omega + \langle (u^\star - u^n) e_i^{\bphi}, e_i^u\rangle_\Omega.
	\end{equation*}
	When using Young's inequality twice with $\delta_3, \delta_4 >0$, we obtain
	\begin{equation*}
	\begin{aligned}
	& \bphi_m \| e_i^u\|^2_\Omega + \dt D a_m \| \nabla e_i^u \|^2_\Omega \leq \dt D \left( \frac{\delta_3}{2} \|(\aef_i^n-\aef^n)\nabla u^n\|^2_{\Omega} + \frac{1}{2\delta_3}\|\nabla e_i^u\|^2_{\Omega}\right) \\
	& \qquad + \frac{\delta_4}{2}\|(u^\star - u^n) e_i^{\bphi}\|^2_{\Omega} + \frac{1}{2\delta_4} \| e_i^u \|^2_{\Omega}.
	\end{aligned}
	\end{equation*}
	We take $\delta_3= \frac{1}{a_m}$ and $\delta_4 = \frac{1}{\bphi_m}$ and due to the Assumption \ref{as:cons} and Lemma \ref{lema:diff} we obtain
	\begin{equation*}
	\frac{\bphi_m}{2} \| e_i^u\|^2_\Omega + \frac{\dt D a_m}{2} \| \nabla e_i^u \|^2_\Omega \leq \frac{\dt D}{2a_m} C_u^2 C_A^2 \| e_i^{\phi}\|^2_{\Omega\times Y} + \frac{1}{2\bphi_m} \bar{u}^{\star 2} \| e_i^{\bphi}\|^2_{\Omega},
	\end{equation*}
	where $\bar{u}^{\star}:= u^{\star}+C_pC_u$ with $C_p$ being a constant coming from the Poincar\'e inequality. Since $|Y| = 1$, one has $\| e_i^{\bphi}\|_\Omega \leq \| e_i^{\phi}\|_{\Omega\times Y}$, implying
	\begin{equation}\label{eq:proof_macro}
	\| e_i^u\|^2_\Omega \leq \left( \frac{\dt D}{a_m \bphi_m} C_u^2 C_A^2 + \frac{\bar{u}^{\star 2}}{\bphi_m^2} \right) \| e_i^{\phi}\|^2_{\Omega \times Y}.
	\end{equation}
	Observe that the constants in \eqref{eq:proof_macro} do not depend on the iteration index, i.e. \eqref{eq:proof_macro} can be written for the index $i-1$ as well. Using this in \eqref{eq:proof_micro} we obtain
	\begin{align}\label{eq:proof_end1}
	&\left( 1 + \frac{\dt}{\lambda^2}\left( \frac{\lstab}{2}-3\mathfrak{M}\right) \right) \| e_i^\phi\|_{\Omega\times Y}^2 \nonumber\\
	& \qquad \leq \left( \lstab\frac{\dt}{2\lambda^2 } + \mathfrak{M}\frac{\dt}{2\lambda^2 }\left( \frac{\dt D}{a_m \bphi_m} C_u^2 C_A^2 + \frac{\bar{u}^{\star 2}}{\bphi_m^2} \right)\right)  \| e_{i\!-\!1}^\phi\|_{\Omega\times Y}^2.
	\end{align}
	
	Clearly, \eqref{eq:proof_end1} can be rewritten as  $\|e^\phi_i\|^2 \leq C  \|e^\phi_{i-1}\|^2$. By taking the time step $\Delta t$ sufficiently small, one obtains $C < 1$, so the error is contractive. Specifically, if $\dt>0$ satisfies the inequality 
	\begin{equation}\label{eq:cond_dt}
	\left( \frac{\mathfrak{M}DC_u^2 C_A^2}{2\lambda^2a_m \bphi_m}\right) \dt^2 + \frac{\mathfrak{M}}{\lambda^2}\left( \frac{\bar{u}^{\star 2}}{2\bphi_m^2}+3\right) \dt  < 1,
	\end{equation}
	then \eqref{eq:proof_end1} is a contraction. By the Banach theorem we conclude that $\| e_i^\phi\|_{\Omega\times Y} \to 0$ as $i\to \infty$. This, together with \eqref{eq:proof_macro} implies that $\| e_i^u\|_{\Omega} \to 0$ as $i\to \infty$, which proves the convergence of the multi-scale iterative scheme. 
	
\end{proof}

\begin{remark}
	The inequality \eqref{eq:cond_dt} imposes a restriction in the time step $\dt$, and can clearly be fulfilled for some real $\dt>0$. This restriction guarantees the convergence of the iterative scheme and does not depend on the spatial discretization. Also note that the convergence is achieved for any starting point. Nevertheless, finding specific bounds for $\dt$ from \eqref{eq:cond_dt} is not obvious because it depends on unknown constants. In \Cref{sec:numeric} we choose $\dt$ based on numerical experiments inspired by \cite{storvik2019optimization}, where a coarse spatial discretization is used to estimate a suitable time step size.
\end{remark}

\section{The adaptive strategy}
\label{sec:adaptivity}
We design a multi-scale adaptive strategy to localize and reduce the error and to optimize the computational cost of the multi-scale simulations.

Let $\trian_{H}$ be a triangular partition of the macro-scale domain $\Omega$ with elements $T$ of diameter $H_{T}$ and $H := \max\limits_{T\in\trian_{H}}{H_{T}}$. We assign one micro-scale domain $Y$ to the barycentre (or integration point) of each macro-scale element $T$. At each micro-scale domain $Y$ we define another triangular partition $\trian_{h}$ with elements $T_\mu$ of diameter $h_{T_\mu}$ and $h := \max\limits_{T_\mu\in\trian_{h}}{h_{T_\mu}}$. In \Cref{fig:meshes0}, the structure and the notation of the meshes are shown.
\begin{figure}[htpb!]
	\centering
	\includegraphics[width=0.6\textwidth]{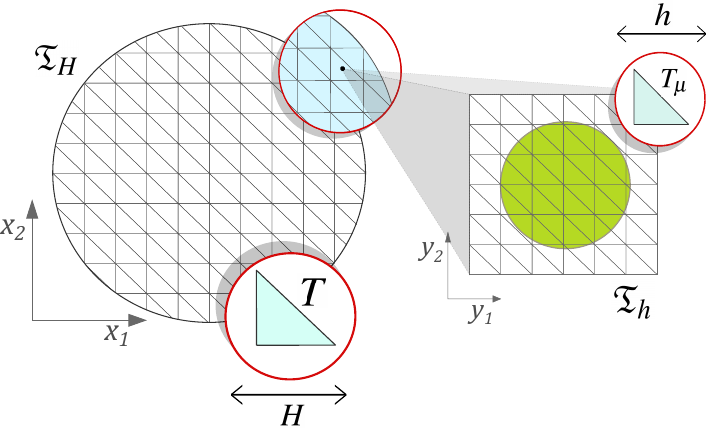}
	\caption{Sketch of the macro-scale and micro-scale meshes. For each $T\in \trian_H$ there is one corresponding micro-scale domain $Y$ with a micro-scale mesh $\trian_h$.}
	\label{fig:meshes0}
\end{figure}
We first present the mesh refinement strategy used in the micro scale and thereafter we turn to the macro-scale adaptive strategy used to optimize the computations.

\subsection{The micro-scale mesh adaptivity}

The accuracy in the solution of the phase field is influenced by the mesh size of the micro-scale discretization. It is necessary to create a fine mesh such that $h\ll \lambda$ to capture the diffuse transition zone. Nevertheless, such a fine uniform mesh would make the computation of the phase field and the effective parameters very expensive. Here we propose an adaptive micro-scale mesh with fine elements only in the diffuse transition zone of the phase field.

The mesh refinement strategy relies on an estimation of the evolution of the phase field. Here we use the fact that $\phi$ is essentially bounded by $0$ and $1$ a.e. and that the large changes in the gradient of $\phi$ are encountered in the transition zone. Nevertheless, other methods or refinement criteria can be used without modifying the whole strategy.

Here the local mesh adaptivity is divided into three main steps: prediction - projection - correction. This strategy is an extension of the predictor-corrector algorithm proposed in \cite{heister2015primal} and by construction, our strategy avoids nonconforming meshes.

For a fixed time $n>0$, consider a micro-scale domain $Y$ and let $\phi^{n\!-\!1}$ be given over a mesh $\trian_h^{n\!-\!1}$. The mesh $\trian_h^{n\!-\!1}$ is "optimal" in the sense that it is fine only in the diffuse transition zone of $\phi^{n\!-\!1}$. Take also an auxiliary coarse mesh $\trian_{c}$, which is uniform with mesh size $h_{\text{max}}\gg \lambda$. 
\begin{description}
	\item[Prediction.] Given the mesh $\trian_h^{n\!-\!1}$ compute a first approximation to the solution of problem $(\mathrm{P}^{\mu, 1}_{\phi})$. We call this approximation the auxiliary solution $\phi^{n\ast}_1$. 
	Project the solution $\phi^{n\ast}_1$ on the coarse mesh $\trian_{c}$. The elements marked to be refined are $T_\mu \in \trian_{c}$ such that \[ \theta_r\lambda \leq \phi^{n\ast}_1|_{T_\mu} \leq 1-\theta_r\lambda   \]
	for some constant $0<\theta_{r}<\frac{1}{2\lambda}$. After marking the triangles, we refine the mesh in the selected zone. The refinement process is repeated until the smallest element is such that $h_{T_\mu} \leq h_{min} \ll \lambda$. The result is a refined mesh $\trian_h^{n\ast}$ that is fine enough at the predicted transition zone of the phase field $\phi^{n\ast}_1$.
	
	\item[Projection] Create a projection mesh $\trian_{r}$ that is the union of the previous mesh and the predicted mesh. The mesh $\trian_{r}= \trian_h^{n\!-\!1}\cup\trian_h^{n\ast}$ is fine enough at the transition zone of $\phi^{n\!-\!1}$ and $\phi^{n\ast}_1$. To properly describe the interface of both $\phi^{n\!-\!1}$ and $\phi^{n\ast}_1$ we project the previous solution $\phi^{n\!-\!1}$ over $\trian_{r}$.
	
	\item[Correction] Given the mesh $\trian_{r}$ and the projection of $\phi^{n\!-\!1}$ compute once more the solution of problem $(\mathrm{P}^{\mu, 1}_{\phi})$. The projection of this result over the mesh $\trian_h^{n}$ corresponds to the solution $\phi^{n}_1$. 
\end{description}
This process is necessary at every time step and every micro-scale domain but we perform the mesh refinement only in the first iteration of the coupled scheme. However, this procedure could be extended for further iterations.
Notice that higher values of the parameter $\theta_r$ lead to coarser meshes and less error control. We will illustrate the role of $\theta_r$ in \Cref{sec:numeric}. 

In \Cref{fig:sketchmicro} we sketch the prediction-projection-correction strategy by zooming in on the transition zone of a phase field. There the mineral is shrinking from the time $n\!-\!1$ to $n$. In \Cref{fig:sketchmicro} (a) and (d) we mark the center of the transition zone of the auxiliary solution $\phi^{n\ast}_1$ and the corrected solution $\phi^{n}_1$, and we see how the mesh follows the transition zone of the phase field.

\begin{figure}[htpb!]
	\centering
	\subfloat[ ]{\includegraphics[width=0.2\textwidth]{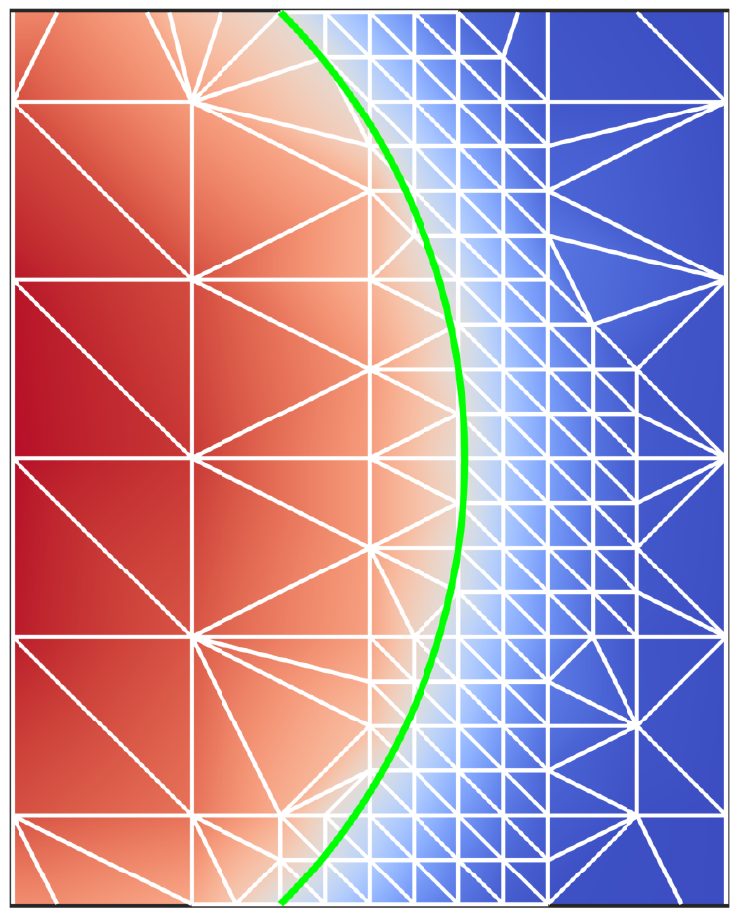}}\hspace{0.2cm}
	\subfloat[ ]{\includegraphics[width=0.2\textwidth]{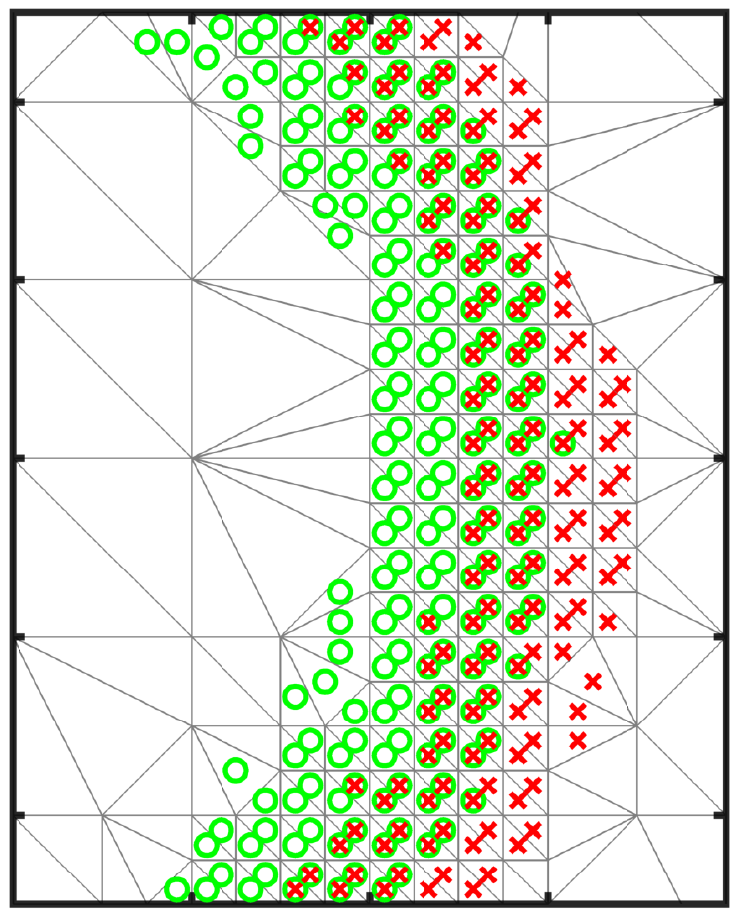}}\hspace{0.2cm}
	\subfloat[ ]{\includegraphics[width=0.2\textwidth]{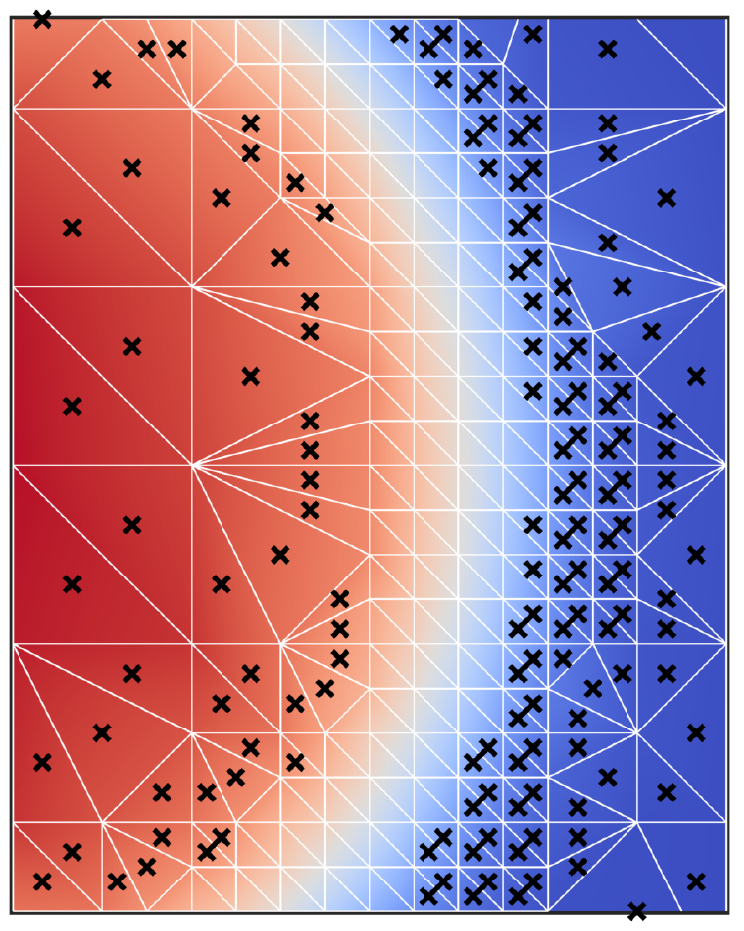}}\hspace{0.2cm}
	\subfloat[ ]{\includegraphics[width=0.2\textwidth]{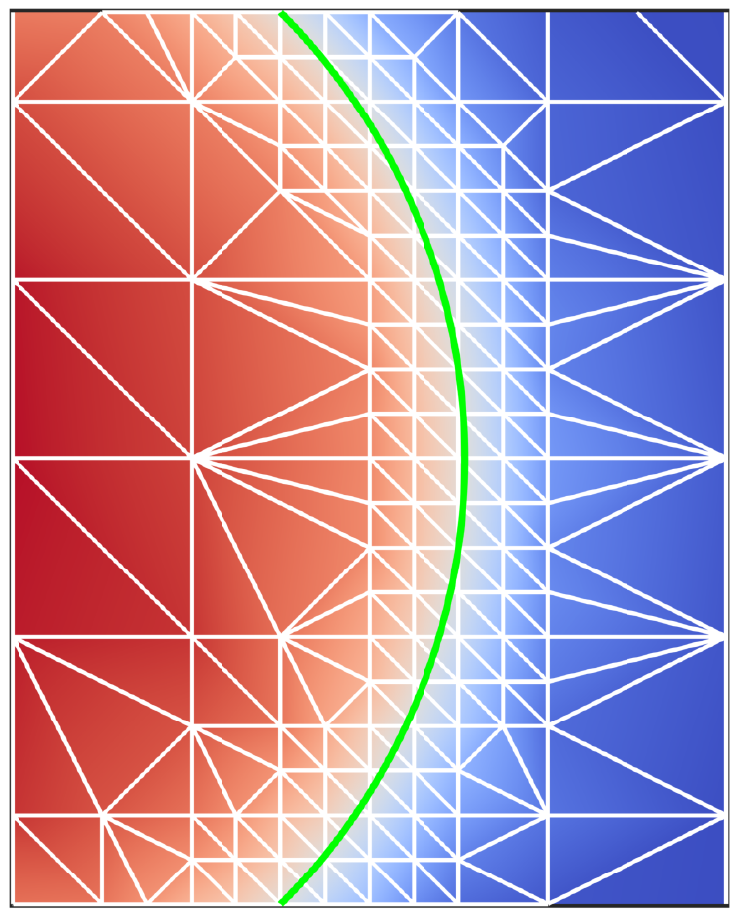}}\hspace{0.2cm}
	\vspace{-0.35cm}
	\captionof{figure}{Prediction- projection - correction strategy. (a) The auxiliary solution $\phi^{n\ast}_1$ over the mesh $\trian_h^{n\!-\!1}$ and the (green) line marks where $\phi^{n\ast}_1=0.5$ indicating the center location of the predicted transition zone. (b) The auxiliary mesh $\trian_r$ and the triangles that belong to the transition zone of $\phi^{n\!-\!1}$ (\textcolor{red}{$\boldsymbol{\times}$}) and $\phi^{n\ast}_1$ (\textcolor{green}{$\boldsymbol{\circ}$}). (c) The solution of problem $(\mathrm{P}^{\mu, 1}_{\phi})$ over $\trian_{r}$ and the elements outside of the transition zone ({$\boldsymbol{\times}$}). (d) The solution $\phi^{n}_1$ over the optimal mesh $\trian_h^{n}$ and the (green) line marks where $\phi^{n}_1=0.5$ indicating the center location of the transition zone.}
	\label{fig:sketchmicro}
\end{figure}

\subsection{The macro-scale adaptivity}

The computations on the micro scale can be optimized by the mesh adaptivity discussed before and the micro-scale cell problems can be computed in parallel. Nevertheless, it is demanding to compute the micro-scale quantities at every element (or integration point) of the macro-scale mesh. Here, the scale separation allows us to solve the model adaptively in the sense of the strategy introduced in \cite{redeker2013fast} and further studied in \cite{redeker2016upscaling}. There the macro-scale adaptivity uses only the solute concentration to locate where the micro-scale features need to be recalculated. Here
we implement a modified adaptive strategy on the micro scale that depends on the solute concentration and the phase-field evolution. With this, we extend the method in \cite{redeker2013fast} to more general settings, including heterogeneous macro-scale domains. 

To be more precise, we define the metric $d_E$ such that it measures the distance of two macro-scale points $\xnn_1, \xnn_2\in \Omega$ in terms of the solute concentration and the phase-field evolution, i.e.
\begin{equation*}
d_E(\xnn_1, \xnn_2;t;\varLambda) := \int_0^t e^{-\varLambda(t-s)}\left( d_u(\xnn_1, \xnn_2;s) +\int_Y d_{\phi}(\xnn_1, \xnn_2, \ynn;s) d\ynn \right) ds.
\end{equation*}
Here $d_u$ and $d_{\phi}$ are defined as follows \[d_u(\xnn_1, \xnn_2;s) := |u(\xnn_1, s)-u(\xnn_2, s)| \text{ and } d_{\phi}(\xnn_1, \xnn_2, \ynn;s) := |\phi(\xnn_1, \ynn, s)-\phi(\xnn_2, \ynn, s)|,\] and $\varLambda \geq 0$ is a history parameter. In the discrete setting we calculate the distance $d_E$ recursively, i.e. 
\begin{align*}
&d_E(\xnn_1, \xnn_2;n\dt;\varLambda) \approx e^{-\varLambda \dt}d_E(\xnn_1, \xnn_2;(n\!-\!1)\dt;\varLambda) \\
&\qquad + \dt \left( d_u(\xnn_1, \xnn_2;n\dt) +\int_Y d_{\phi}(\xnn_1, \xnn_2, \ynn;n\dt) d\ynn \right).
\end{align*}
The spatial integrals are also calculated numerically depending on the spatial discretization. 

At each time $n\geq0$ we divide the set of macro-scale points (elements) into a set of \textit{active} points ($N_A(n)$) and a set of \textit{inactive} points ($N_I(n)$). Specifically, $N_{\text{Total}} = N_A(n) \cup N_I(n)$ and $N_A(n) \cap N_I(n) = \emptyset$ for all $n\geq0$. 

The micro-scale cell problems will only be solved for points that are active. In this way, the effective parameters and the porosity are updated only in such points. For the inactive point, the effective parameters and the porosity are updated by using the \textit{Copy method} described in \cite{redeker2016upscaling} and explained below. 

Let $0\leq C_r$, $C_c<1$ be given and define the refinement and coarsening tolerances as follows
\begin{equation*}
\textit{tol}_r(t) := C_r \cdot \max\limits_{\xnn_1, \xnn_2\in \Omega} \left\lbrace d_E(\xnn_1, \xnn_2;t;\varLambda) \right\rbrace \quad \text{and} \quad \textit{tol}_c(t) := C_c\cdot\textit{tol}_r(t).
\end{equation*}

For $n>0$ and on the first multi-scale iteration, i.e. before the iterative process, the solutions $u^{n-1}(\xnn)$ and $\phi^{n-1}(\xnn,\ynn)$ for all $\xnn\in \Omega$ and $\ynn \in Y$ are given. The adaptive process consists of the following steps
\begin{itemize}	
	
	\item Initially, for $n=0$ all the points are inactive, i.e. $N_A(0) = \emptyset$ and $N_I(0) = N_{\text{Total}}$.
	\item Update the set of active points $N_A(n)$ and $N_I(n)$.
	\begin{itemize}
		\item Set $N_A(n) =N_A(n\!-\!1)$ and $N_I(n) =N_I(n\!-\!1)$. For each active point $\xnn_A \in N_A(n)$ repeat the following: if there exists another active node $\xnn_B\in N_A(n)$ such that $d_E(\xnn_A, \xnn_B;(n\!-\!1)\dt;\varLambda)<\textit{tol}_c$, then $\xnn_A$ is deactivated, i.e. $\xnn_A \in N_I(n)$. Otherwise, $\xnn_A \in N_A(n)$.
		
		\item For each inactive point $\xnn_I \in N_I(n)$ repeat the following: if $N_A(n) = \emptyset$  the point $\xnn_I$ is activated. Otherwise, calculate the distance to all the active nodes.
		If $\min\limits_{\xnn_A\in N_A(n)} \left\lbrace d_E(\xnn_I, \xnn_A;(n\!-\!1)\dt;\varLambda) \right\rbrace > \textit{tol}_r$ then the point $\xnn_I$ is activated, i.e., $\xnn_I \in N_A(n)$. 
	\end{itemize}
	
	\item Associate all the inactive points to the most similar active point. In other words, an inactive point $\xnn_I \in N_I(n)$ is associated with $\xnn_A \in N_A(n)$ if \\ $\xnn_A =  \underset{\xnn\in N_A(n)}{\mathrm{argmin}}\left\lbrace d_E(\xnn_I, \xnn;(n\!-\!1)\dt;\varLambda)\right\rbrace $.	
\end{itemize}

After updating the sets of active and inactive points we use the multi-scale iterations to solve the micro- and macro-scale problems. At each multi-scale iteration ($i>0$) we solve \eqref{eq:microphase_i}, \eqref{eq:microA} (and \eqref{eq:microK}) and transfer the solutions $\phi^{n}_i$, $\aef^{n}_i$ (and $\kef^{n}_i$) from the active points to their associated inactive ones. We then solve the macro-scale problem  \eqref{eq:macro2_i} (and \eqref{eq:macro1_i}) and continue the multi-scale iterations until convergence.

The two tolerances $\textit{tol}_r$ and $\textit{tol}_c$ are controlled through the values of $C_r$ and $C_c$. For a fixed value of $C_r$ the role of $C_c$ is to control the upper bound for the distance between active points. In other words, higher values of $C_c$ imply that more active points in $N_A(n\!-\!1)$ remain active in $N_A(n)$. On the other hand, for a fixed value of $C_c$ the role of $C_r$ is to control the upper bound for the distance between active and inactive points. Namely, higher values of $C_r$ imply that less inactive points in $N_I(n)$ become active. In accordance with \cite{redeker2013fast} and to avoid a complete update of the set of active nodes, it is wise to use smaller values for $\textit{tol}_c$ than for $\textit{tol}_r$. Therefore, in \Cref{sec:ex1} we analyse the role of $C_r$ in the macro-scale error control when $C_c$ is fixed and is chosen to be small.

\subsection{The multi-scale adaptive algorithm}
We combine the multi-scale iterative scheme and the adaptive strategies in a simple algorithm, see \textbf{Algorithm 1}. Even though we showed the convergence of the multi-scale iterative scheme in a simplified setting disregarding the flow, we mention the solution of the effective permeability $\kef^n_i$ and the flow problem \eqref{eq:macro1_i} in \textbf{Algorithm 1}. The reason for this is that in the numerical tests, specifically in \Cref{sec:ex2}, we evidence that the iterative scheme also converges in the complete scenario.

\begin{algorithm}[H]
	\SetAlgoLined
	\KwResult{Concentration $u$, porosity $\bphi$ (and pressure $p$).}
	Given the initial conditions $u_I$ and $\phi_I$
	
	\For{time $t^n$}{
		Adjust the set $N_A(n)$ of the active macro-scale points
		
		Take $i=1$ and $u^{n}_0=u^{n\!-\!1}$
		
		\While{$\epsilon_M^{n, i}\geq \textit{tol}_M$}{
			\For{$\xnn\in N_A(n)$}{
				\If{i==1}{
					Adaptivity on the micro-scale meshes
				}
				Solve \eqref{eq:microphase_i} using the L-scheme until $\epsilon_\mu^{n, i, j} \leq \textit{tol}_\mu$
				
				Compute the effective matrix $\aef^n_i$ (and $ \kef^n_i$)
			}
			For $\xnn\in N_I(n)$ copy the solution from the nearest $\xnn\in N_A(n)$
			
			Solve the problem \eqref{eq:macro2_i} (and \eqref{eq:macro1_i})
			
			Next iteration $i=i+1$
		}
		Next time $n=n+1$
	}
	\caption{The multi-scale iterative scheme using adaptive strategies on both scales}
\end{algorithm}

\section{The numerical results}
\label{sec:numeric}
In this section, we present two numerical tests for the multi-scale iterative scheme. We restrict our implementations to the 2D case and all parameters specified in the following examples are non-dimensional according to the non-dimensionalization in \cite{bringedal2019phase}.

For the first test, in \Cref{sec:ex1} we use a simple setting where the performance of the multi-scale adaptive techniques are investigated. In \Cref{sec:ex2} we analyse an anisotropic and heterogeneous case where different shapes of the initial phase field are used. The numerical solutions of macro- and micro-scale problems \eqref{eq:macro1}, \eqref{eq:macro2}, \eqref{eq:microphase} and \eqref{eq:microA} are computed using the lowest order Raviart-Thomas elements (see \cite{bahriawati_three_2005}). For the micro-scale problems \eqref{eq:microK} we use the Crouzeix–Raviart elements (see \cite[Section~8.6.2]{boffi2013mixed}). The following (non-dimensional) constants have been used in all the simulations
\begin{equation}\label{eq:param}
D=1; \quad u^\star=1; \quad u_{\text{eq}}=0.5;\quad \gamma=0.01;\quad \lambda = 0.08; \quad \delta=1\text{E-4}.
\end{equation}

\subsection{Test case 1. Circular shaped phase field}
\label{sec:ex1}
Consider the macro-scale domain $\Omega = \left( 0, 1\right) \times \left( 0, \frac{1}{2}\right) $ and take $\ttime=0.25$. The system is initially in equilibrium, i.e. the initial concentration is $u(\xnn, 0)=u_{\text{eq}}$ and $p(\xnn, 0)=0$ for all $\xnn\in \Omega$. A dissolution process is triggered by having a fixed concentration $u=0$ in the lower-left corner of the domain $\Omega$. We take homogeneous Neumann boundary conditions everywhere else for both the solute concentration and pressure problems. At every micro-scale domain $Y$ the initial phase field $\phi_I$ has a circular shape with initial porosity $\bphi_0=0.5$. This configuration is displayed in \Cref{geoex1}. We allow the mineral to dissolve until a maximum porosity $\bphi_M = 0.9686$ is reached.

For the time discretization, even though \Cref{te:teorema}  gives a theoretical restriction on $\dt$, the estimation of an accurate bound is not evident. Here we choose $\dt$ experimentally by choosing an initial value of $\dt$ which is sufficiently small to ensure convergence of the micro-scale non-linear solver (see \Cref{te:teoremamicro}). If the multi-scale iterations converge in the first time step, this value of $\dt$ is used in the whole simulation. Otherwise, smaller values of $\dt$ are tested. Here the time step is chosen to be $\dt=0.01$, which was found to always ensure convergence in these tests.

\begin{figure}[htpb!]
	\centering
	\includegraphics[width=0.42\textwidth]{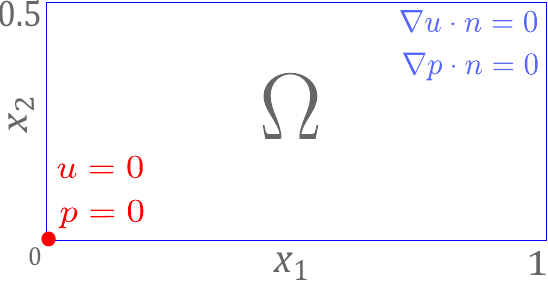}\hspace{1.5cm}
	\includegraphics[width=0.30\textwidth]{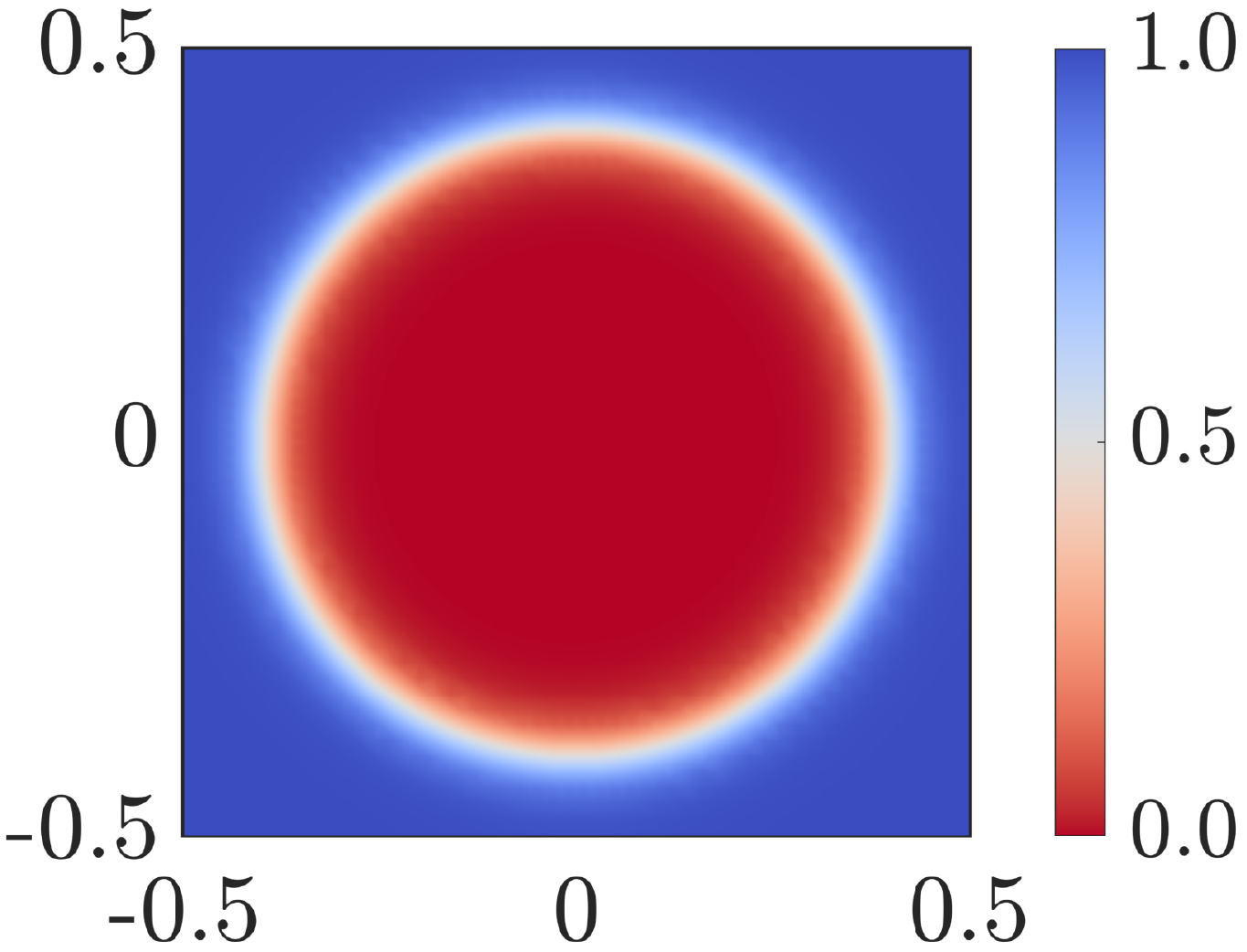}
	\captionof{figure}{The configuration of the macro scale (left) and phase-field initial condition (right) - Test case 1.}\label{geoex1}
\end{figure}

\subsubsection{The micro-scale non-linear solver and adaptivity}
To study the features of the micro-scale non-linear solver and the micro-scale refinement strategy, we look closer on the micro-scale domain $Y$ corresponding to the macro-scale location $\xnn=(0, 0)$ with an initial phase field as shown in \Cref{geoex1} and a constant concentration $u=0$.

Concerning the behavior of the micro-scale non-linear solver, we take dynamically the value of the linearization parameter $\llin= \max\left(|2\lambda f(u) + 8\gamma|,|2\lambda f(u) - 8\gamma|\right)$, which changes at every multi-scale iteration if the solute concentration $u$ changes. This choice of $\llin$ gives convergence of the micro-scale iterations, as shown in \cite{pop2004mixed}. We use this choice of $\llin$ in all the simulations below as well as the micro-scale stopping criterion $\textit{tol}_\mu =1$E$-8$. We choose $\textit{tol}_\mu$ so small to ensure sufficient accuracy of the micro-scale problems and to not influence the multi-scale convergence. 
For all the micro-scale meshes used in \Cref{table:micro-scalerefinement} the average number of micro-scale iterations is $13$. Here we do not iterate between scales and we choose $\lstab=0$ having no effect on the convergence of the non-linear solver. 

In \Cref{fig:microscale_Ex1} we show the phase field at time $t^n=0.10$. On each micro-scale domain $Y$ we use an initial uniform mesh with $200$ elements and apply three different values for the mesh refinement parameter, namely $\theta_r = 1, \, 2, $ and $5$.
\begin{figure}[htpb!]
	\centering
	\includegraphics[width=0.32\textwidth]{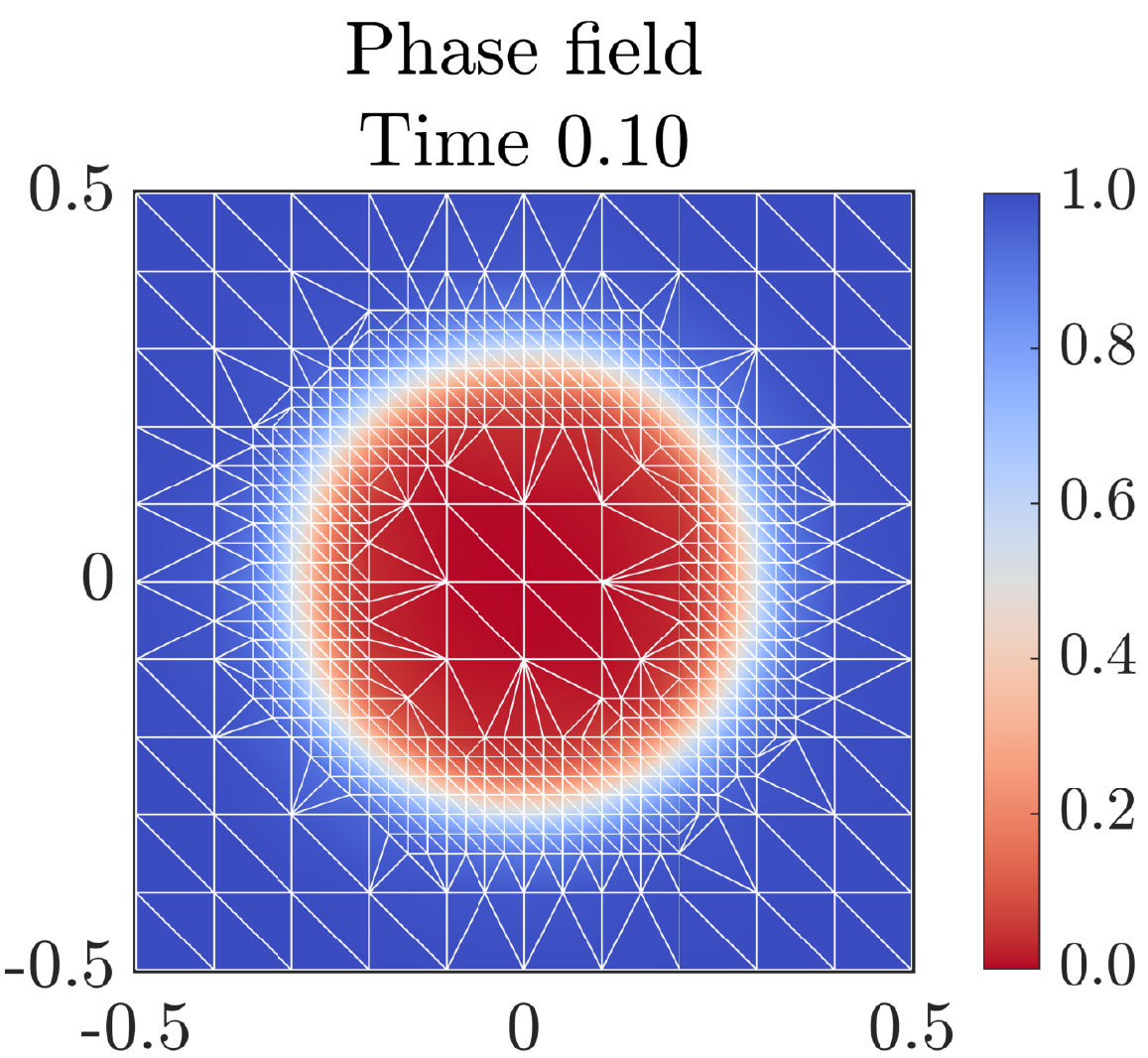}
	\includegraphics[width=0.32\textwidth]{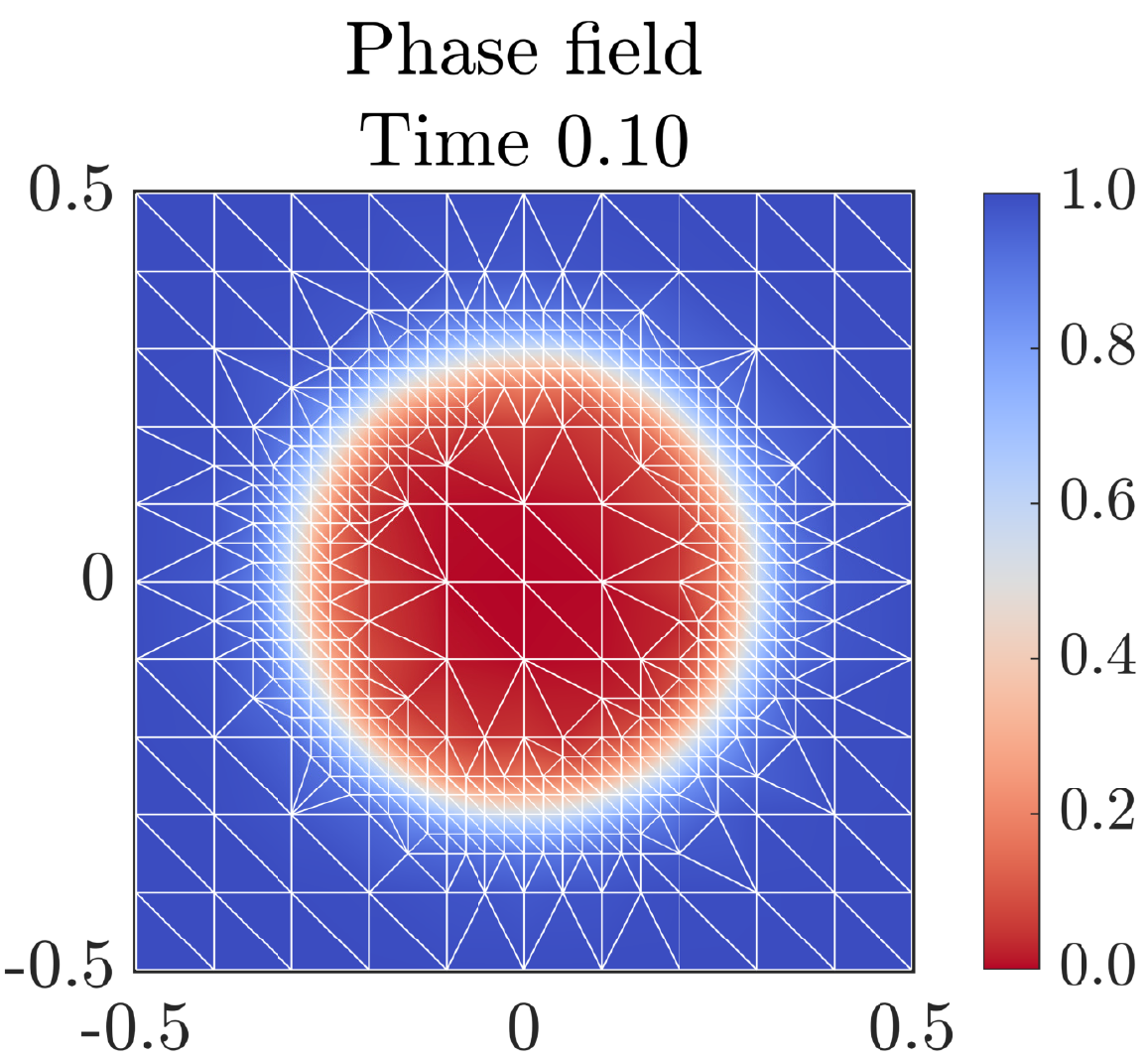}
	\includegraphics[width=0.32\textwidth]{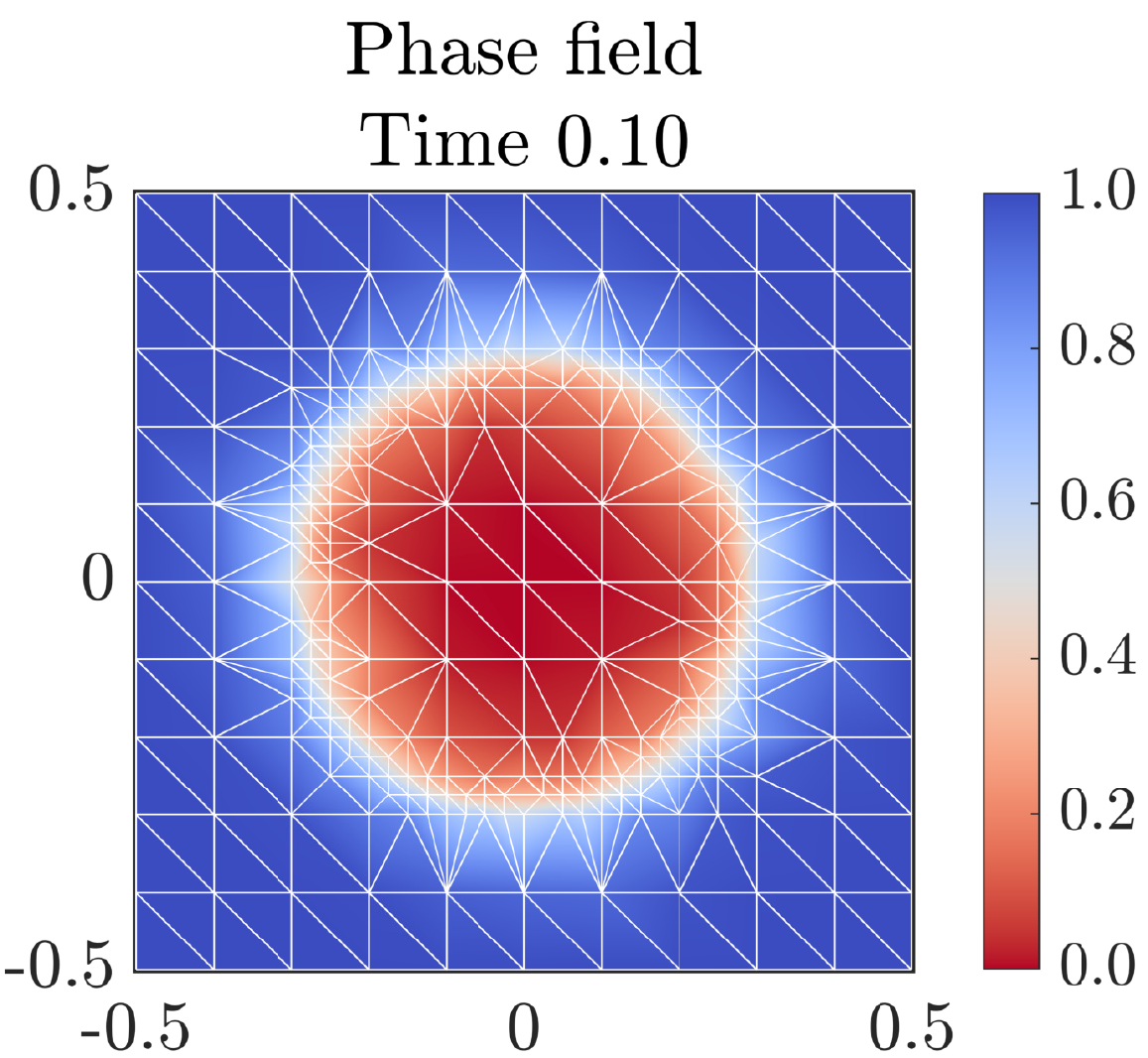}
	\captionof{figure}{The phase field $\phi^n(\xnn)$ corresponding to the macro-scale location $\xnn=(0, 0)$ at the time $t^n = 0.10$. Refinement parameters $\theta_r = 1, \, 2, $ and $5$ (left to right).}
	\label{fig:microscale_Ex1}
\end{figure}

It is clear that the micro-scale refinement parameter slightly changes the representation of the phase-field transition zone. This result is also evident in \Cref{table:micro-scalerefinement}. There we show a comparison between the micro-scale solutions when using different values of $\theta_r$ and the reference solution $\phi_{\text{ref}}$. We use a fixed uniform mesh with $7.20$E+$3$ elements and mesh size $h= 2.36$E-$2\ll \lambda$ to compute the reference solution $\phi_{\text{ref}}$. In \Cref{table:micro-scalerefinement} we report the average number of elements for each micro-scale mesh (\#Elements) and there the accuracy of the numerical solution is provided through the $L^2$-error, namely $E_{\phi} := \| \phi_{\text{ref}} - P_h(\phi)\|_{L^2([0, \ttime];L^2(Y))}$ with $P_h(\phi)$ being the projection of the solution $\phi$ over the reference mesh. 

All the meshes in \Cref{fig:microscale_Ex1} and \Cref{table:micro-scalerefinement} are constructed such that the minimum diameter in the mesh is $h_{T_\mu} \leq h_{min} = \frac{\lambda}{3}$. In \Cref{fig:microscale_Ex1}, the length of the smallest edge in the meshes is $\min\limits_{T_\mu\in\trian_{h}}h_{T_\mu} = 2.50$E-$2$ and the length of the largest edge (located far from the transition zone) is $h_{\text{max}} = 1.41$E-$1$.
\begin{table}[htpb!]
	\centering
	\renewcommand{\arraystretch}{1}
	\begin{tabular}{c|c|c|c|c}
		\hline
		$\theta_r$ & \#Elements & \%\#Elements & $E_{\phi}$ & \%$E_{\phi}$ \\
		\hline
		$0.5$ & $1.20$E+$3$ & $16.72$\% & $9.69$E-$3$ & $2.27$\%  \\
		$1$ & $1.04$E+$3$ & $14.51$\% & $1.01$E-$2$ & $2.37$\%  \\
		$2$ & $8.64$E+$2$ & $12.00$\% & $1.19$E-$2$ & $2.79$\%  \\
		$5$ & $5.60$E+$2$ & $7.77$\%  & $1.99$E-$2$ & $4.68$\%  \\
		\hline
	\end{tabular}
	\caption{The micro-scale adaptive results for a varying refining parameter $\theta_r$. The column \%\#Elements corresponds to the percentage of the original number of elements used in each mesh and \%$E_{\phi}$ is the relative error compared to the reference solution.}
	\label{table:micro-scalerefinement}
\end{table}

Smaller values of $\theta_{r}$ lead to better error control, but those values also imply more degrees of freedom and therefore increase the computational effort. In the following numerical experiments, we choose $\theta_{r}=2$ to control the error on the micro scale and, at the same time, limit the number of elements at each micro-scale domain.

\subsubsection{The multi-scale coupling and the macro-scale adaptivity}

We study the convergence of the multi-scale iterative scheme for different values of the parameter $\lstab$. In \Cref{te:teorema} the value of $\lstab$ is restricted to be $\lstab>6\mathfrak{M}$. Using the parameters in \eqref{eq:param} we obtain that $\mathfrak{M} \geq 1.12$. In \Cref{fig:Lstab} we compare the convergence of the multi-scale iterative scheme when using different values of $\lstab$. Specifically, in \Cref{fig:Lstab} we show the number of iterations used at the first time step for eleven different values of $\lstab$. It is evident that the conditions in \Cref{te:teorema} are rather restrictive and in practice, one can achieve convergence using smaller values of $\lstab\geq0$. For very small values of $\lstab$, the iterations needed in the multi-scale iterative scheme remain constant, which we highlight in \Cref{fig:Lstab}. Here we choose $\textit{tol}_M=1$E-$6$ for the multi-scale stopping criterion and we do not use the macro-scale adaptive strategy, i.e., we solve all the micro-scale problems. After this study, we choose $\lstab = 1$E$-4$ in all the simulations below.


\begin{figure}
	\centering
	\includegraphics[width=0.45\linewidth]{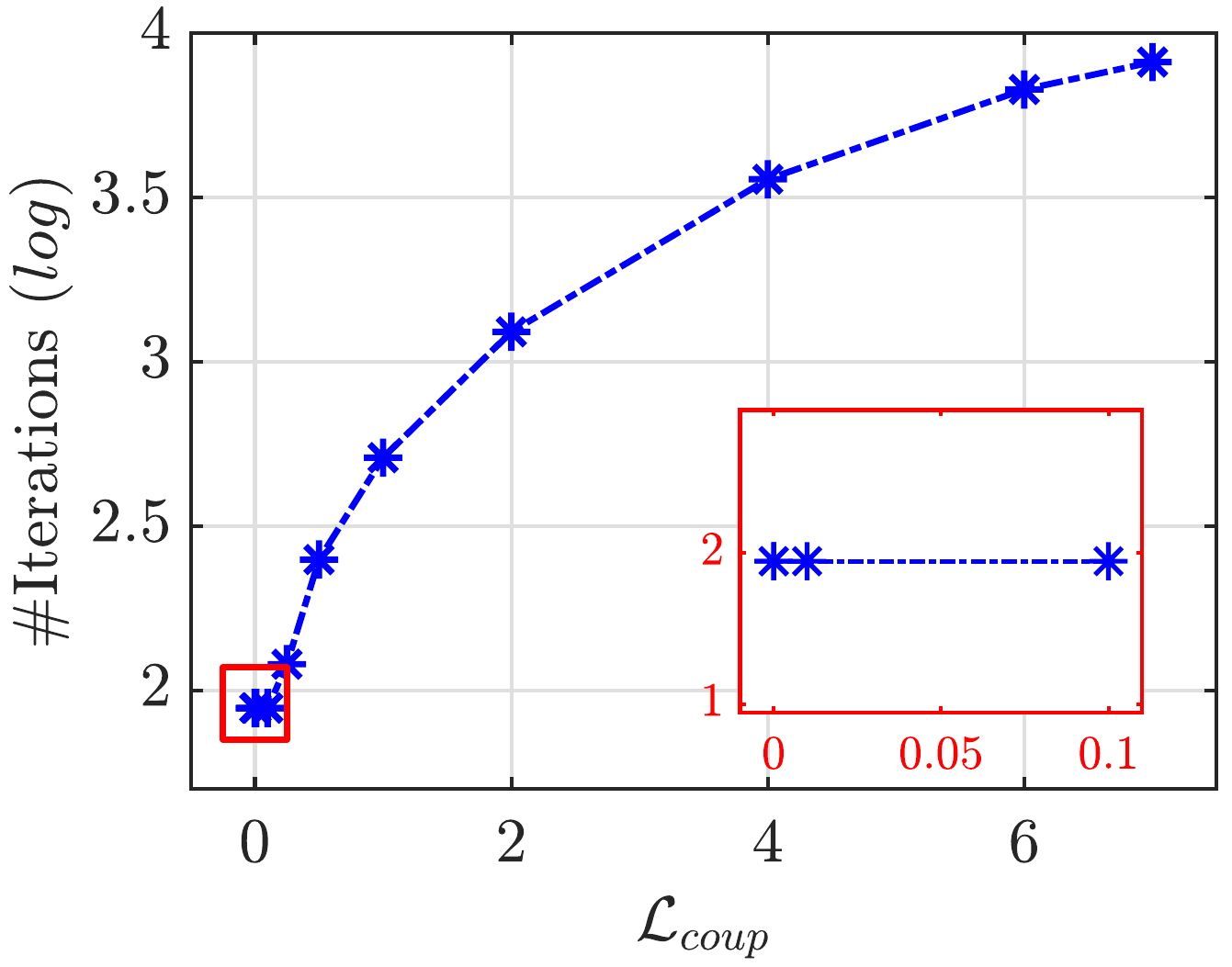}
	\caption{The number of multi-scale iterations ($log$) at time $t=0.01$ for different values of $\lstab$. Zoom in of the plot for small values of $\lstab$.}
	\label{fig:Lstab}
\end{figure}

In \Cref{fig:active_nodes} and \Cref{table:macro-scalerefinement} the results of the macro-scale adaptivity are shown. We choose the history parameter $\varLambda=0.1$ and the coarsening parameter $C_c=0.2$ based on the sensitivity analysis presented in \cite{redeker2013fast} and used in \cite{redeker2016upscaling}. \Cref{fig:active_nodes} illustrates the effect of the refinement parameter $C_r$ on the proportion of active nodes. There, the different intensities and sizes represent the percentage of the total number of times that each element was active during the whole simulation. 
\begin{figure}[htpb!]
	\centering
	\includegraphics[width=0.45\textwidth]{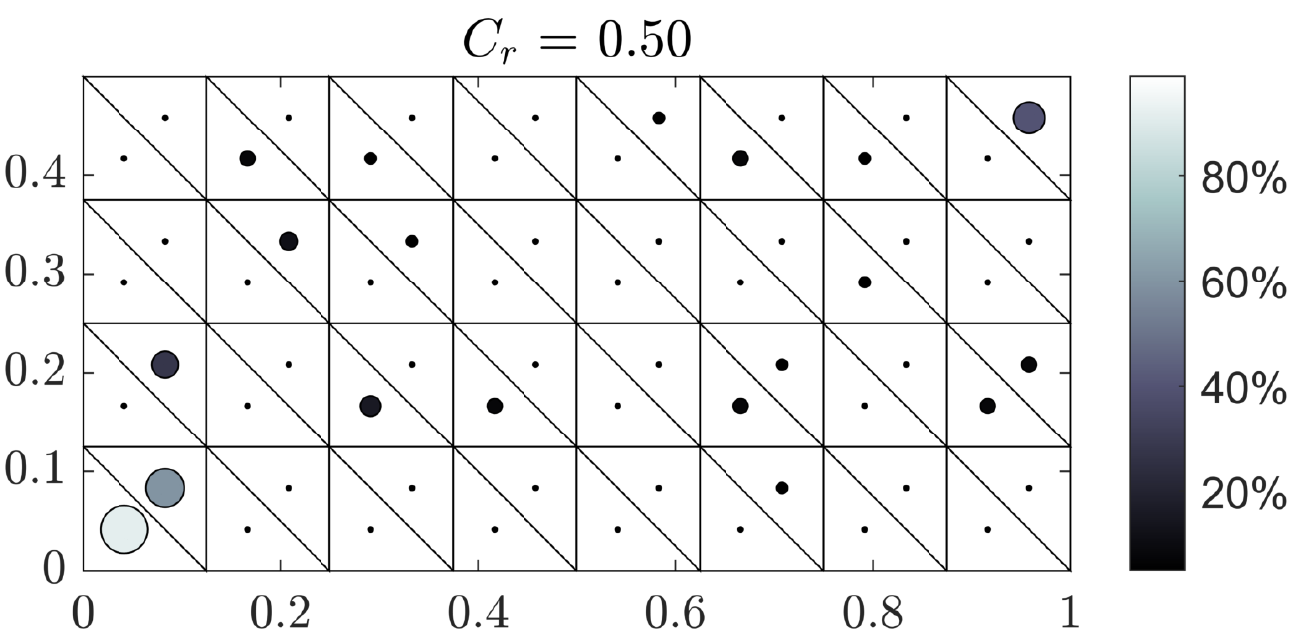}\hspace{0.5cm}
	\includegraphics[width=0.45\textwidth]{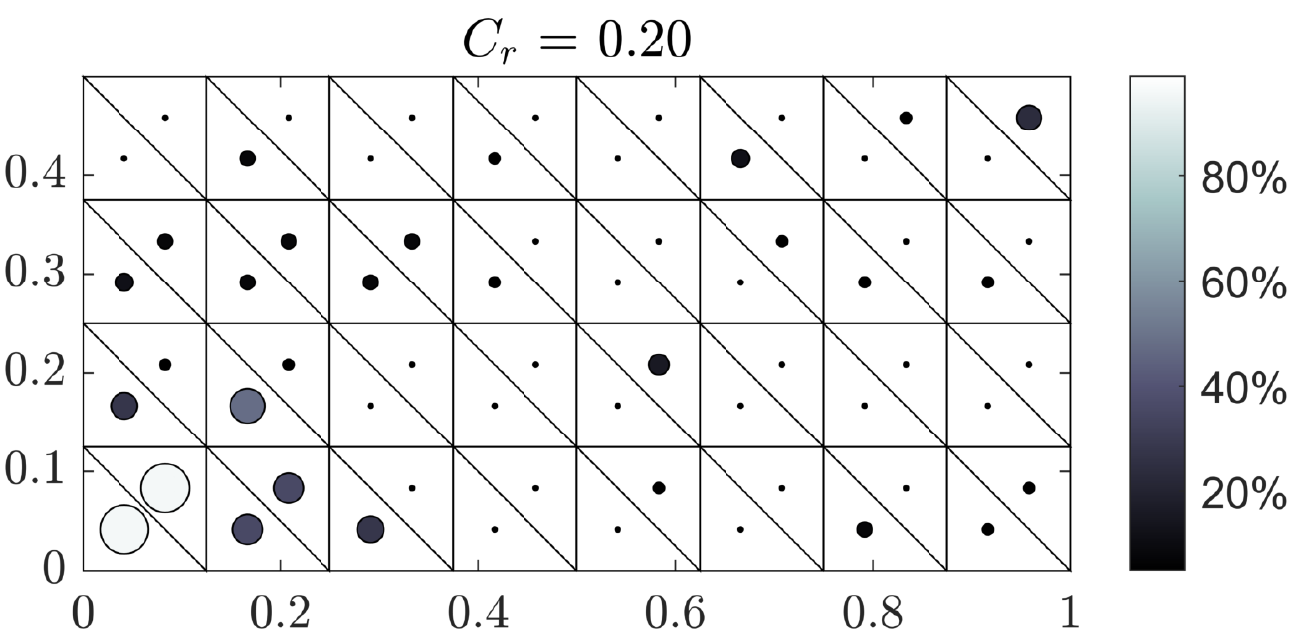}\\
	\includegraphics[width=0.45\textwidth]{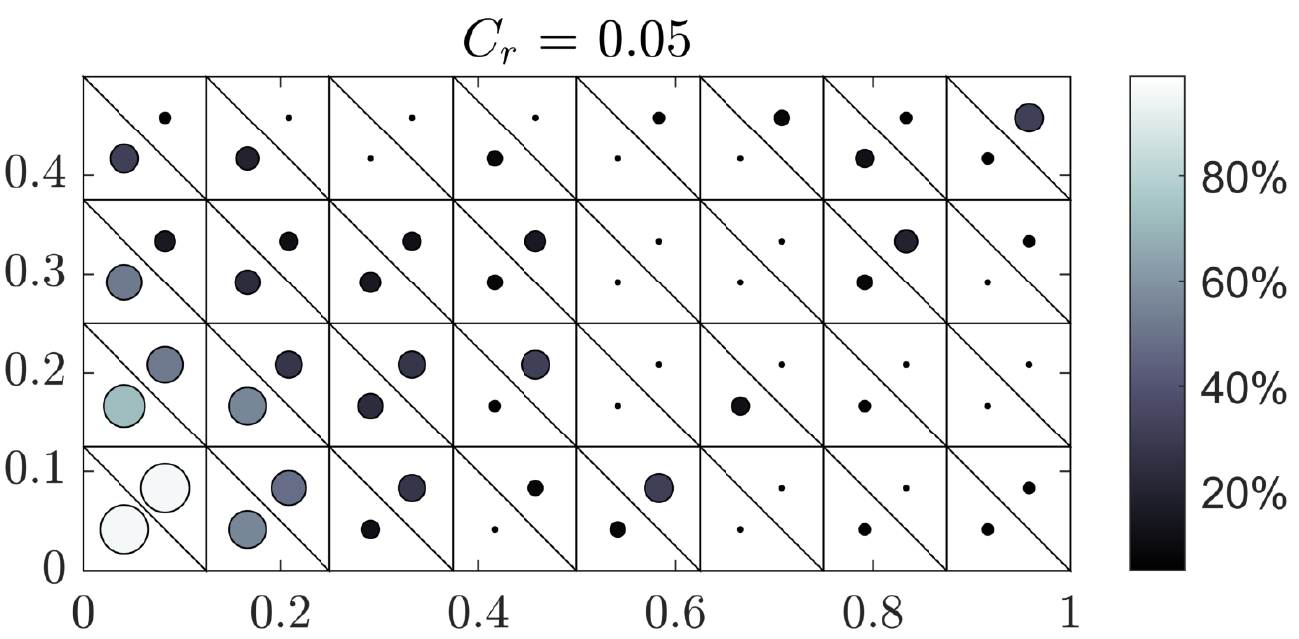}\hspace{0.5cm}
	\includegraphics[width=0.45\textwidth]{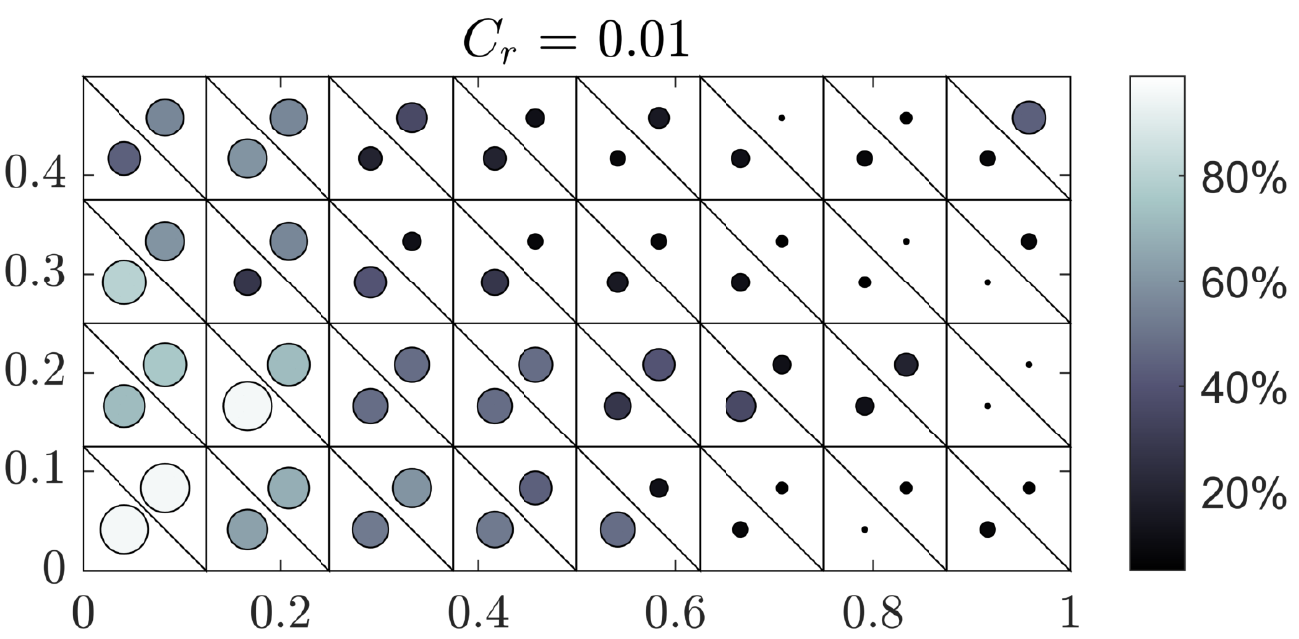}
	\captionof{figure}{The results of the macro-scale adaptive strategy for different values of the refinement parameter $C_r=0.5, \, 0.2, \, 0.05, $ and $0.01$. Different intensities and sizes indicate the percentage of times that each macro-scale element was active.}
	\label{fig:active_nodes}
\end{figure}

\begin{table}[htpb!]
	\centering
	\renewcommand{\arraystretch}{1}
	\begin{tabular}{c|c|c|c|c|c|c}
		\hline
		$C_r$ & \#Active & \%\#Active & $E_u$ & \%$E_u$ & $E_{\bphi}$ & \%$E_{\bphi}$ \\
		\hline
		$0.50$  &$82$  & $5.13$\%   & $8.26$E-$3$ & $5.23$\% & $2.00$E-$2$  & $10.16$\%  \\
		$0.20$  &$134$ & $8.38$\%   & $7.11$E-$3$ & $4.50$\% & $1.26$E-$2$  & $6.41$\%  \\
		$0.05$ & $257$ & $16.06$\% & $2.05$E-$3$ & $1.30$\% & $4.92$E-$3$  & $2.51$\%   \\
		$0.01$ & $512$ & $32.00$\% & $7.14$E-$4$ & $0.45$\% & $1.81$E-$3$  & $0.92$\%  \\
		\hline
	\end{tabular}
	\caption{The adaptive results for $\varLambda=0.1$, $C_c = 0.2$ and a varying refining parameter $C_r$. The columns \%\#Active, \%$E_u$ and \%$E_{\bphi}$ correspond to the average percentage of the original number of active elements used in each case and the relative errors with respect to the reference solution.}
	\label{table:macro-scalerefinement}
\end{table}
In \Cref{table:macro-scalerefinement} we analyse the effect of the macro-scale adaptive strategy on the $L^2$-error of the concentration and porosity. We call $u_{\text{ref}}$ and $\bphi_{\text{ref}}$ the solutions that corresponds to $C_r = 0$, i.e., the solutions of the test case without using the macro-scale adaptive strategy. The number of active nodes in the reference case is $1600$. 
\Cref{table:macro-scalerefinement} compares the following $L^2$-errors with the number of macro-scale active elements during the whole simulation
\begin{equation*}
E_u   := \| u_{\text{ref}} - u\|_{L^2([0, \ttime];L^2(\Omega))} \quad \text{and} \quad
E_{\bphi} := \| \bphi_{\text{ref}} - \bphi\|_{L^2([0, \ttime];L^2(\Omega))}.
\end{equation*} 
As expected and coinciding with \cite{redeker2013fast}, larger values of $C_r$ imply less error control. Nevertheless, when $C_r$ increases the computational cost of the simulations decreases and the convergence of the multi-scale iterative scheme is not affected.

Finally, we show the multi-scale results of the complete algorithm when using $\lstab =1$E-$4$ and $C_r = 0.05$. \Cref{fig:microscale_Ex1_sol} shows the evolution of the phase field corresponding to three different macro-scale locations. There we also show the corresponding micro-scale mesh that captures the movement of the phase-field transition zone.
\begin{figure}[htpb!]
	\centering
	\includegraphics[width=0.32\textwidth]{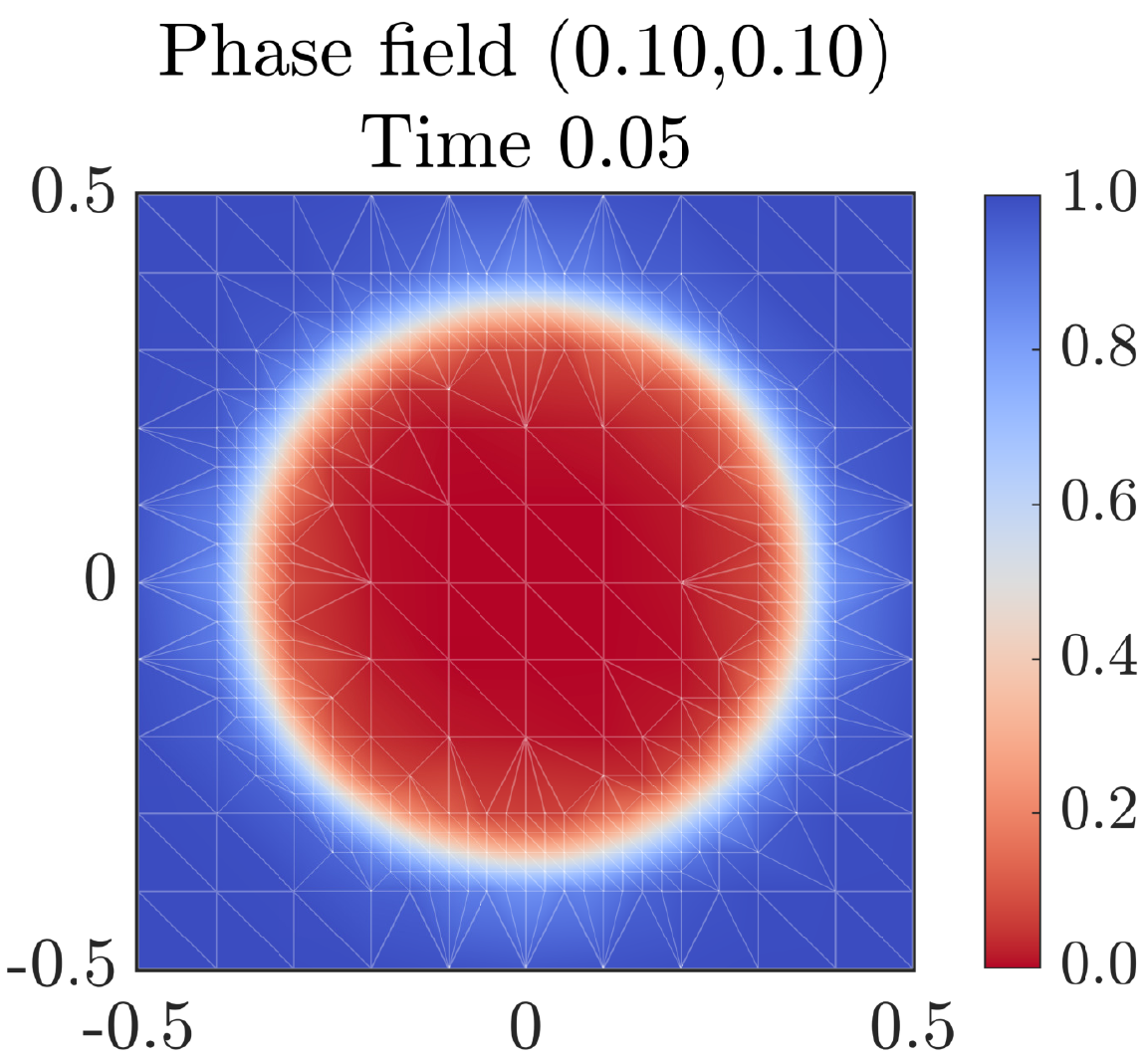}
	\includegraphics[width=0.32\textwidth]{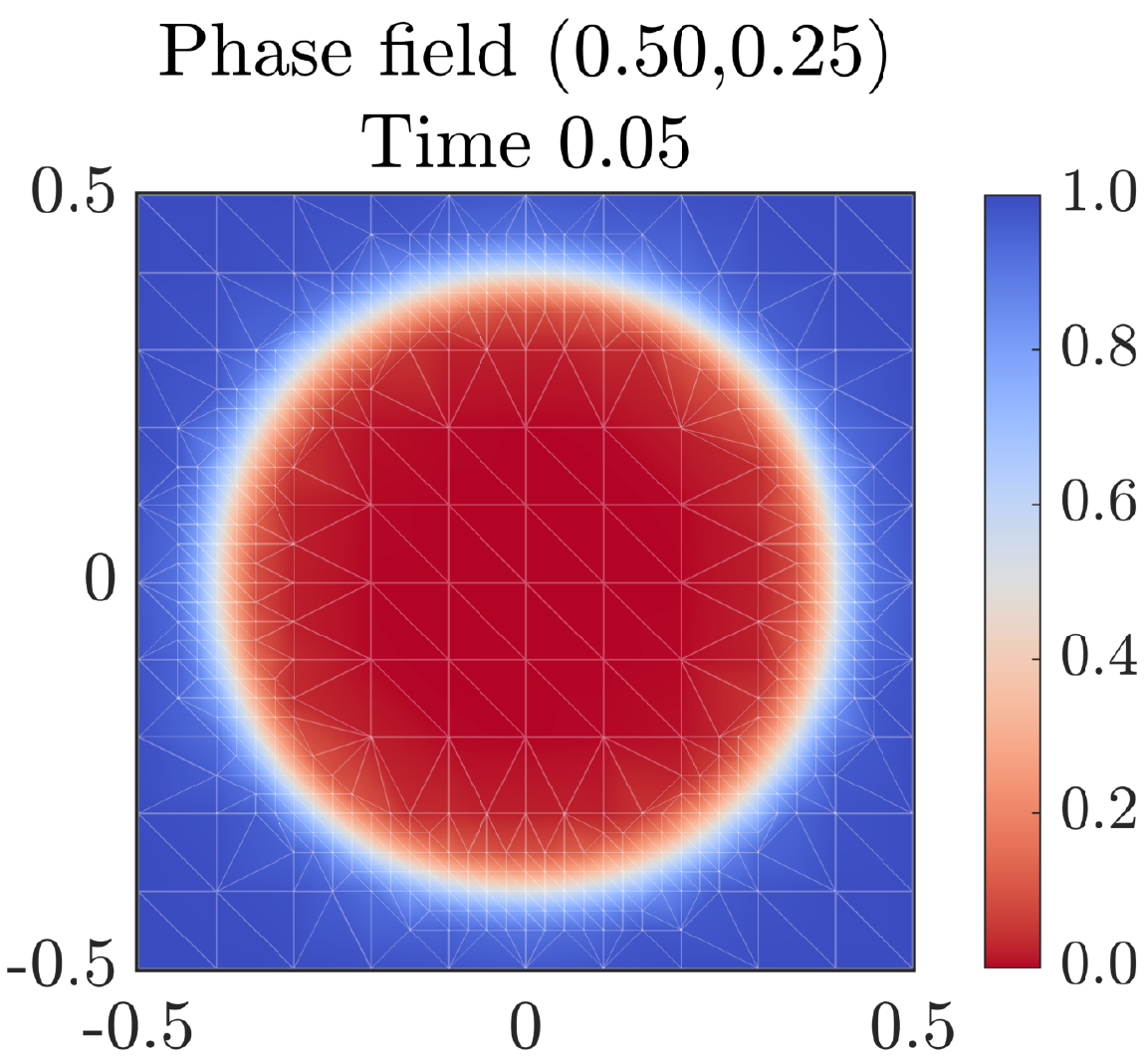}
	\includegraphics[width=0.32\textwidth]{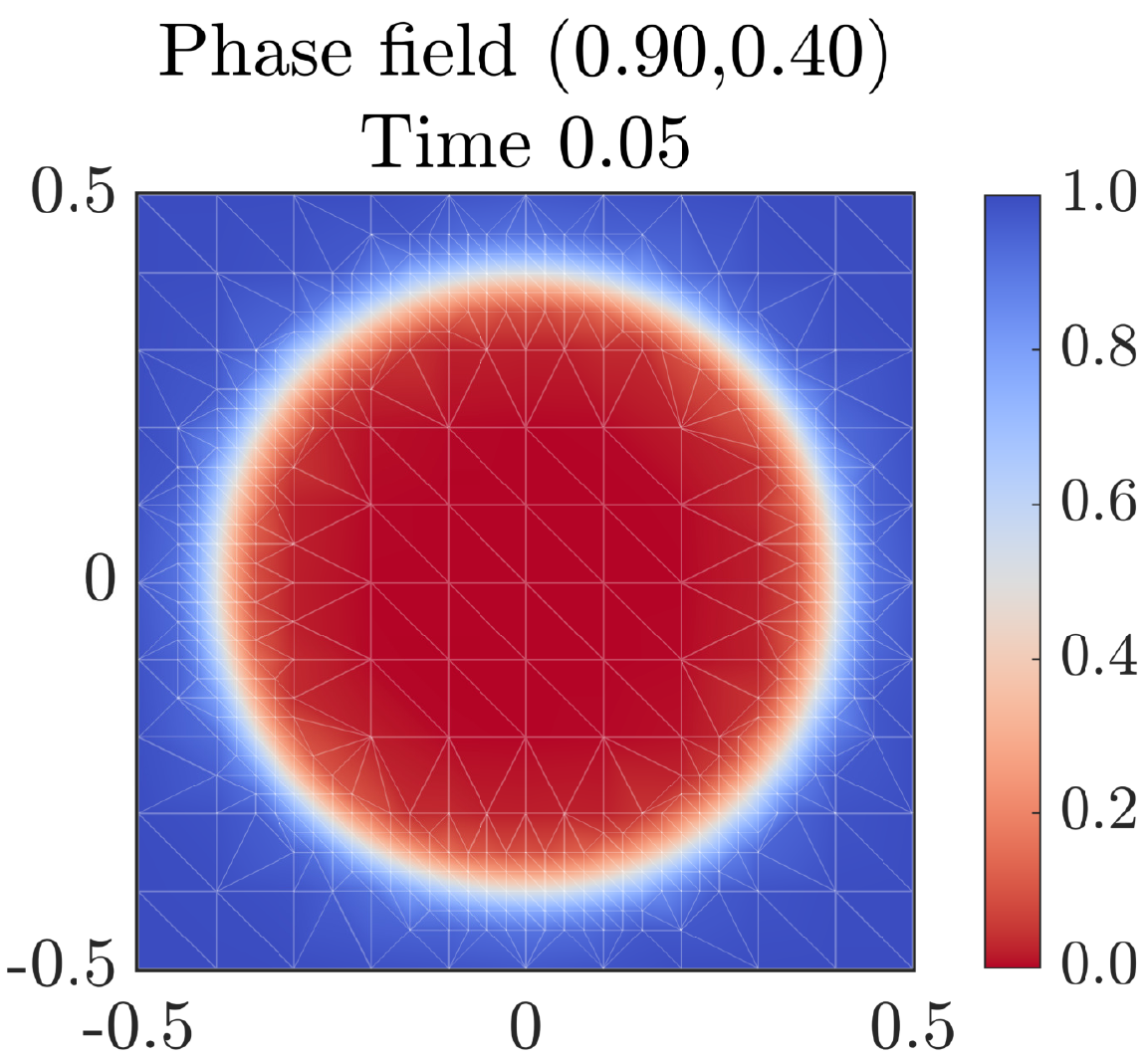}\\
	\includegraphics[width=0.32\textwidth]{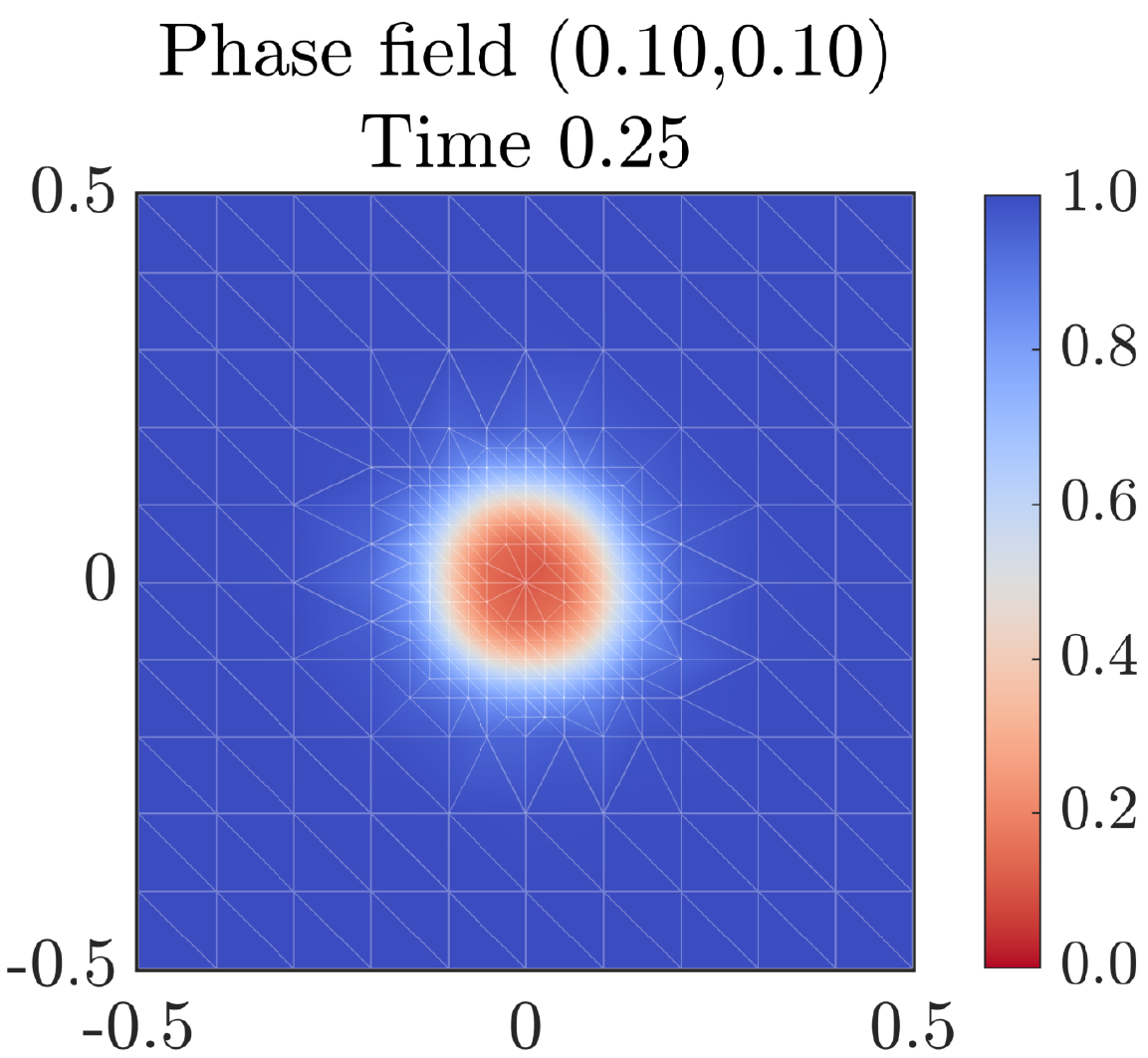}
	\includegraphics[width=0.32\textwidth]{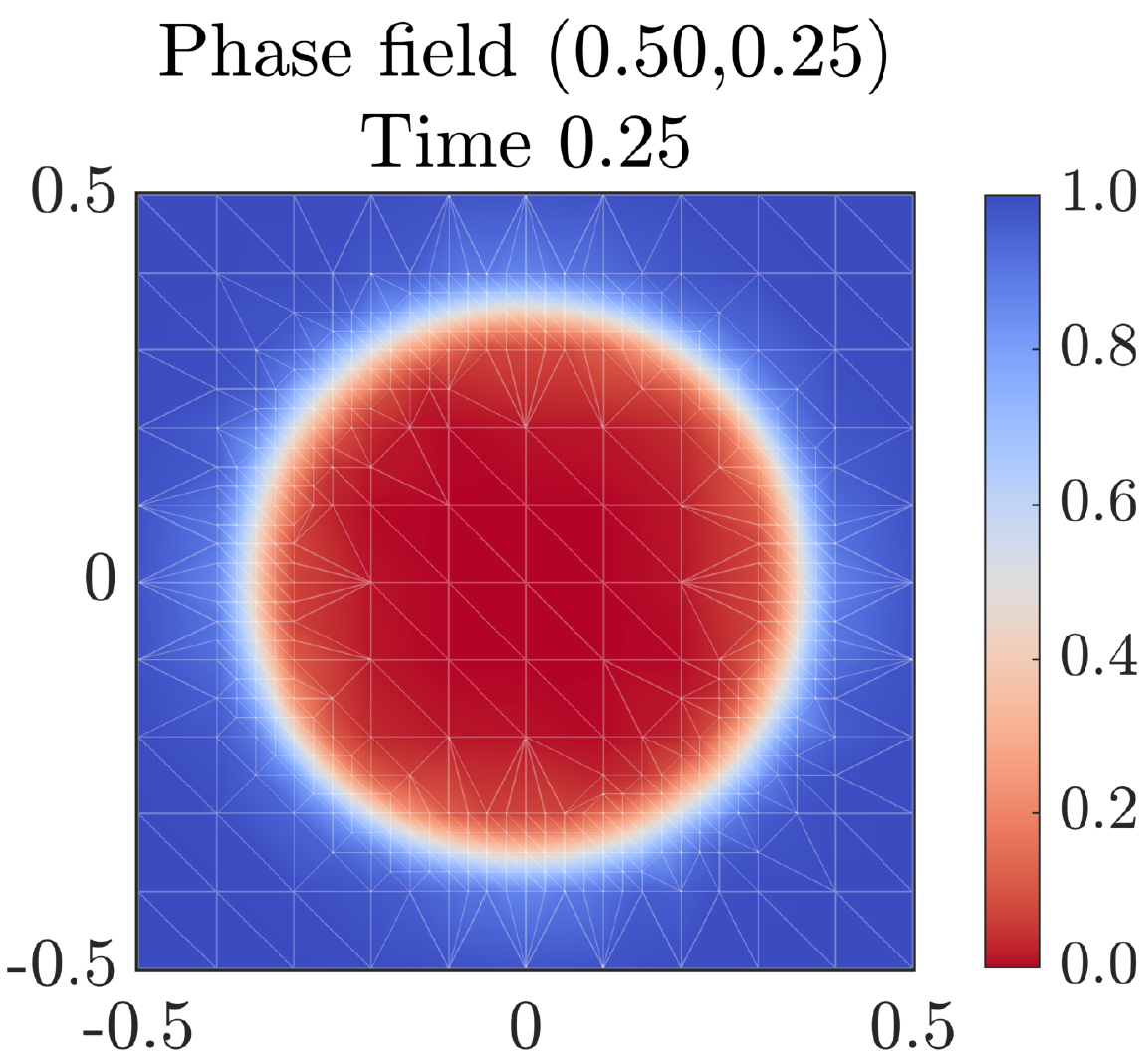}
	\includegraphics[width=0.32\textwidth]{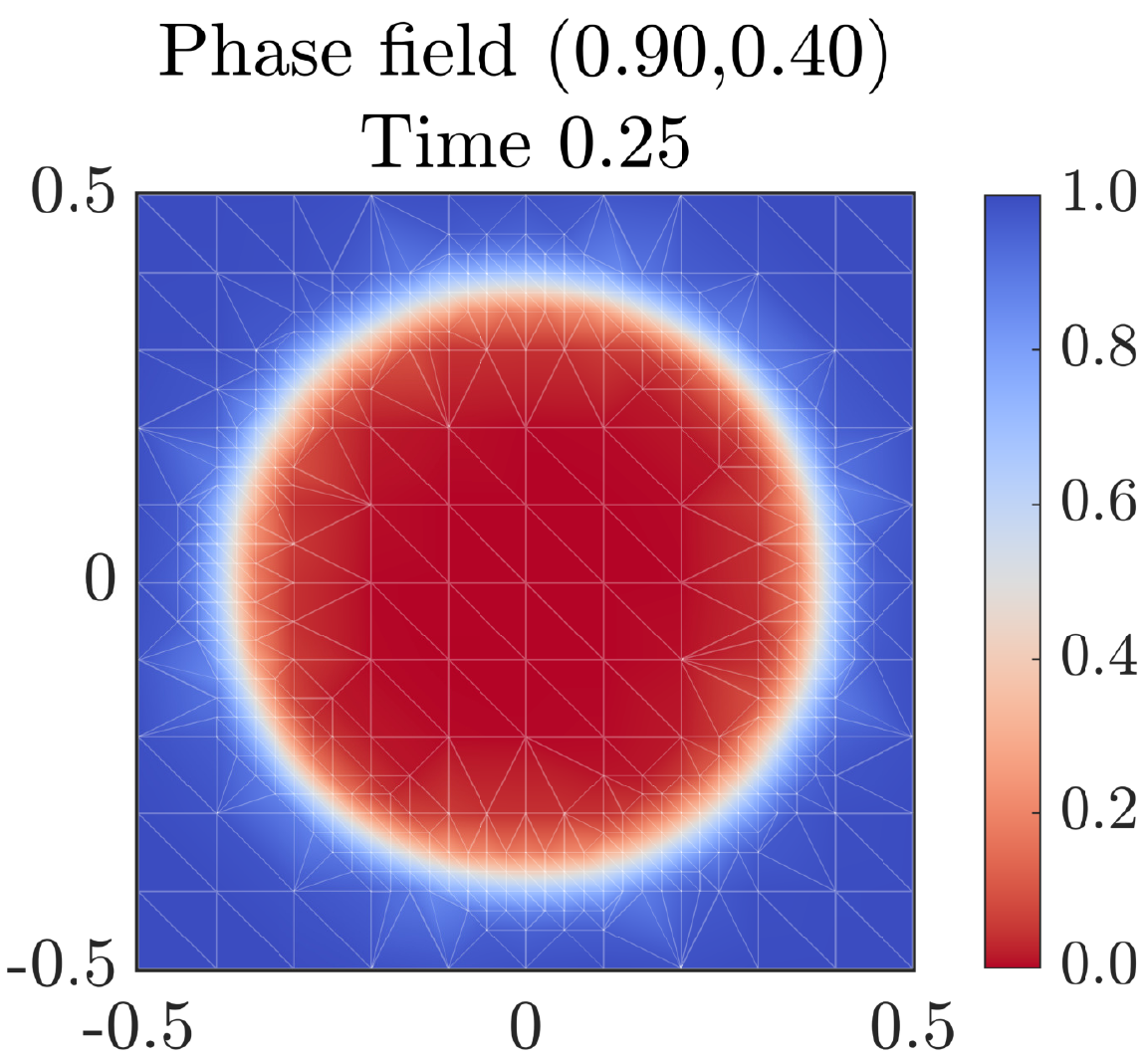}
	\captionof{figure}{The evolution of the phase fields corresponding the macro-scale locations $\xnn=(0.1, 0.1)$, $\xnn=(0.5, 0.25)$, $\xnn=(0.9, 0.4)$ (left to right) at two times $t^n = 0.05$ (top) and $t^n = 0.25$ (bottom).}
	\label{fig:microscale_Ex1_sol}
\end{figure}

The macro-scale solute concentration and porosity are displayed in \Cref{fig:macroscale_Ex1_sol}. The effective parameters are shown in \Cref{fig:effec_Ex1_sol}. We highlight that even if we are not computing the flow in this case, the effective permeability can still be calculated. Where the concentration decreases, it induces the dissolution of the mineral, which then increases the porosity, diffusivity and permeability until the micro-scale cells reach the maximum porosity $\bphi_M$.
\begin{figure}[htpb!]
	\centering
	\includegraphics[width=0.4\textwidth]{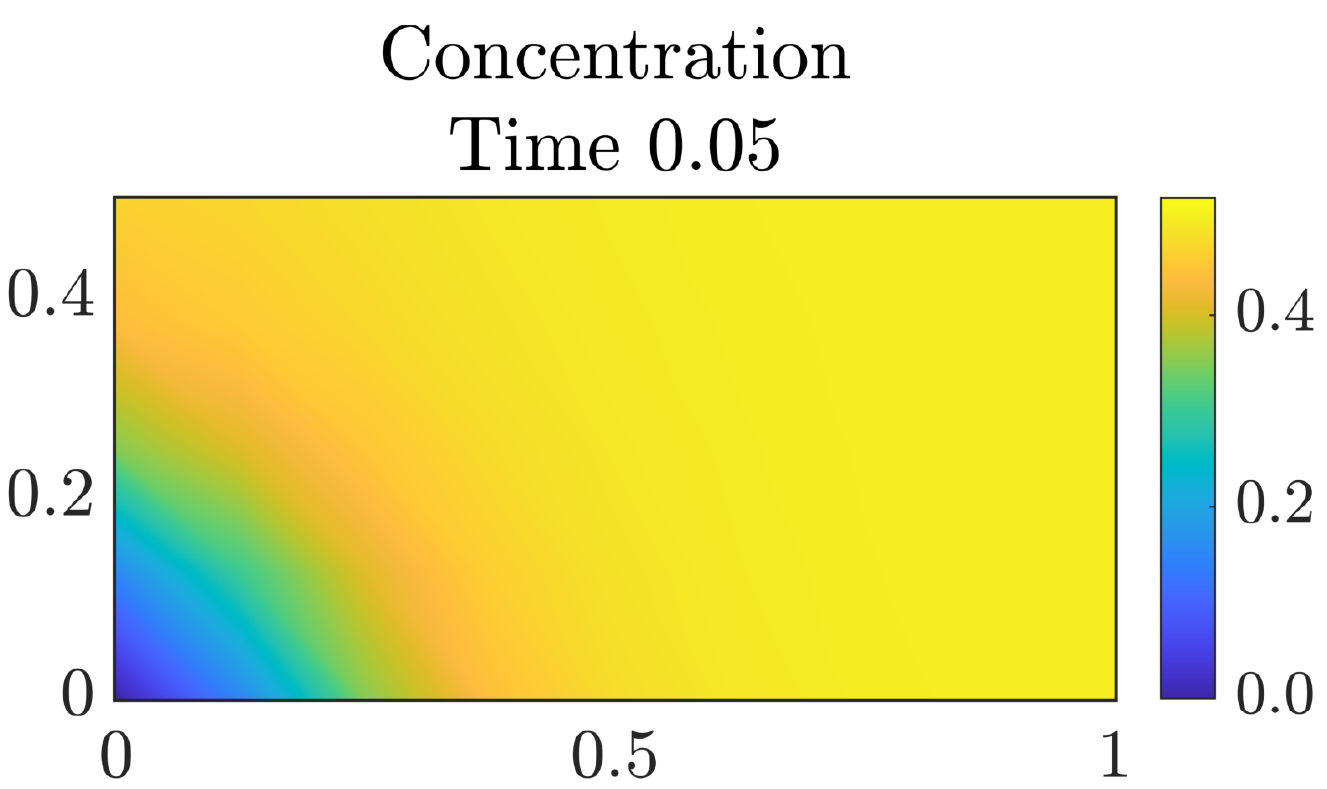}\hspace{0.5cm}
	\includegraphics[width=0.4\textwidth]{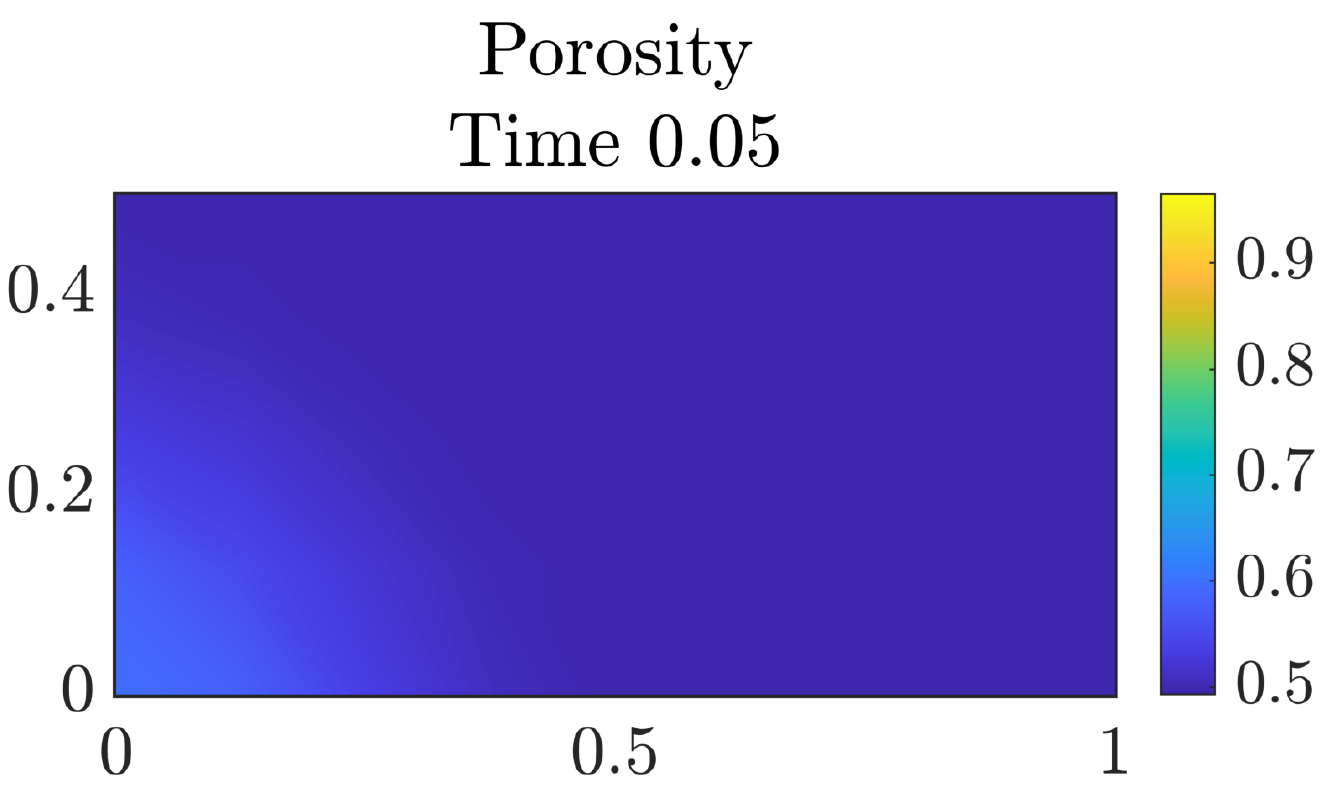}\\
	\includegraphics[width=0.4\textwidth]{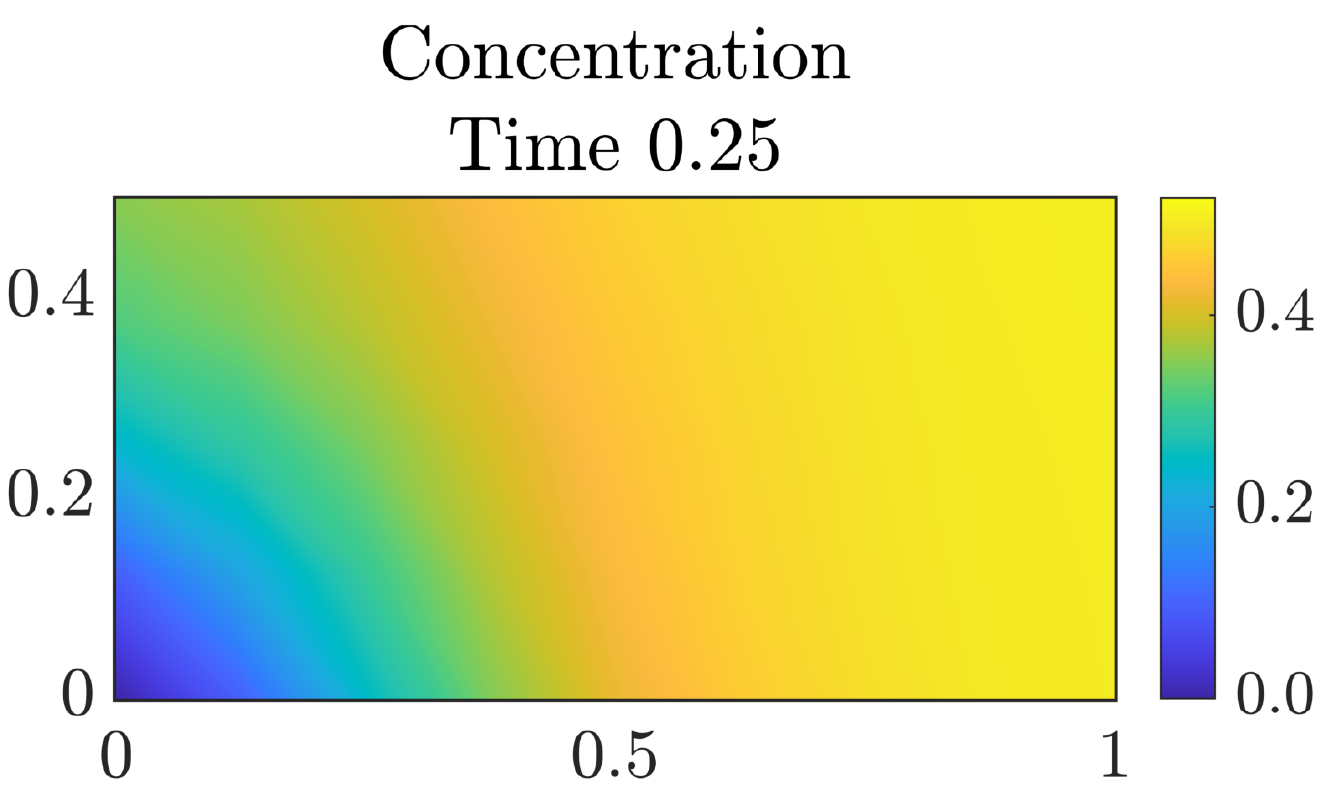}\hspace{0.5cm}
	\includegraphics[width=0.4\textwidth]{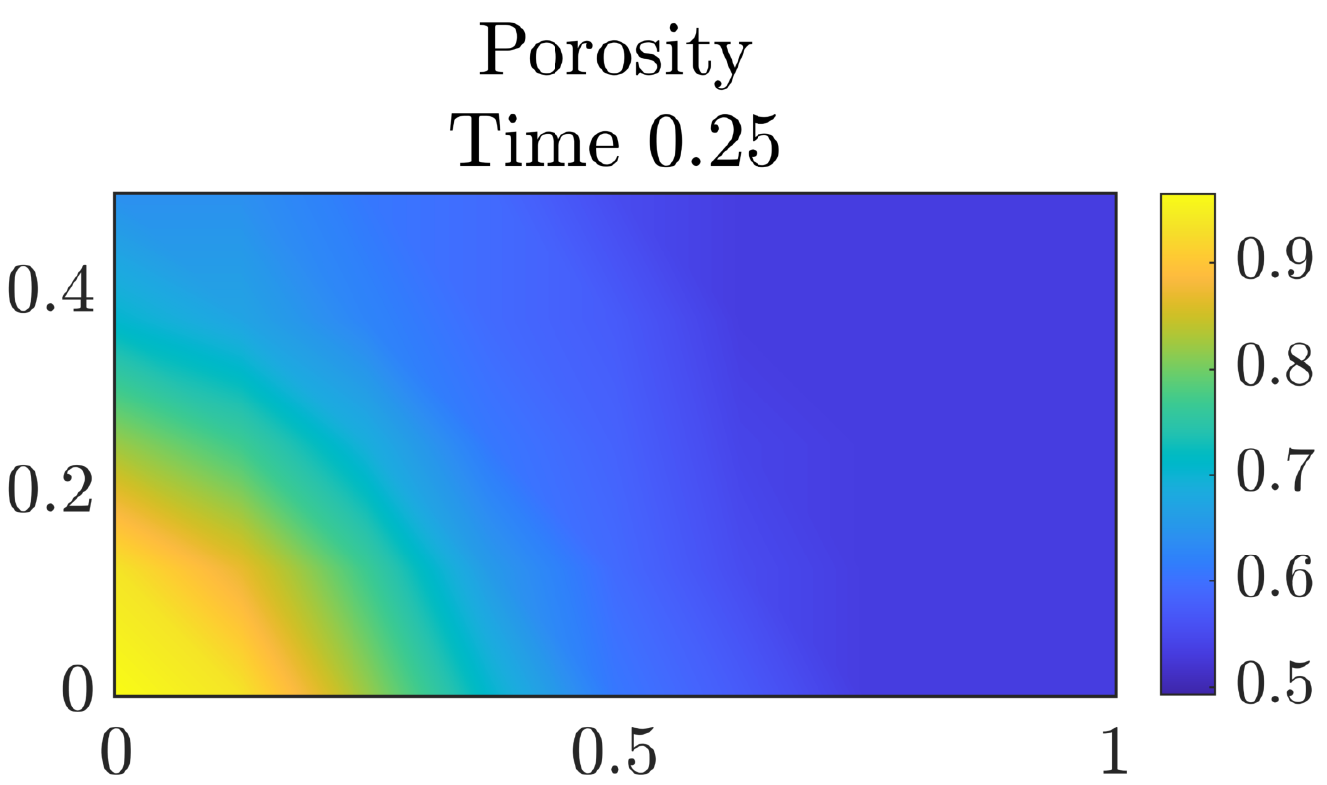}
	\captionof{figure}{The numerical solution of the concentration $u^n$ (left) and porosity $\bphi$ (right) at two times $t^n = 0.05$ (top) and $0.25$ (bottom).}
	\label{fig:macroscale_Ex1_sol}
\end{figure}
\begin{figure}[htpb!]
	\centering
	\includegraphics[width=0.4\textwidth]{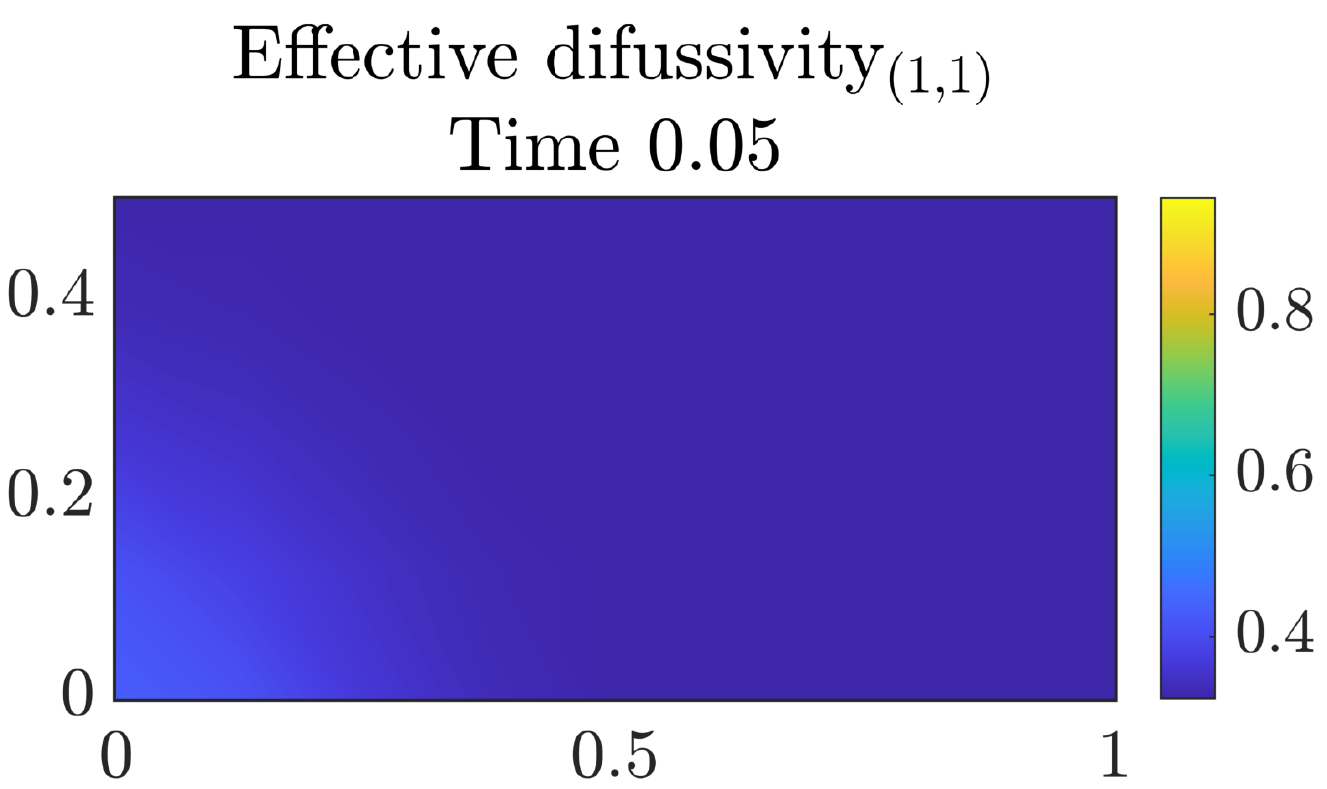}\hspace{0.5cm}
	\includegraphics[width=0.4\textwidth]{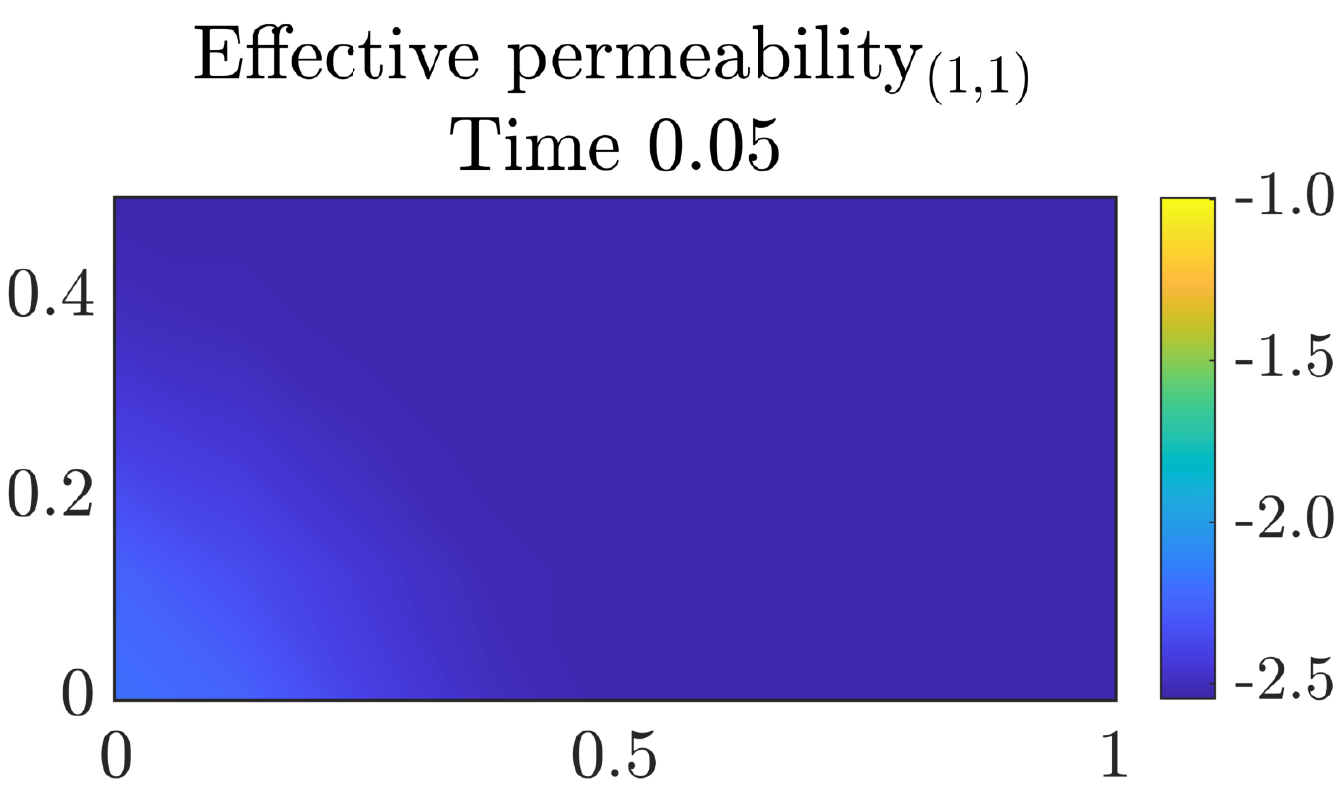}\\
	\includegraphics[width=0.4\textwidth]{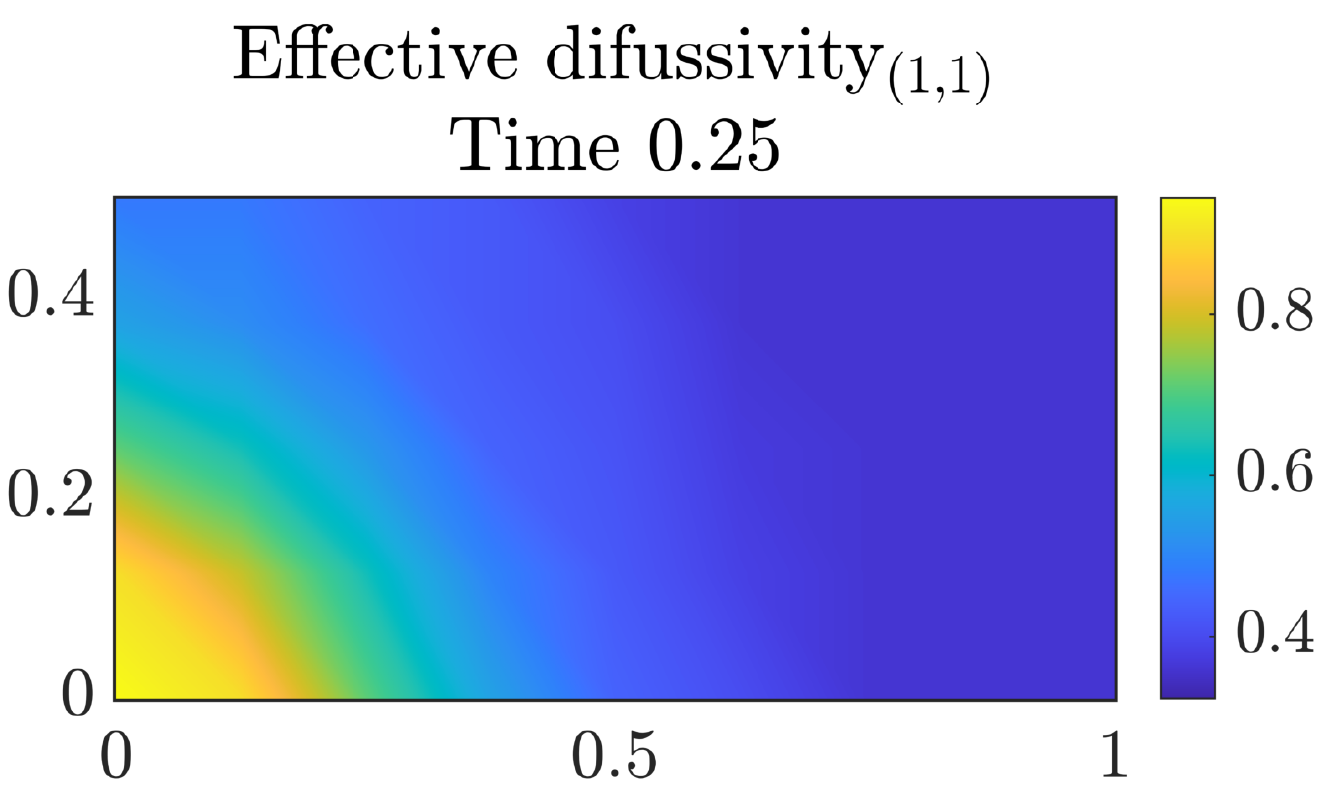}\hspace{0.5cm}
	\includegraphics[width=0.4\textwidth]{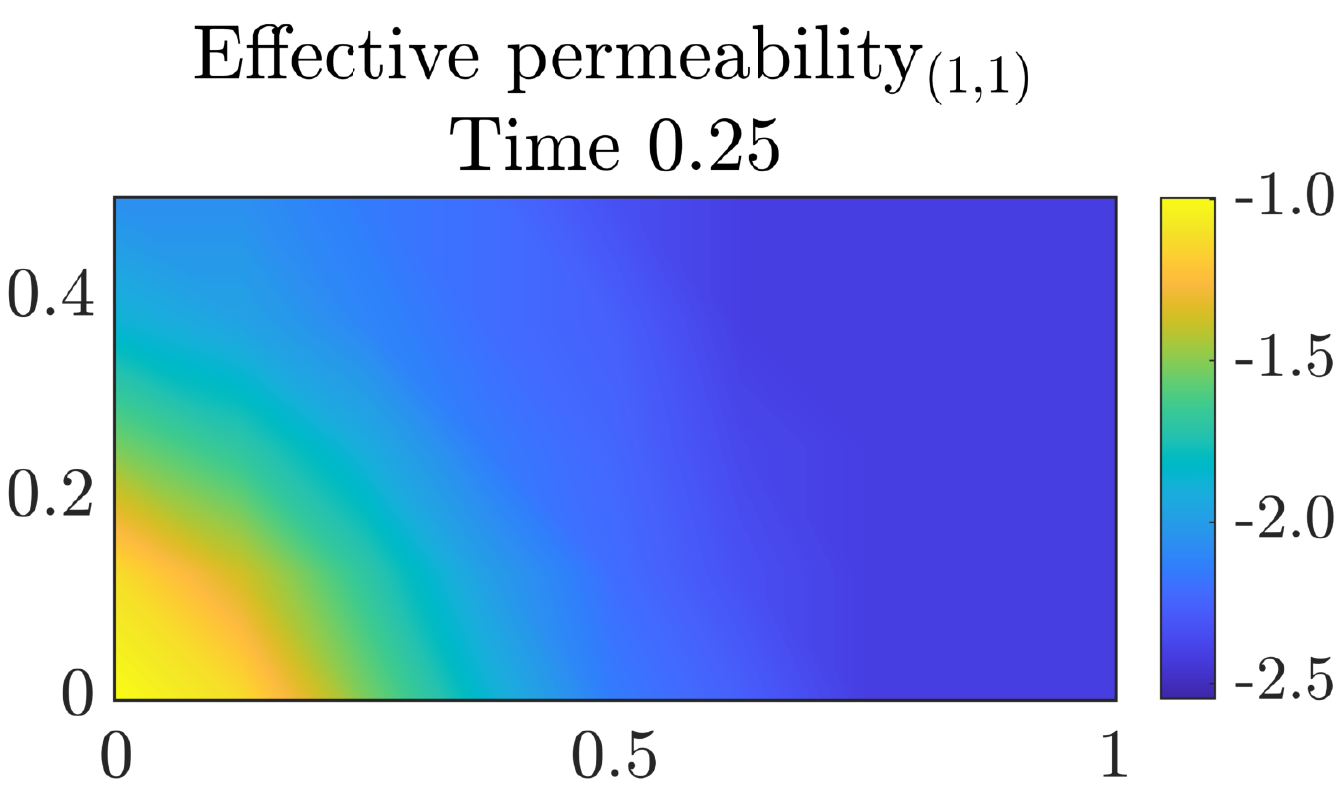}
	\captionof{figure}{The first components of the effective difussivity tensor (left) and the effective permeability tensor ($Log_{10}$) (right) at two times $t^n =0.05 $ (top) and $0.25$ (bottom).}
	\label{fig:effec_Ex1_sol}
\end{figure}

Due to the symmetry of the phase field at the micro scale, the expected results are isotropic effective tensors. The non-diagonal components of $\aef$ and $\kef$ are close to zero and can be neglected. Moreover, due to the similarity between $\kef_{1, 1}$ and $\kef_{2, 2}$ and between $\aef_{1, 1}$ and $\aef_{2, 2}$ we only show one of the components of the effective parameters in \Cref{fig:effec_Ex1_sol}.

In this test case the average number of degrees of freedom in both scales is $2.2$E+$5$ per time step. At the micro scale we have $64$ elements and for each active element we solve the phase-field problem and update the porosity and the effective parameters at each iteration. All the micro-scale problems have been solved in parallel.

Finally, in \Cref{fig:error_Ex1} we show the convergence of $\epsilon_M^{n, i}$ at different times. The linear convergence of the multi-scale iterative scheme is evident in \Cref{fig:error_Ex1}. We highlight that the total number of iterations in the multi-scale iterative scheme does not increase in time. By comparing \Cref{fig:error_Ex1} and \Cref{fig:Lstab} we evidence that the convergence of the multi-scale iterative scheme is not being affected by the macro-scale adaptivity. 
\begin{figure}[htpb!]
	\centering
	\includegraphics[width=0.45\textwidth]{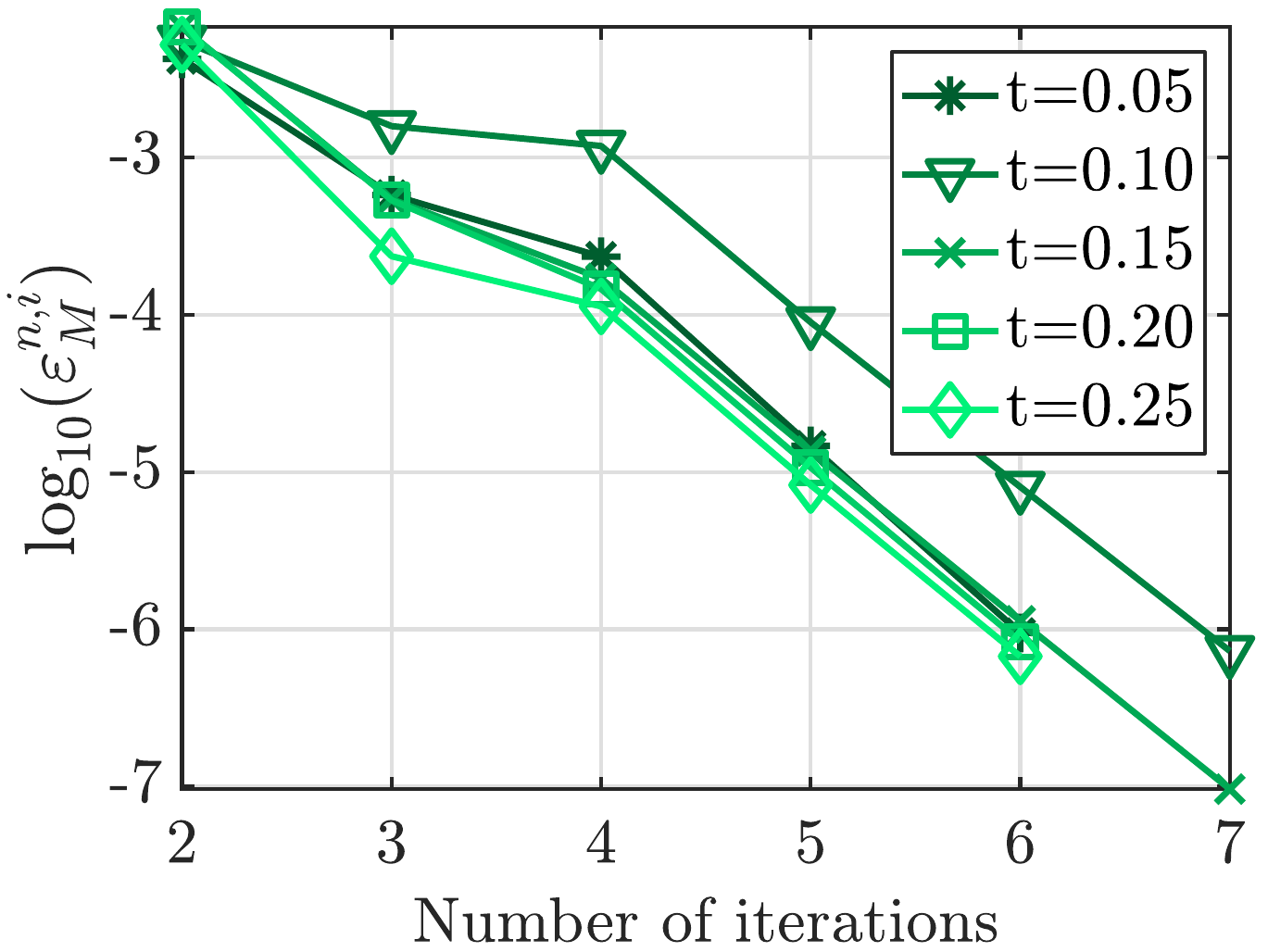}
	\captionof{figure}{The convergence of the multi-scale iterative scheme for five different times.}
	\label{fig:error_Ex1}
\end{figure}

\subsection{Test case 2. Anisotropic case}
\label{sec:ex2}

Consider the macro-scale domain $\Omega = \left( 0, 1\right) \times\left( 0, \frac{1}{2}\right) $ where the system is initially in equilibrium, i.e. the initial concentration is $u(\xnn, 0)=u_{\text{eq}}$ and $p(\xnn, 0)=0$ for all $\xnn\in \Omega$.
A dissolution process is triggered by imposing a Dirichlet condition for the concentration on the right boundary of $\Omega$, i.e., $u=0$. The Dirichlet condition for the pressure on the left boundary of $\Omega$ is $p=0.25$ and $p=0$ on the right boundary.
On the micro scale, we consider an initially inhomogeneous distribution of the mineral. We define two sub-domains of $\Omega$; the left half is $\Omega_l:= \left( 0, 0.5\right) \times \left( 0, 0.5\right) $ and the right half $\Omega_r:= \left( 0.5, 1\right) \times\left( 0, 0.5\right) $. The initial phase field is chosen to be
\begin{align*}
\phi_I(\xnn,\ynn) & = \begin{cases}
\phi_l^0(\ynn), & \text{ if } \xnn \in \Omega_l, \\
\phi_r^0(\ynn), & \text{ otherwise}, 
\end{cases} 
\end{align*}
where the micro-scale functions $\phi_l^0$ and $\phi_r^0$ are taken as follows
\begin{align*}
\phi_l^0(\ynn) & = \begin{cases}
0, & \text{if } \ynn\in[-0.4, 0.4]\times[-0.3, 0.3], \\
1, & \text{otherwise}, 
\end{cases} \\ \phi_r^0(\ynn) & = \begin{cases}
0, & \text{if } \ynn\in[-0.3, 0.3]\times[-0.4, 0.4], \\
1, & \text{otherwise}.
\end{cases}
\end{align*}

The configuration of the test case 2 is displayed in \Cref{geoex2}. With this example we show the potential of the model and the numerical strategy in a heterogeneous scenario. Here we add the flow that was dismissed during the proofs in \Cref{sec:Analysis}. The following parameters have been used in the simulation
\[\lstab=1\text{E-}4; \quad \theta_r=2; \quad C_r=0; \quad \bphi_M = 0.9686.\]
For the simulation time we take $\ttime=0.25$ and the time step is chosen to be $\dt=0.01$ as explained before.

\begin{figure}[htpb!]
	\centering
	\includegraphics[width=0.5\textwidth]{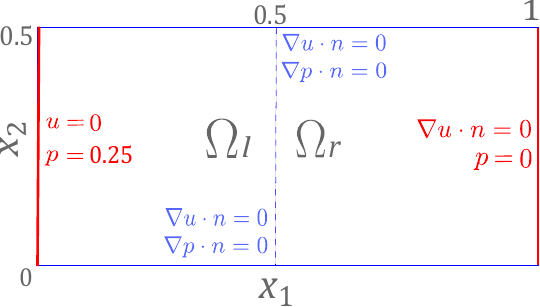}\\
	\includegraphics[width=0.35\textwidth]{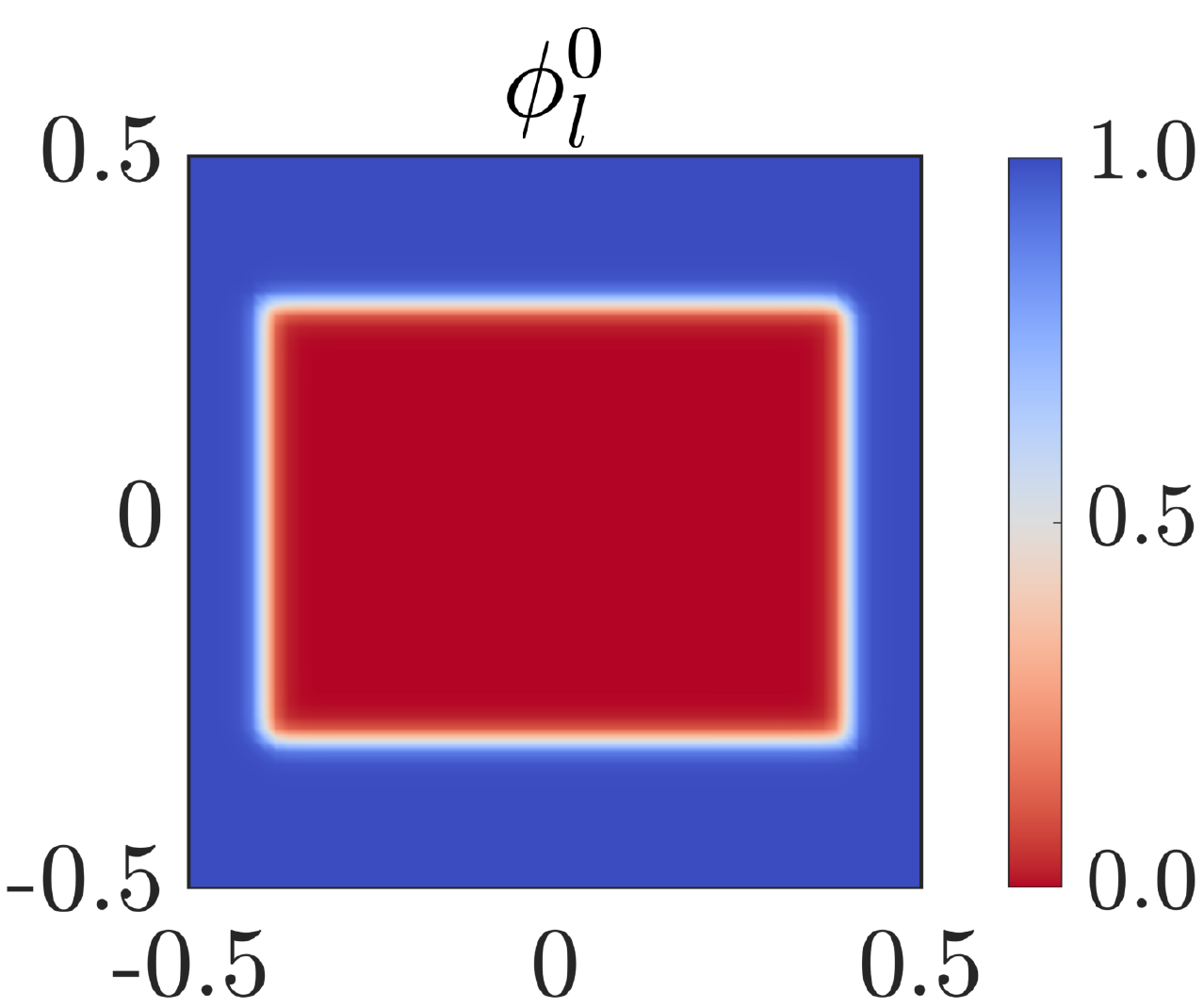}\hspace{0.5cm}
	\includegraphics[width=0.35\textwidth]{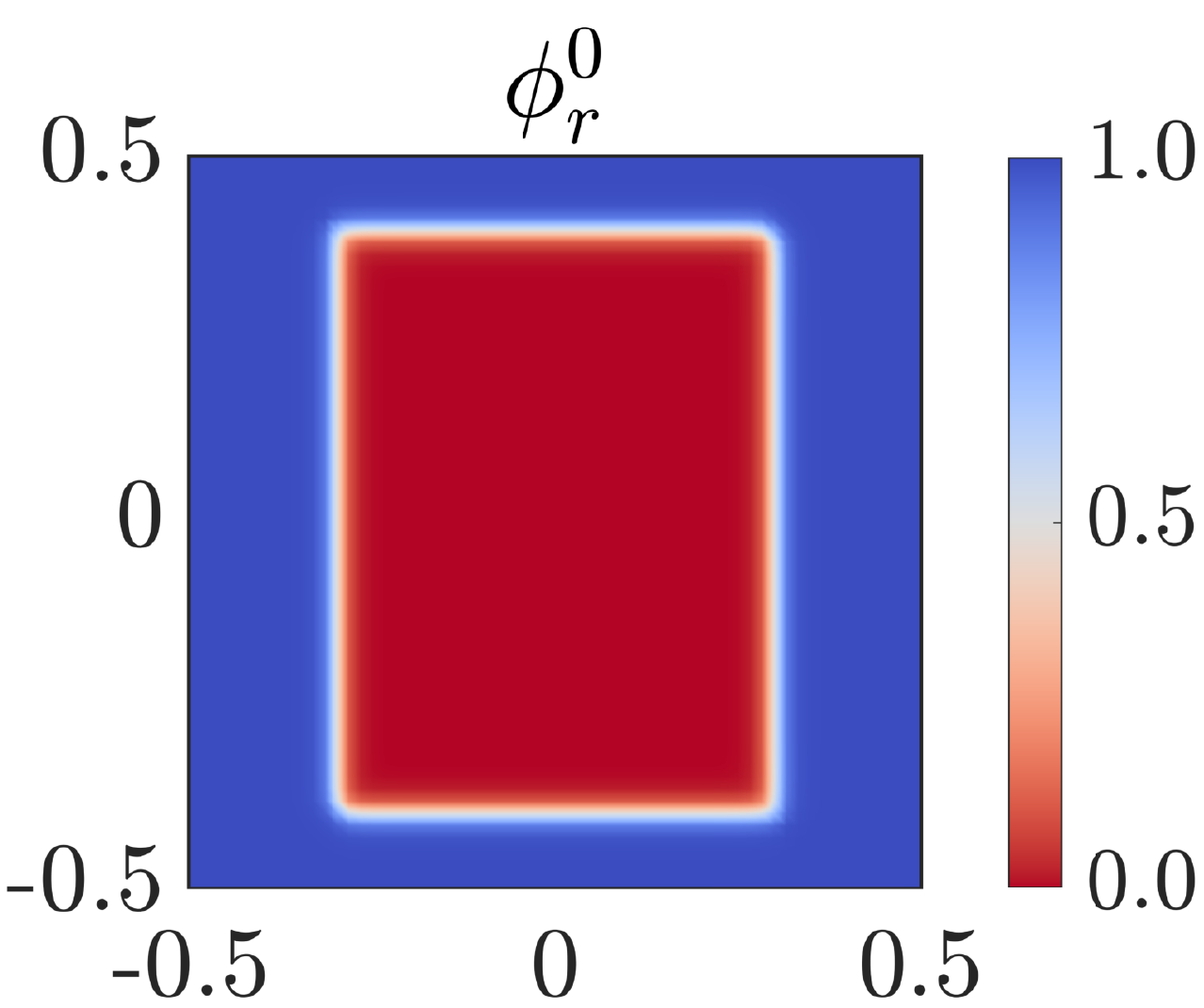}
	\captionof{figure}{The configuration of the macro scale (top) and the phase-field initial conditions (bottom) - Test case 2.}\label{geoex2}
\end{figure}

Due to the structure of this example and the chosen boundary and initial conditions, the macro-scale solution does not depend on the vertical component. Therefore the 1D projection of the macro-scale solutions in the horizontal direction is sufficient to understand the behavior of the whole system. The macro-scale adaptive strategy is unnecessary as the natural choice is to fix the nodes located at the lowest part of the macro-scale domain to be active.

In Figures \ref{fig:microscale_Ex2} and \ref{fig:microscale2_Ex2} we show the evolution of the phase field corresponding to different macro-scale locations. On each micro-scale domain $Y$ we use an initial uniform mesh with $800$ elements and the minimum diameter $h_{T_\mu}$ in the refined mesh is $h_{T_\mu}= 0.025$. Moreover, for the micro-scale non-linear solver we choose  $\llin= \max\left(|2\lambda f(u) + 8\gamma|,|2\lambda f(u) - 8\gamma|\right)$ and $\textit{tol}_\mu=1$E$-8$. 

\begin{figure}[htpb!]
	\centering
	\includegraphics[width=0.32\textwidth]{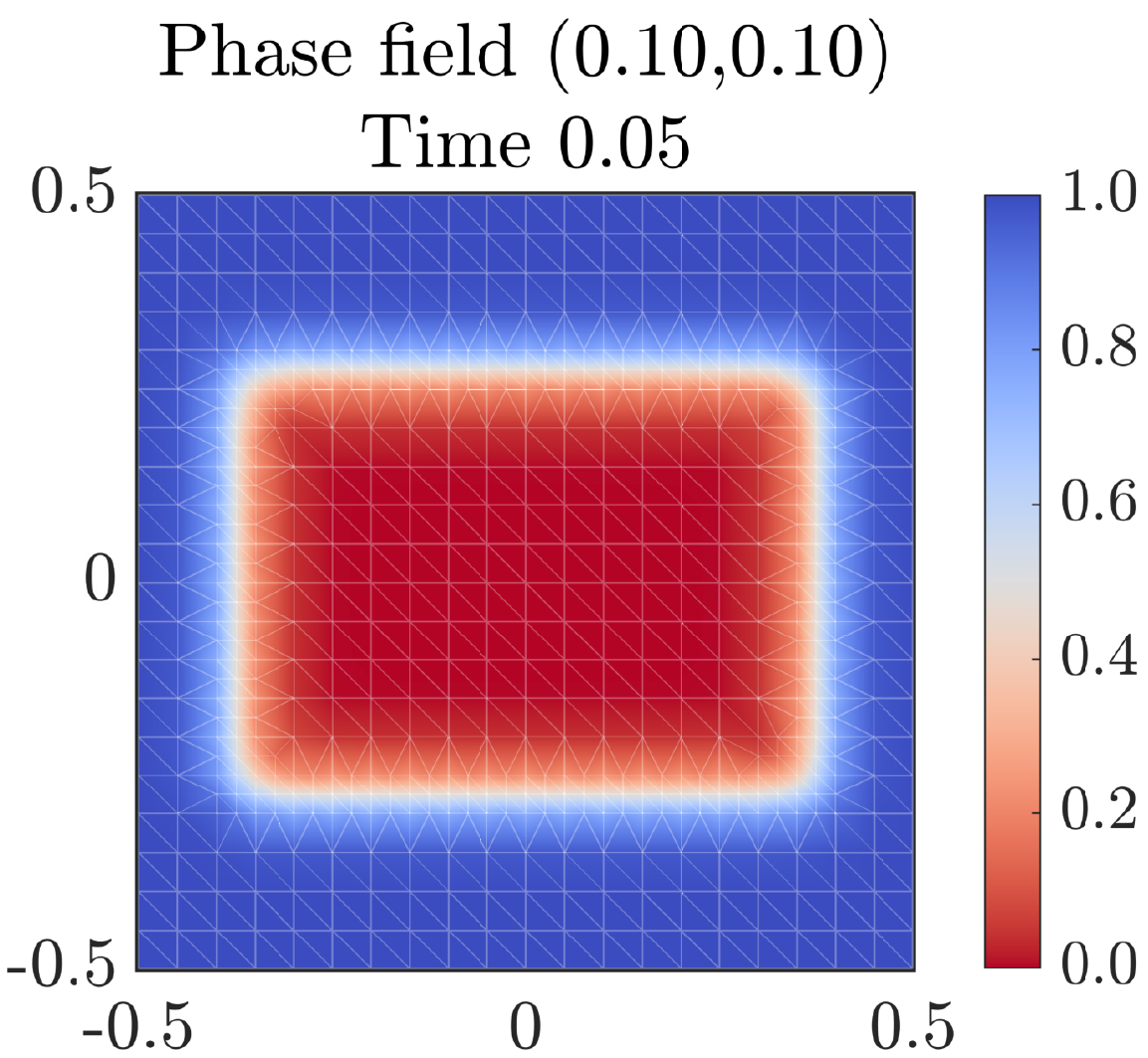}
	\includegraphics[width=0.32\textwidth]{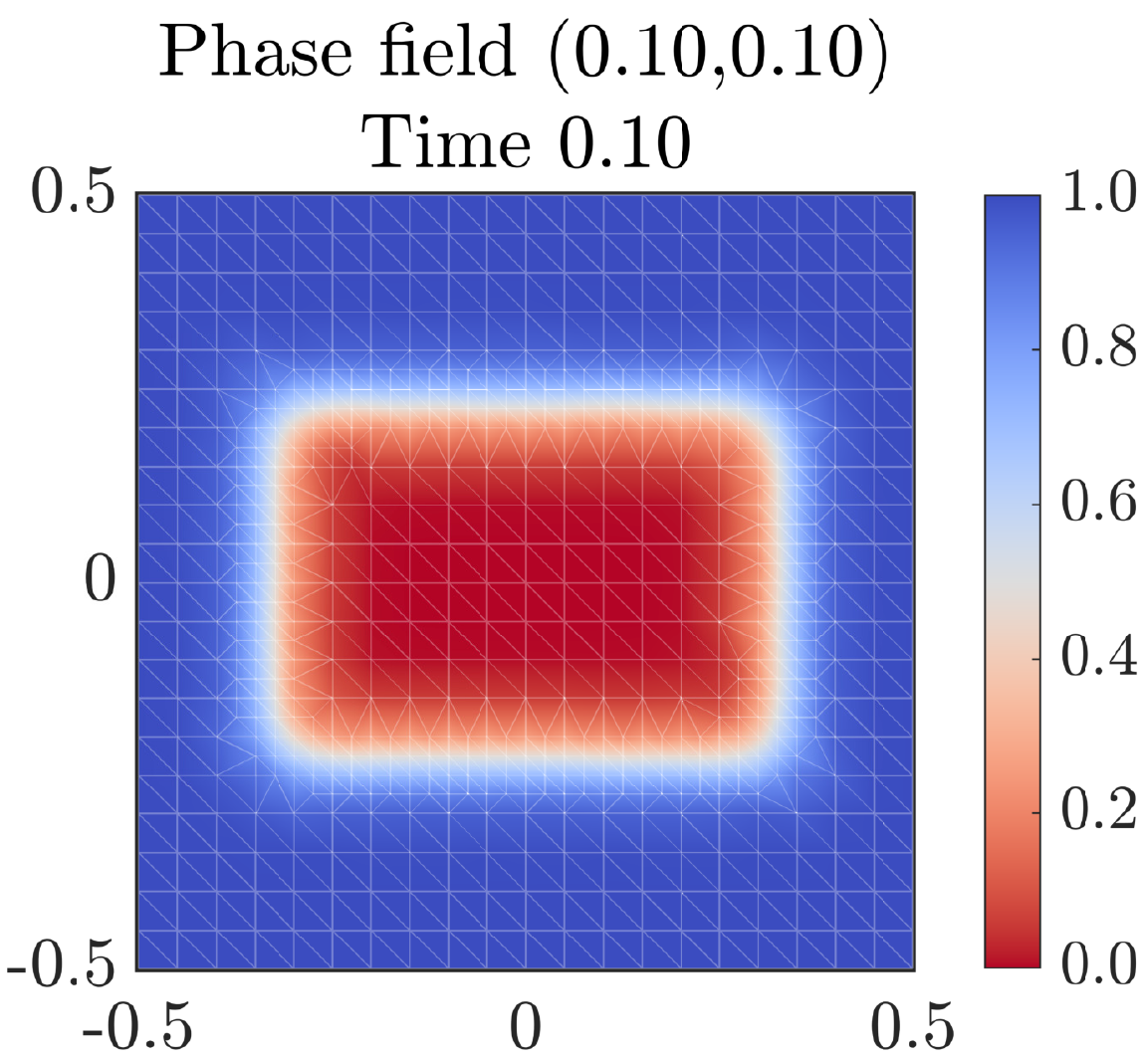}
	\includegraphics[width=0.32\textwidth]{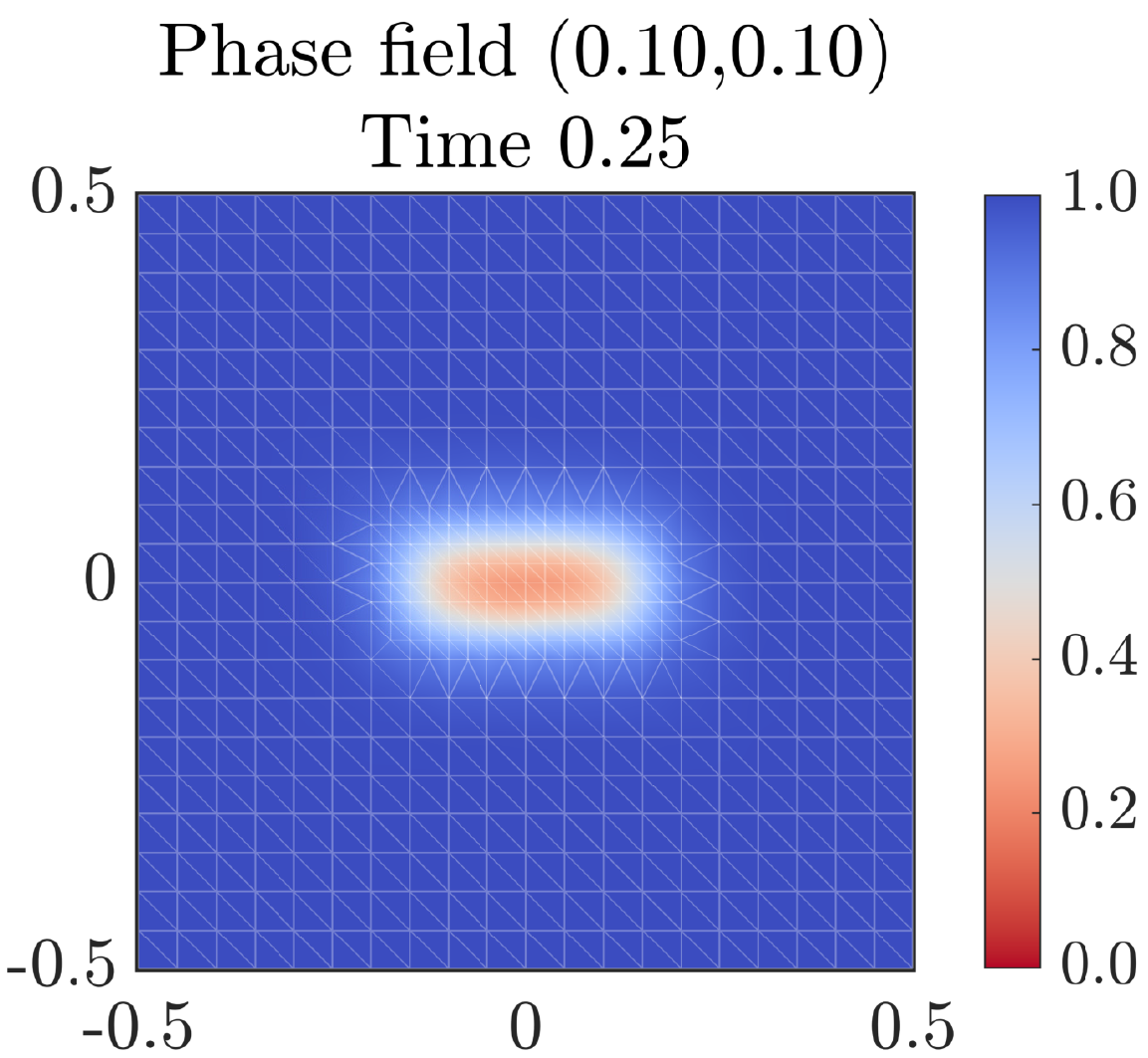}\\
	\includegraphics[width=0.32\textwidth]{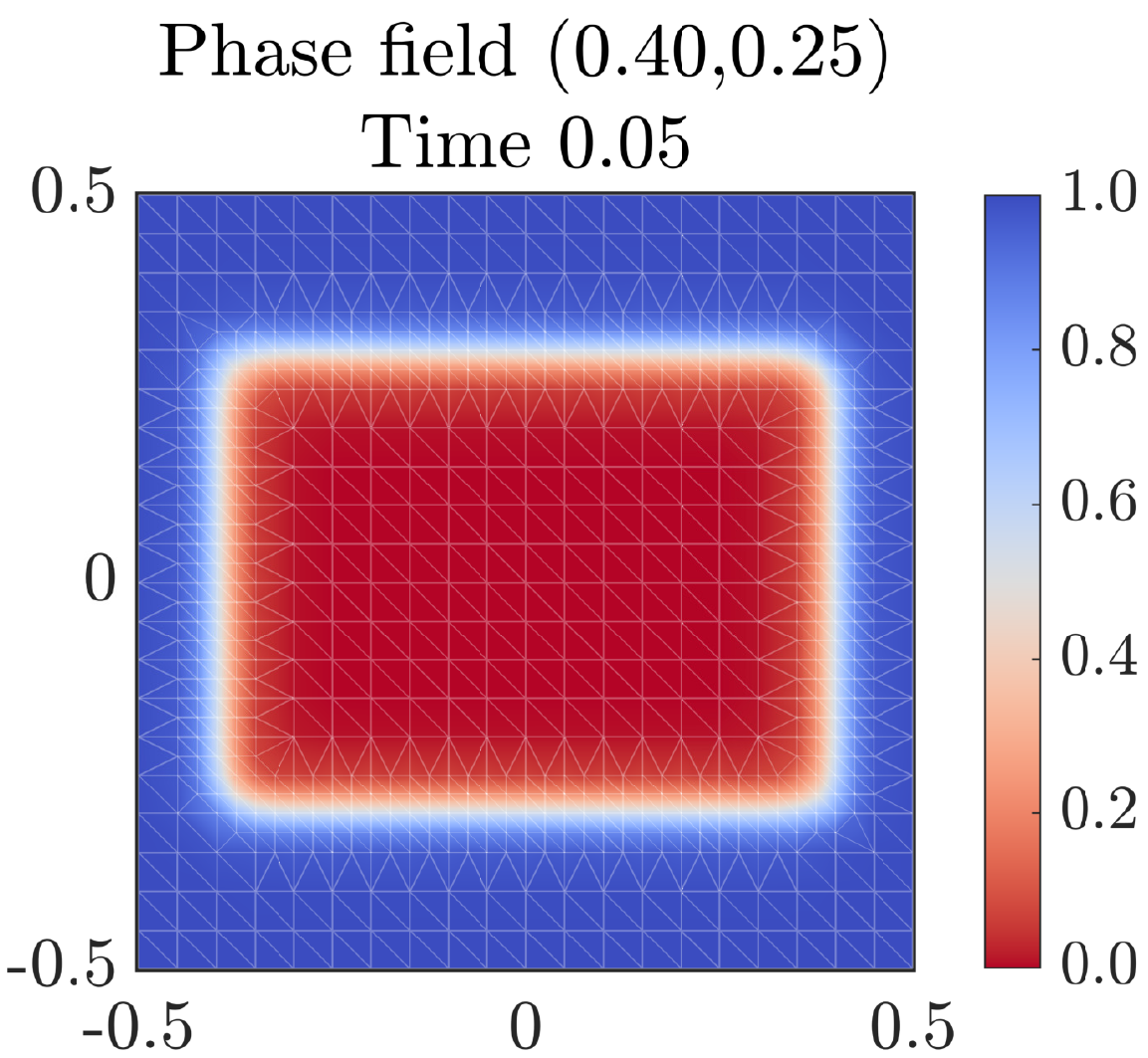}
	\includegraphics[width=0.32\textwidth]{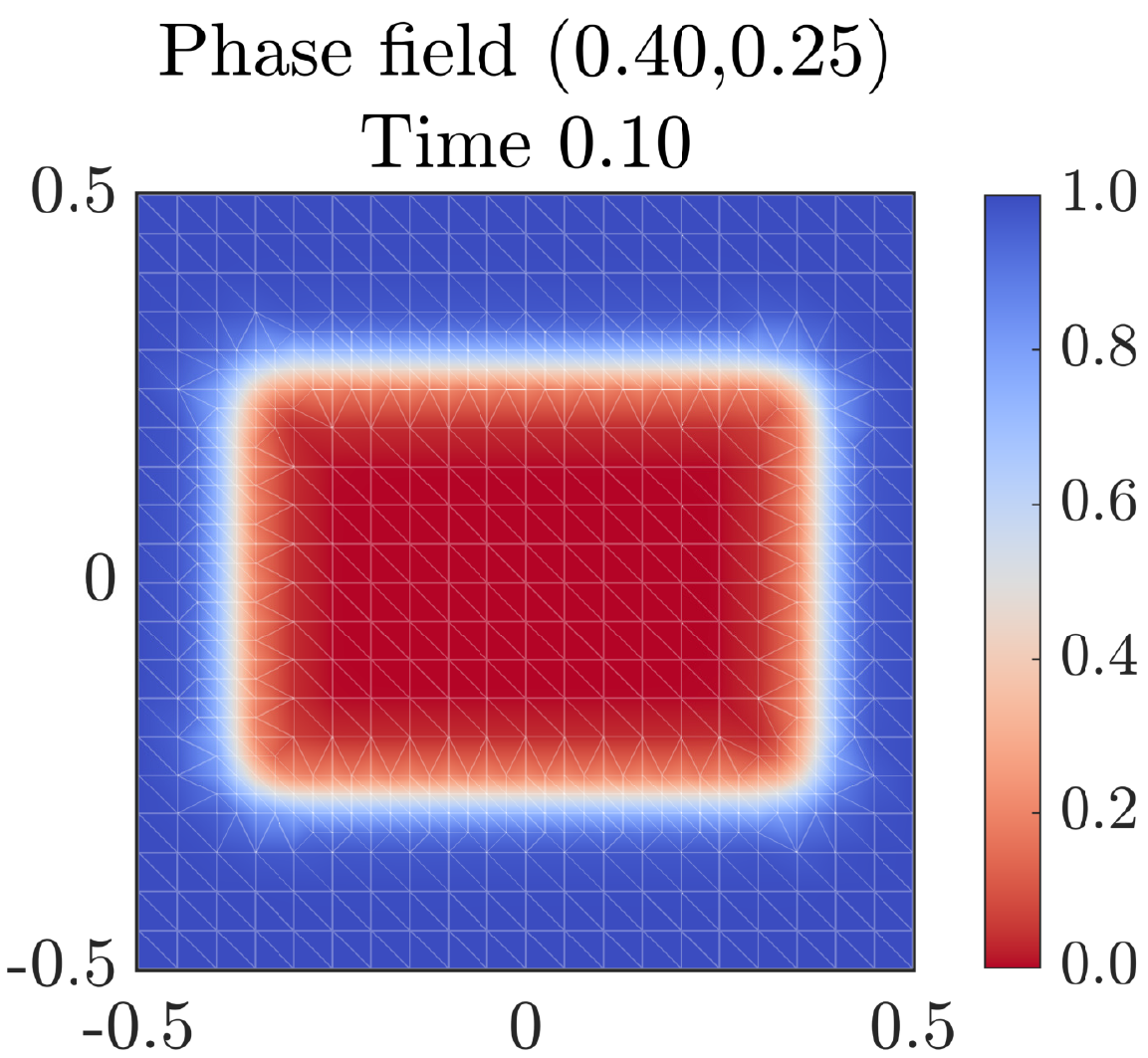}
	\includegraphics[width=0.32\textwidth]{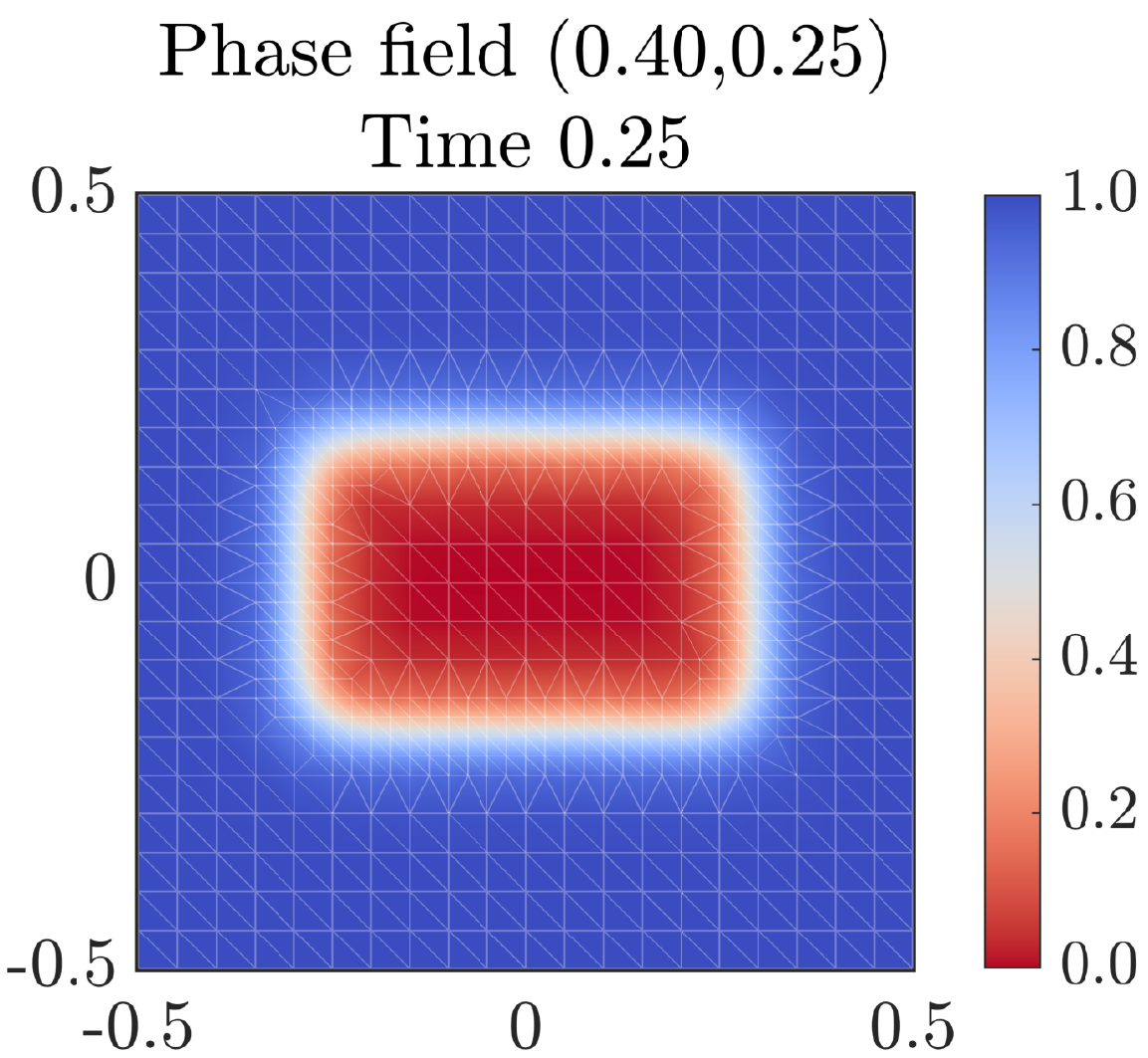}
	\captionof{figure}{The evolution of the phase fields $\phi_l$ corresponding to the macro-scale locations $\xnn=(0.1, 0.1)$ (top) and $\xnn=(0.4, 0.25)$ (bottom) at three times $t^n = 0.05, \, 0.10 $ and $0.25$ (left to right).}
	\label{fig:microscale_Ex2}
\end{figure}
\begin{figure}[htpb!]
	\centering
	\includegraphics[width=0.32\textwidth]{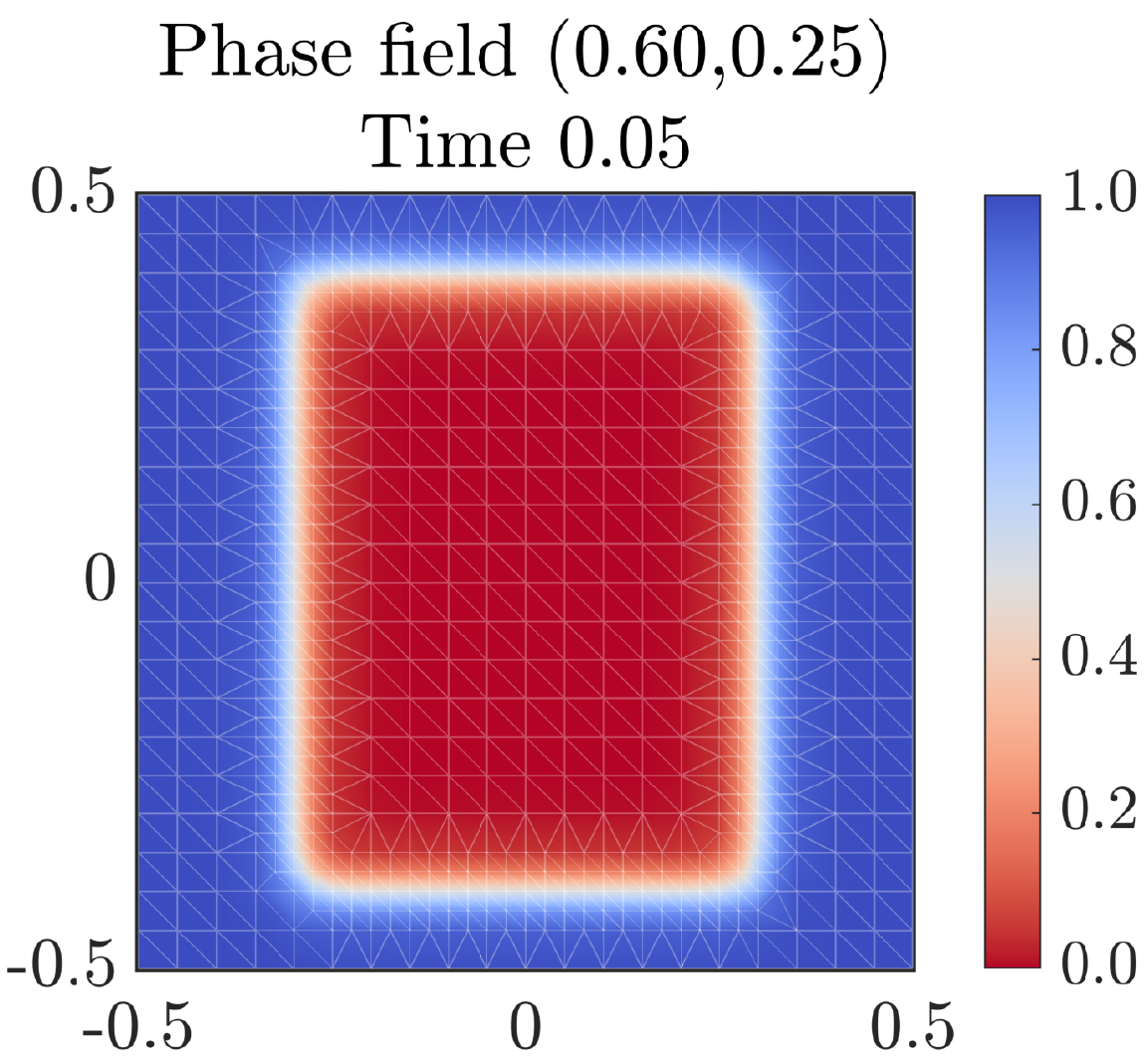}
	\includegraphics[width=0.32\textwidth]{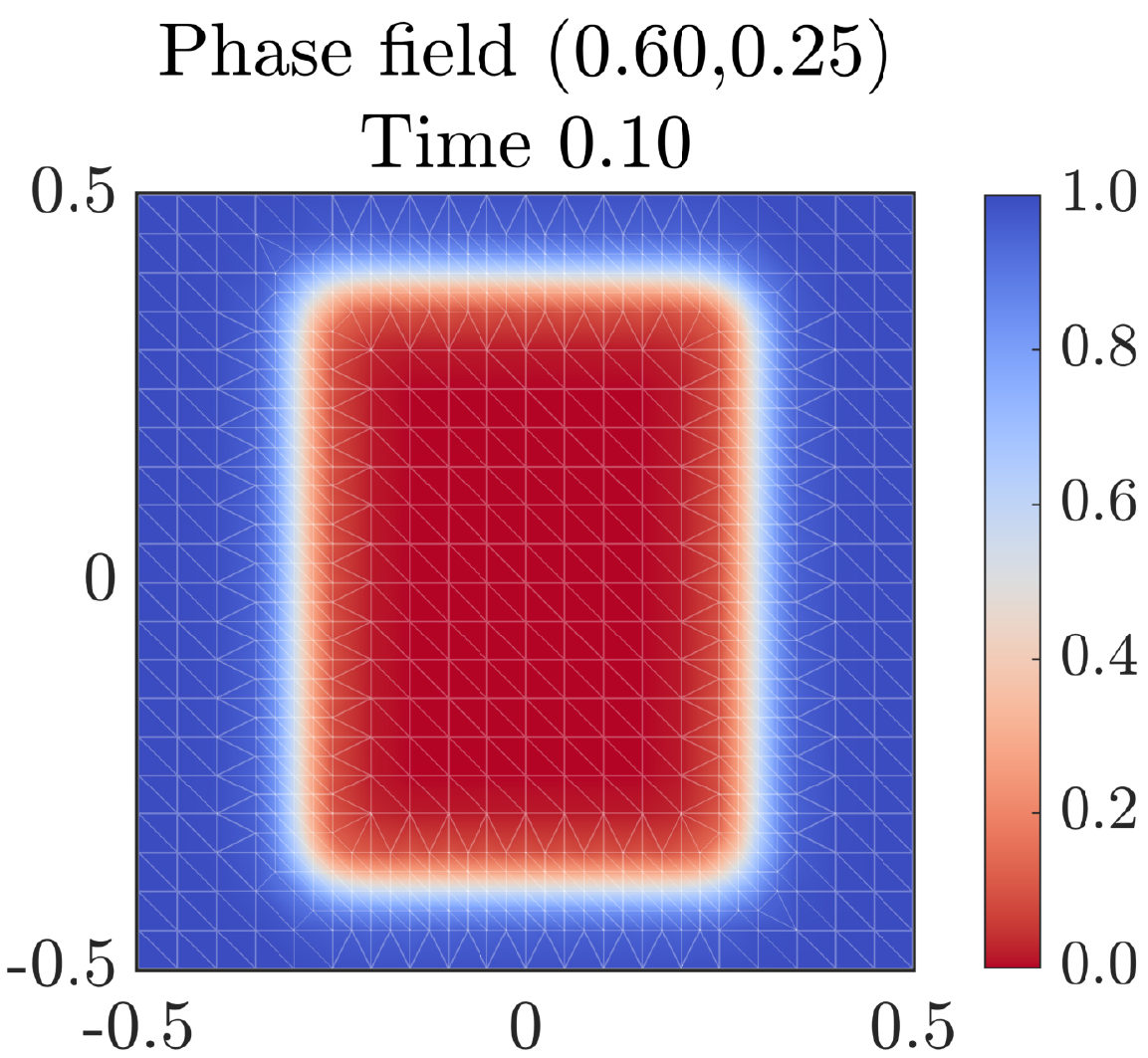}
	\includegraphics[width=0.32\textwidth]{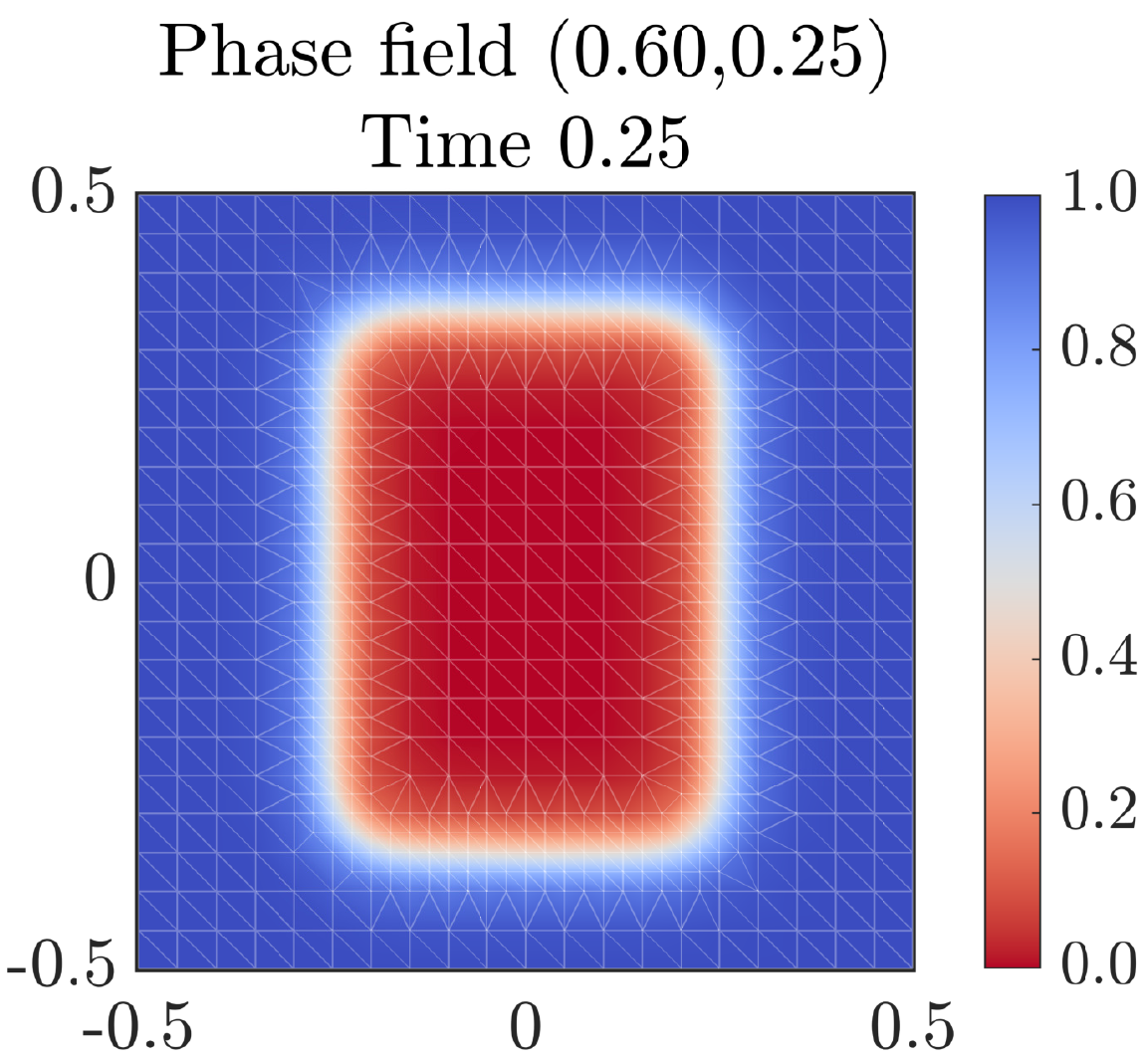}\\
	\includegraphics[width=0.32\textwidth]{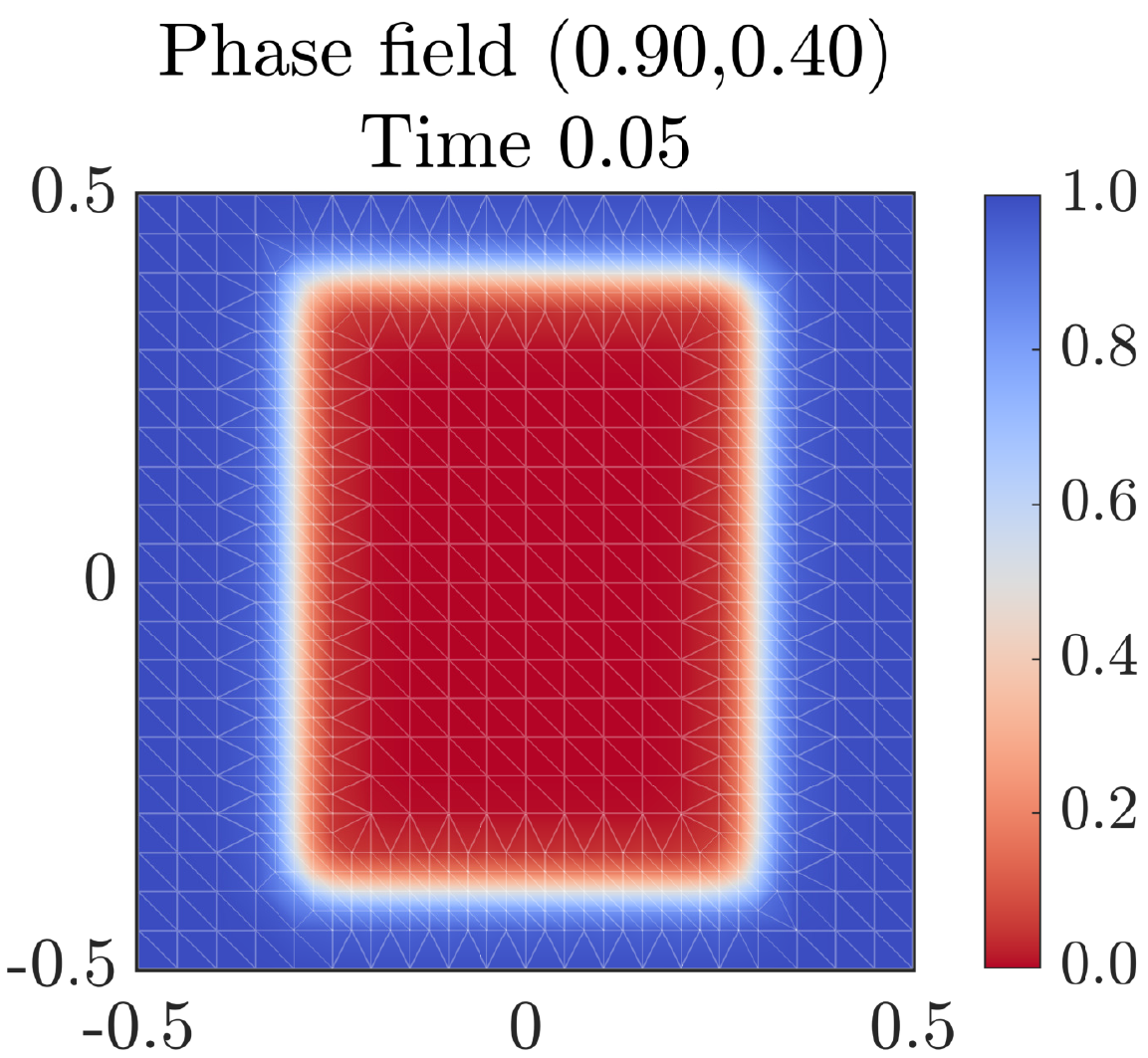}
	\includegraphics[width=0.32\textwidth]{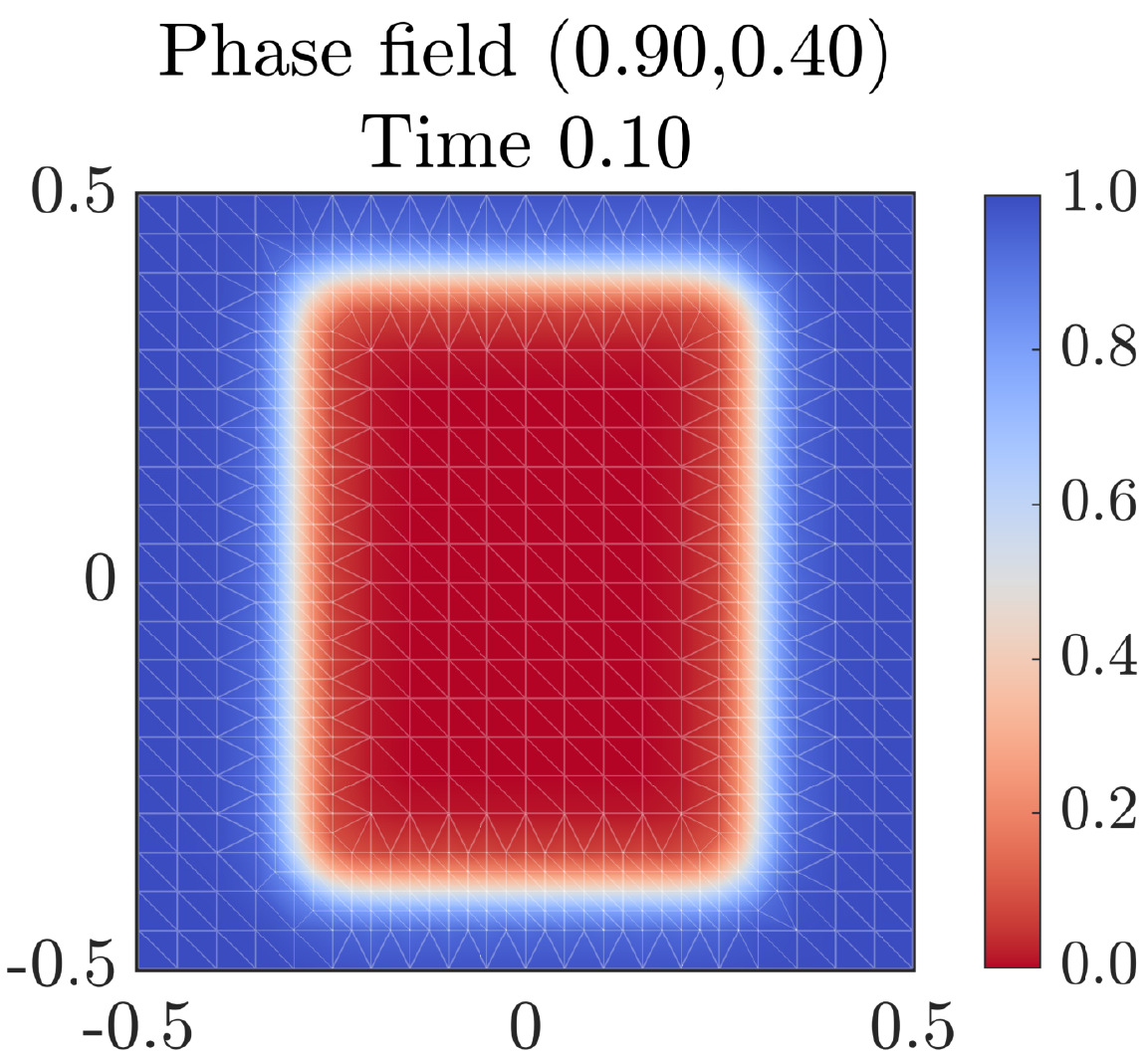}
	\includegraphics[width=0.32\textwidth]{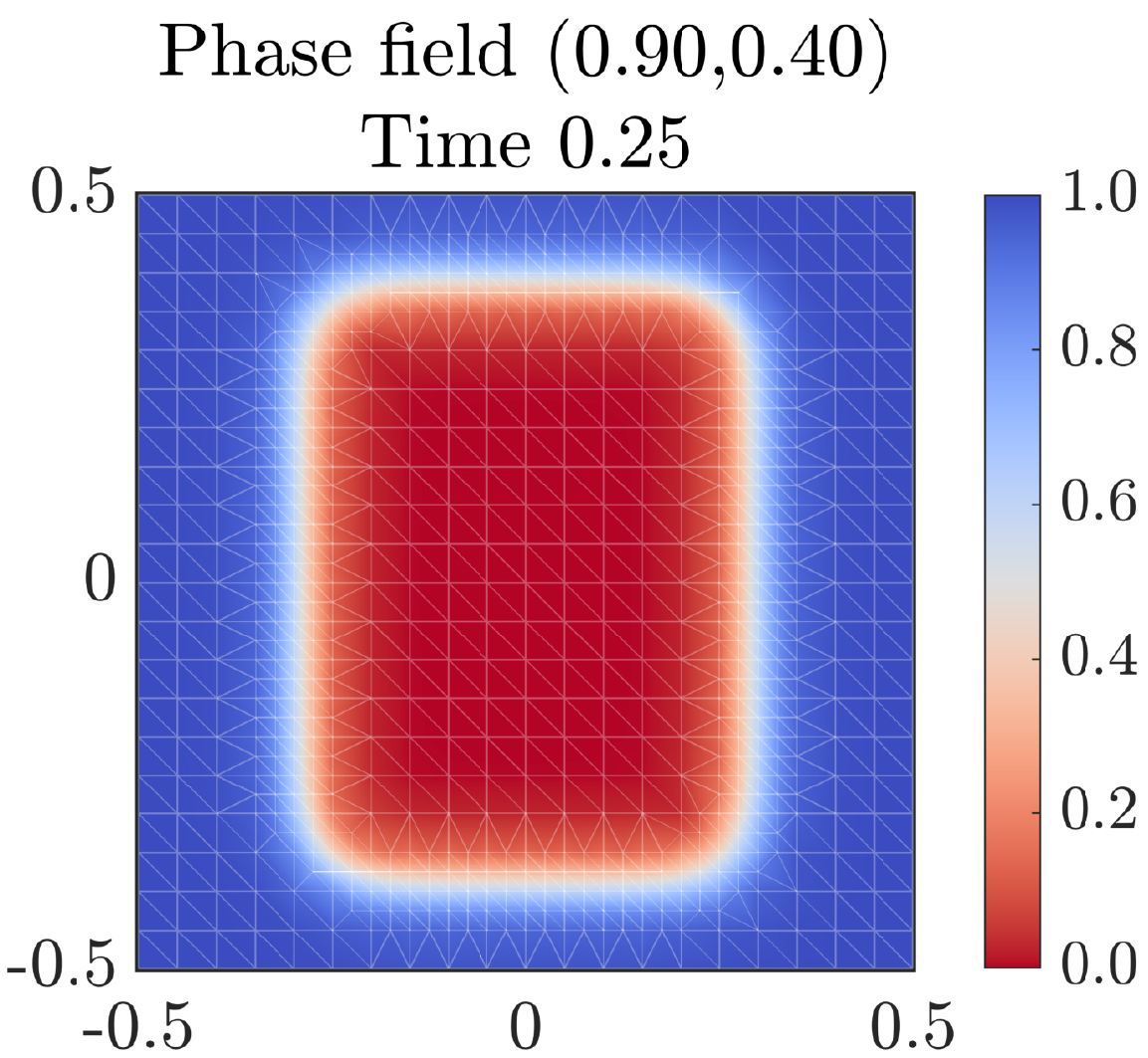}
	\captionof{figure}{The evolution of the phase fields $\phi_r$ corresponding to the macro-scale locations $\xnn=(0.6, 0.25)$ (top) and $\xnn=(0.9, 0.4)$ (bottom) at three times $t^n = 0.05, \, 0.10 $ and $0.25$ (left to right).}
	\label{fig:microscale2_Ex2}
\end{figure}

The 1D projection of the macro-scale solute concentration, pressure and porosity is displayed in \Cref{fig:sol1D_Ex2}. As expected, where the concentration decreases, the dissolution of the mineral is induced, which then increases the porosity. This effect is also evident in \Cref{fig:effec1D2_Ex2}, where the 1D projection of the effective parameters is displayed.

\begin{figure}[htpb!]
	\centering
	\includegraphics[width=0.32\textwidth]{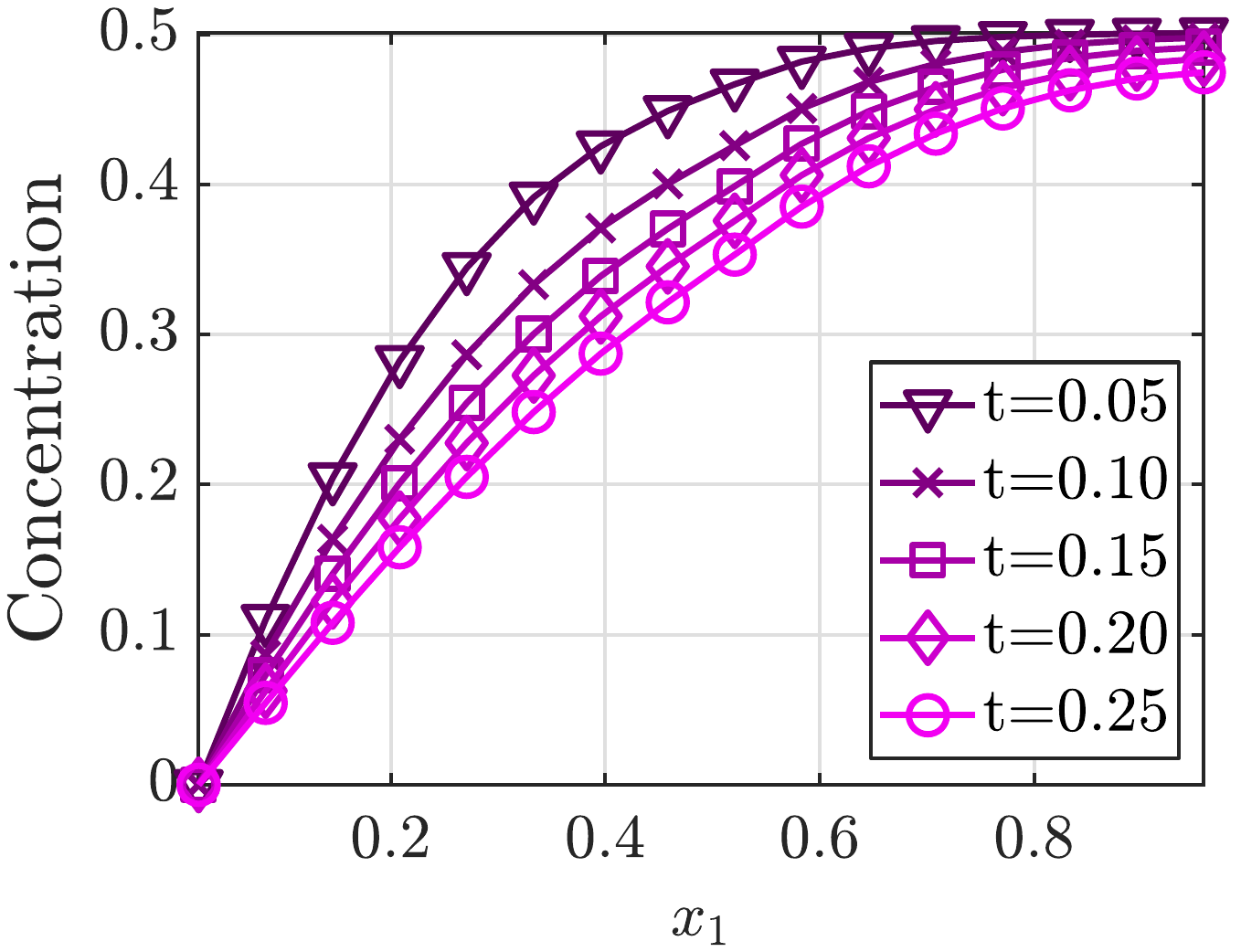}
	\includegraphics[width=0.32\textwidth]{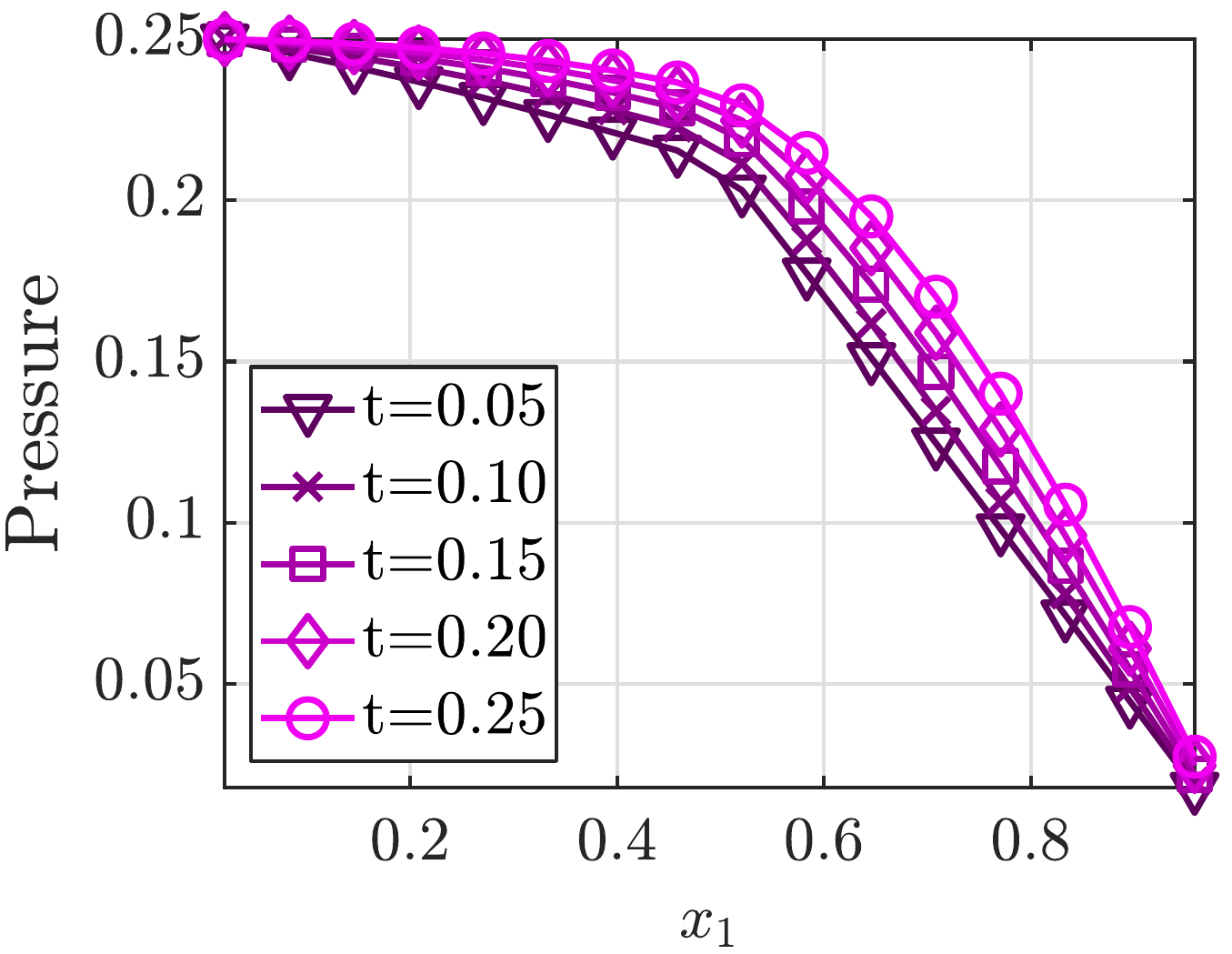}
	\includegraphics[width=0.32\textwidth]{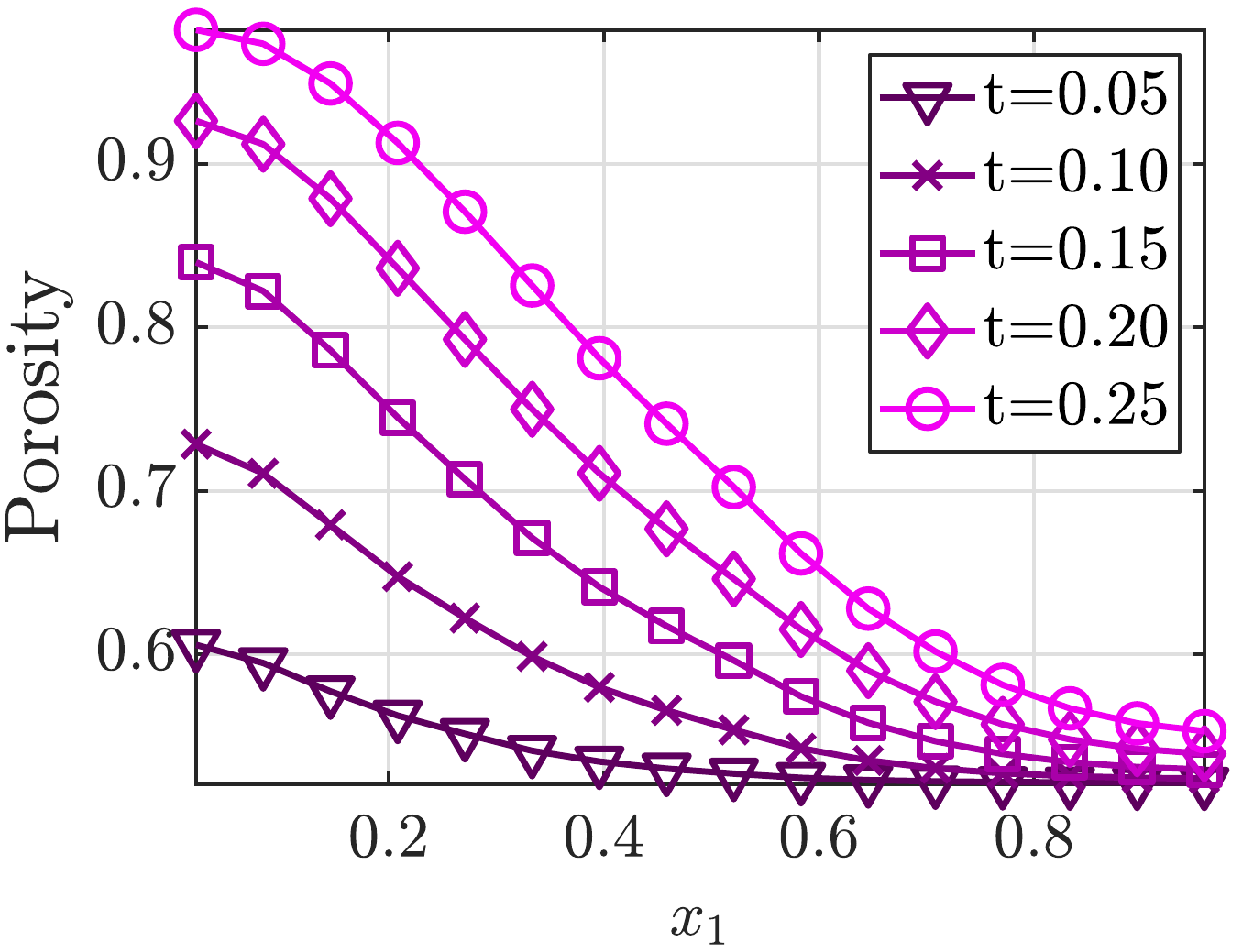}
	\captionof{figure}{The 1D projection of the concentration $u^n(\xnn)$, pressure $p(\xnn)$ and porosity $\bphi(\xnn)$ for five different times.}
	\label{fig:sol1D_Ex2}
\end{figure}

In this test case, the phase fields $\phi_l^0$ and $\phi_r^0$ are both asymmetric and for this reason, the expected results are anisotropic effective tensors. The non-diagonal components of $\aef$ and $\kef$ are however close to zero and can be neglected.
In \Cref{fig:effec1D2_Ex2} we display the diagonal components of both effective tensors. Notice the discontinuous behavior of the effective parameters as a result of the macro-scale heterogeneous distribution.

\begin{figure}[htpb!]
	\centering
	\includegraphics[width=0.4\textwidth]{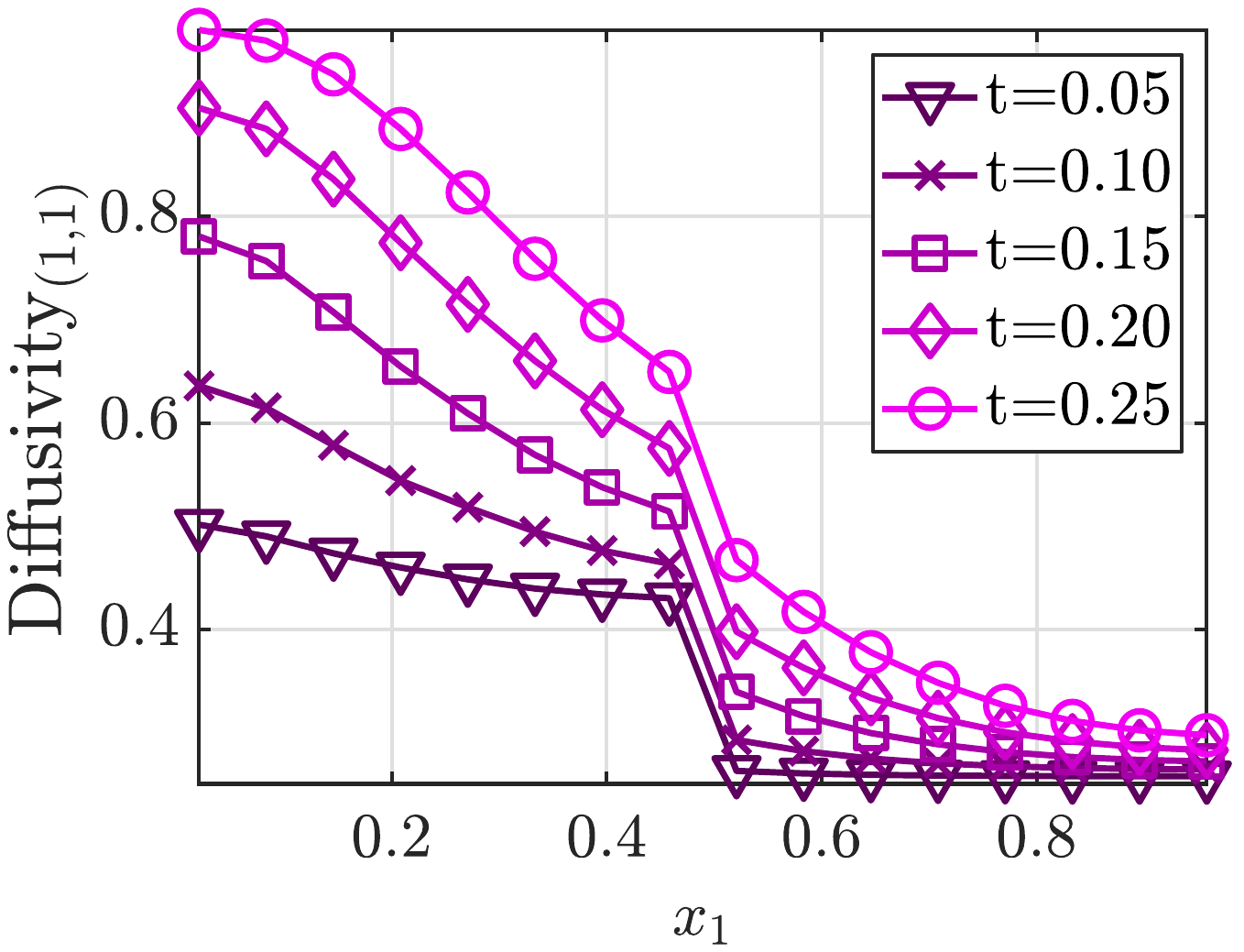}\hspace{0.5cm}
	\includegraphics[width=0.4\textwidth]{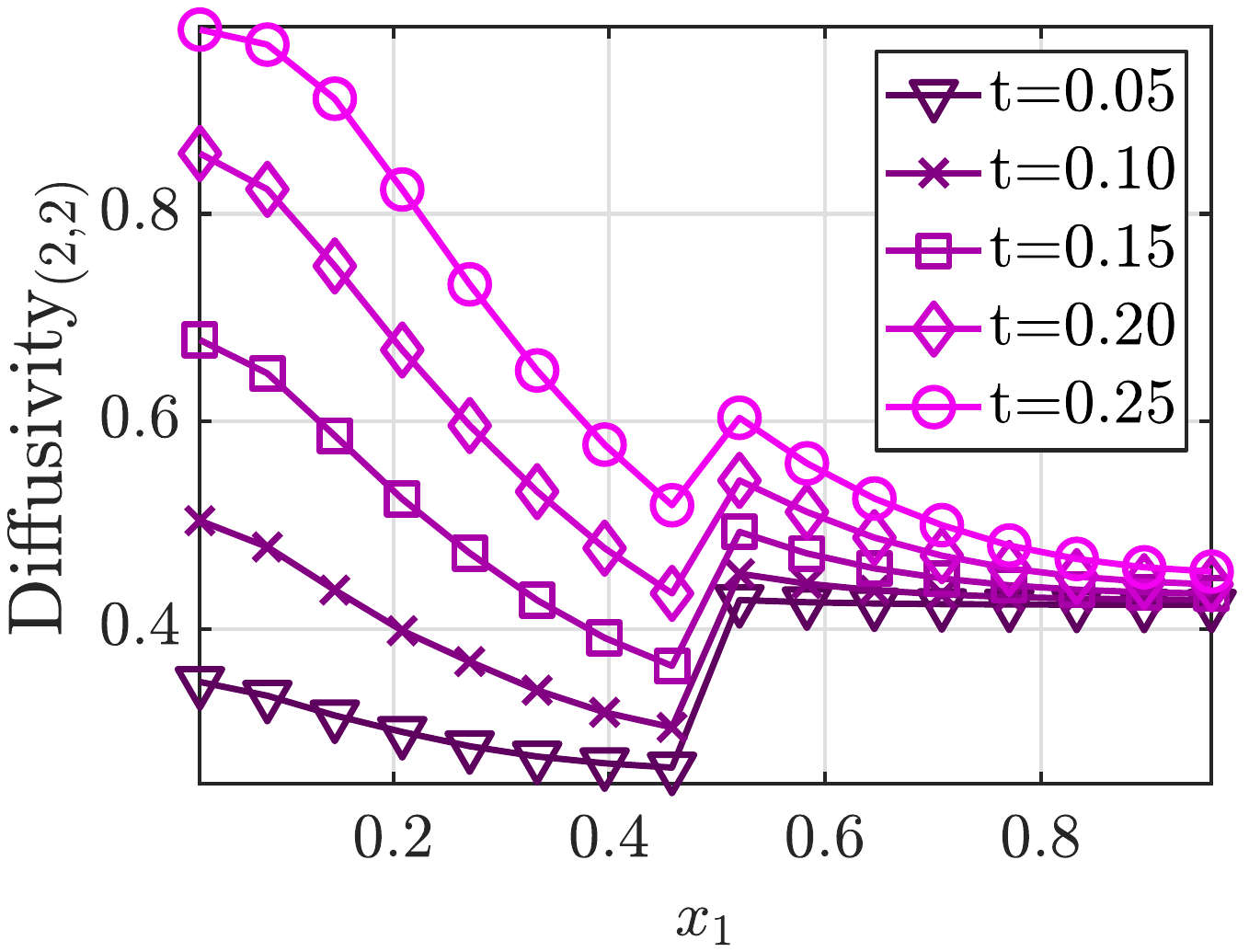}\\
	\includegraphics[width=0.4\textwidth]{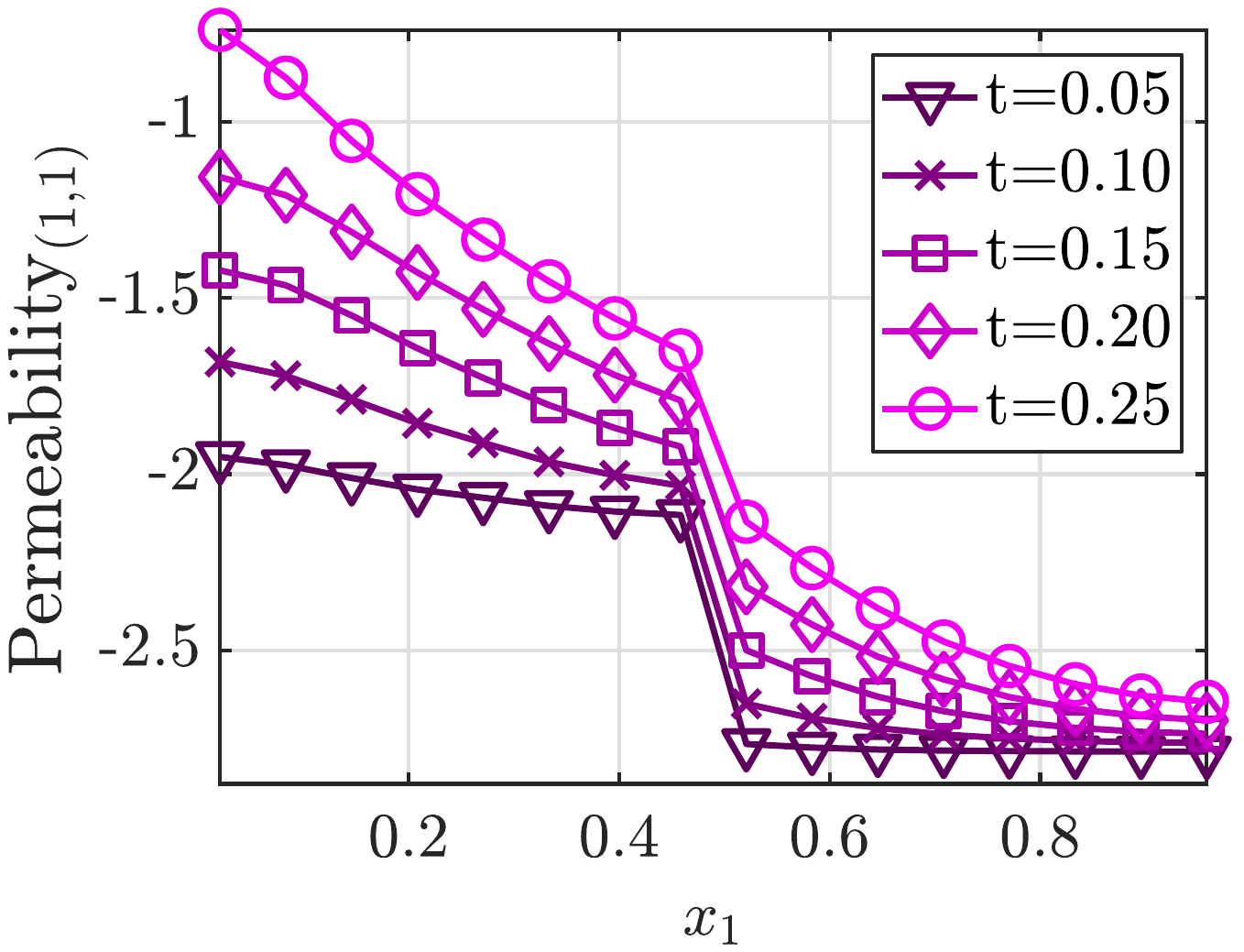}\hspace{0.5cm}
	\includegraphics[width=0.4\textwidth]{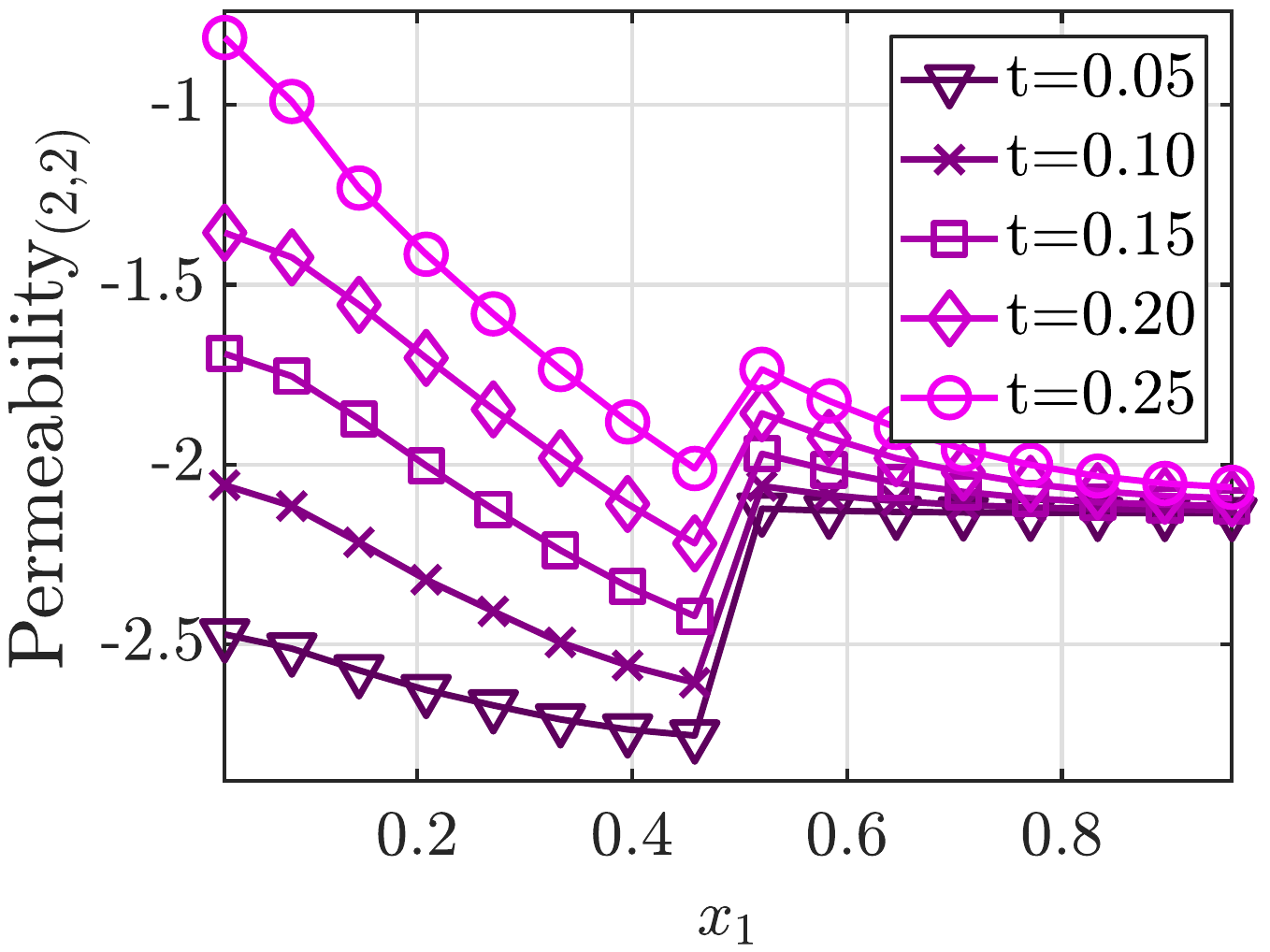}
	\captionof{figure}{The 1D projection of the diagonal components of effective diffusion tensor (top) and the effective permeability tensor ($Log_{10}$) (bottom) for five different times.}
	\label{fig:effec1D2_Ex2}
\end{figure}

In the 2D macro-scale domain we have $256$ elements. The porosity and the effective parameters must be updated only on the 32 elements located at the lowest part of the domain (1D projection) and copied (transferred in a sense explained in \Cref{sec:adaptivity}) over the whole 2D macro-scale domain. Following this, we obtain that the average number of degrees of freedom in both scales is $2.2$E+$5$ per time step. 
\begin{figure}[htpb!]
	\centering
	\includegraphics[width=0.45\textwidth]{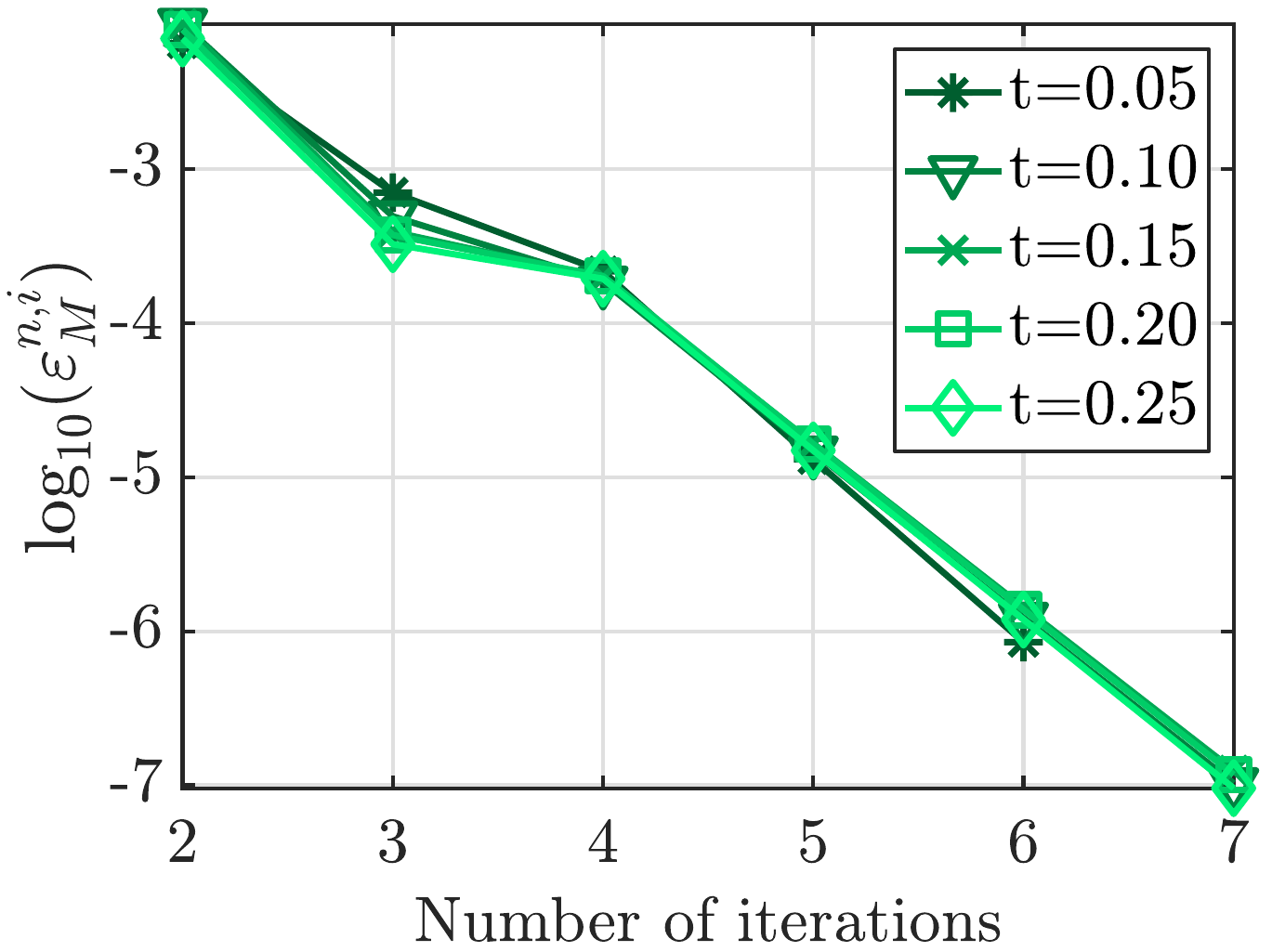}
	\captionof{figure}{The convergence of the multi-scale iterative scheme for five different times.}
	\label{fig:error_Ex2}
\end{figure}

Finally, in \Cref{fig:error_Ex2}, we show the convergence of $\epsilon_M^{n, i}$ at different times when the stopping criterion is $\textit{tol}_M=1$E-$6$. Notice that in this test case the total number of iterations remains constant in time and the convergence is shown to be linear. 

\section{Conclusions}
\label{sec:conclus}
We have presented a multi-scale iterative strategy that can be applied to models involving coupling of scales. In particular, we used this multi-scale iterative scheme to solve the two-scale phase-field model proposed in \cite{bringedal2019phase}. The resulting simulations show the influence of the micro-scale structural changes on the macro-scale parameters.

We calculate macro-scale quantities that are valid at the Darcy scale. Besides the macro-scale concentration and pressure, we calculate effective permeability, diffusivity, and porosity, which depend on the evolution of the phase field at the micro scale.
We have proven the convergence of the multi-scale non-linear iterative scheme and combined it with a robust micro-scale linearization strategy and adaptive strategies on both scales.
We use mesh refinement to reduce the numerical error in the solution of the phase-field evolution on the micro scale. For the macro scale, our adaptive strategy aims to localize where the effective parameters need to be recalculated. 
The multi-scale iterative scheme is shown to be convergent under some assumptions on the coupling parameter $\lstab$ and for sufficiently small time steps. However, the numerical examples show that the scheme converges under even milder restrictions on the coupling parameter $\lstab$ and the linearization parameter $\llin$.

Moreover, our numerical scheme can be parallelized. The local cell problems related to the micro scale are decoupled and can straightforwardly be solved in parallel.

It is relevant to mention that besides the theory considered in this paper, the applicability of this strategy is vast. Extensions of our adaptive algorithm, including more complex micro-scale models, are possible. The next research steps aim to prove the convergence of the full numerical scheme, including the error analysis of the micro-scale cell problems.

\subsection*{Acknowledgements}
The authors are supported by the Research Foundation - Flanders (FWO) through the Odysseus programme (Project G0G1316N). We thank Thomas Wick, Florin Adrian Radu, Kundan Kumar and Markus Gahn, who made valuable suggestions and contributed to the ideas behind this manuscript.

\bibliography{biblio_def}

\begin{thebibliography}{10}
\expandafter\ifx\csname url\endcsname\relax
  \def\url#1{\texttt{#1}}\fi
\expandafter\ifx\csname urlprefix\endcsname\relax\def\urlprefix{URL }\fi
\expandafter\ifx\csname href\endcsname\relax
  \def\href#1#2{#2} \def\path#1{#1}\fi

\bibitem{knabner1995analysis}
P.~Knabner, C.~van Duijn, S.~Hengst, An analysis of crystal dissolution fronts
  in flows through porous media. part 1: Compatible boundary conditions,
  Advances in Water Resources 18~(3) (1995) 171--185.

\bibitem{moszkowicz1996diffusion}
P.~Moszkowicz, J.~Pousin, F.~Sanchez, Diffusion and dissolution in a reactive
  porous medium: Mathematical modelling and numerical simulations, Journal of
  Computational and Applied Mathematics 66~(1-2) (1996) 377--389.

\bibitem{bouillard2007diffusion}
N.~Bouillard, R.~Eymard, R.~Herbin, P.~Montarnal, Diffusion with dissolution
  and precipitation in a porous medium: mathematical analysis and numerical
  approximation of a simplified model, ESAIM: Mathematical Modelling and
  Numerical Analysis 41~(6) (2007) 975--1000.

\bibitem{kumar2014convergence}
K.~Kumar, I.~S. Pop, F.~A. Radu, Convergence analysis for a conformal
  discretization of a model for precipitation and dissolution in porous media,
  Numerische Mathematik 127~(4) (2014) 715--749.

\bibitem{agosti2015analysis}
A.~Agosti, L.~Formaggia, A.~Scotti, Analysis of a model for precipitation and
  dissolution coupled with a {D}arcy flux, Journal of Mathematical Analysis and
  Applications 431~(2) (2015) 752--781.

\bibitem{kumar2016homogenization}
K.~Kumar, M.~Neuss-Radu, I.~S. Pop, Homogenization of a pore scale model for
  precipitation and dissolution in porous media, IMA Journal of Applied
  Mathematics 81~(5) (2016) 877--897.

\bibitem{hoffmann2017existence}
J.~Hoffmann, S.~Kr{\"a}utle, P.~Knabner, Existence and uniqueness of a global
  solution for reactive transport with mineral precipitation-dissolution and
  aquatic reactions in porous media, SIAM Journal on Mathematical Analysis
  49~(6) (2017) 4812--4837.

\bibitem{van2009FREE}
T.~van Noorden, Crystal precipitation and dissolution in a thin strip, European
  Journal of Applied Mathematics 20~(1) (2009) 69--91.

\bibitem{kumar2011effective}
K.~Kumar, T.~van Noorden, I.~S. Pop, Effective dispersion equations for
  reactive flows involving free boundaries at the microscale, Multiscale
  Modeling \& Simulation 9~(1) (2011) 29--58.

\bibitem{van2008stefan}
T.~van Noorden, I.~S. Pop, A {S}tefan problem modelling crystal dissolution and
  precipitation, IMA Journal of Applied Mathematics 73~(2) (2008) 393--411.

\bibitem{muntean2009moving}
A.~Muntean, M.~B{\"o}hm, A moving-boundary problem for concrete carbonation:
  global existence and uniqueness of weak solutions, Journal of Mathematical
  Analysis and Applications 350~(1) (2009) 234--251.

\bibitem{Kota}
K.~Kumazaki, A.~Muntean, Global weak solvability, continuous dependence on
  data, and large time growth of swelling moving interfaces., Interfaces and
  Free Boundaries 22~(1) (2020) 27--50.

\bibitem{kumar2013reactive}
K.~Kumar, M.~F. Wheeler, T.~Wick, Reactive flow and reaction-induced boundary
  movement in a thin channel, SIAM Journal on Scientific Computing 35~(6)
  (2013) B1235--B1266.

\bibitem{mabuza2014conservative}
S.~Mabuza, D.~Kuzmin, S.~{\v{C}}ani{\'c}, M.~Buka{\v{c}}, A conservative,
  positivity preserving scheme for reactive solute transport problems in moving
  domains, Journal of Computational Physics 276 (2014) 563--595.

\bibitem{mabuza2014nonlinear}
S.~Mabuza, D.~Kuzmin, A nonlinear {ALE-FCT} scheme for non-equilibrium reactive
  solute transport in moving domains, International Journal for Numerical
  Methods in Fluids 76~(11) (2014) 875--908.

\bibitem{mabuza2016modeling}
S.~Mabuza, S.~{\v{C}}ani{\'c}, B.~Muha, Modeling and analysis of reactive
  solute transport in deformable channels with wall adsorption--desorption,
  Mathematical Methods in the Applied Sciences 39~(7) (2016) 1780--1802.

\bibitem{muha2013}
B.~Muha, S.~{\v{C}}ani{\'c}, Existence of a weak solution to a nonlinear
  fluid–structure interaction problem modeling the flow of an incompressible,
  viscous fluid in a cylinder with deformable walls, Arch. Rational Mech. Anal.
  207 (2013) 919–--968.

\bibitem{bringedal2015model}
C.~Bringedal, I.~Berre, I.~S. Pop, F.~A. Radu, A model for non-isothermal flow
  and mineral precipitation and dissolution in a thin strip, Journal of
  Computational and Applied Mathematics 289 (2015) 346--355.

\bibitem{van2009crystal2}
T.~van Noorden, Crystal precipitation and dissolution in a porous medium:
  effective equations and numerical experiments, Multiscale Modeling \&
  Simulation 7~(3) (2009) 1220--1236.

\bibitem{schulz2017strong}
R.~Schulz, N.~Ray, F.~Frank, H.~Mahato, P.~Knabner, Strong solvability up to
  clogging of an effective diffusion--precipitation model in an evolving porous
  medium, European Journal of Applied Mathematics 28~(2) (2017) 179--207.

\bibitem{schulz2019beyond}
R.~Schulz, N.~Ray, S.~Zech, A.~Rupp, P.~Knabner, Beyond {K}ozeny-{C}arman:
  predicting the permeability in porous media, Transport in Porous Media
  130~(2) (2019) 487--512.

\bibitem{bringedal2016upscaling}
C.~Bringedal, I.~Berre, I.~S. Pop, F.~A. Radu, Upscaling of non-isothermal
  reactive porous media flow with changing porosity, Transport in Porous Media
  114~(2) (2016) 371--393.

\bibitem{caginalp1988dynamics}
G.~Caginalp, P.~C. Fife, Dynamics of layered interfaces arising from phase
  boundaries, SIAM Journal on Applied Mathematics 48~(3) (1988) 506--518.

\bibitem{ratz2016diffuse}
A.~R{\"a}tz, Diffuse-interface approximations of osmosis free boundary
  problems, SIAM Journal on Applied Mathematics 76~(3) (2016) 910--929.

\bibitem{van2011phase}
T.~van Noorden, C.~Eck, Phase field approximation of a kinetic moving-boundary
  problem modelling dissolution and precipitation, Interfaces and Free
  Boundaries 13~(1) (2011) 29--55.

\bibitem{redeker2016upscaling}
M.~Redeker, C.~Rohde, I.~S. Pop, Upscaling of a tri-phase phase-field model for
  precipitation in porous media, IMA Journal of Applied Mathematics 81~(5)
  (2016) 898--939.

\bibitem{bringedal2019phase}
C.~Bringedal, L.~von Wolff, I.~S. Pop, Phase field modeling of precipitation
  and dissolution processes in porous media: Upscaling and numerical
  experiments, Multiscale Modeling \& Simulation 18~(2) (2020) 1076--1112.

\bibitem{redeker2016pod}
M.~Redeker, B.~Haasdonk, A {POD-EIM} reduced two-scale model for precipitation
  in porous media, Mathematical and Computer Modelling of Dynamical Systems
  22~(4) (2016) 323--344.

\bibitem{efendiev2009multiscale}
Y.~Efendiev, T.~Y. Hou, Multiscale finite element methods: theory and
  applications, Vol.~4, Springer Science \& Business Media, 2009.

\bibitem{engquist2007heterogeneous}
B.~Engquist, X.~Li, W.~Ren, E.~Vanden-Eijnden, Heterogeneous multiscale
  methods: a review, Communications in Computational Physics 2~(3) (2007)
  367--450.

\bibitem{garttnernumerical}
S.~G{\"a}rttner, N.~Ray, P.~Frolkovic, P.~Knabner, Efficiency and accuracy of
  micro-macro models for mineral dissolution/precipitation, Preprint Series
  Angewandte Mathematik 407, University of Erlangen,
  https://www.math.fau.de/department/forschung/preprint-reihe-angewandte-mathematik/
  (2020).

\bibitem{ray2019numerical}
N.~Ray, J.~Oberlander, P.~Frolkovic, Numerical investigation of a fully coupled
  micro-macro model for mineral dissolution and precipitation, Computational
  Geosciences 23~(5) (2019) 1173--1192.

\bibitem{Manuela_proced}
M.~Bastidas, C.~Bringedal, I.~S. Pop, Numerical simulation of a phase-­field
  model for reactive transport in porous media, in: Numerical Mathematics and
  Advanced Applications ENUMATH 2019, Lecture Notes in Computational Science
  and Engineering, Vol. 139, Springer International, 2020.

\bibitem{brun2019iterative}
M.~K. Brun, T.~Wick, I.~Berre, J.~M. Nordbotten, F.~A. Radu, An iterative
  staggered scheme for phase field brittle fracture propagation with
  stabilizing parameters, Computer Methods in Applied Mechanics and Engineering
  361 (2020) 112752.

\bibitem{redeker2013fast}
M.~Redeker, C.~Eck, A fast and accurate adaptive solution strategy for
  two-scale models with continuous inter-scale dependencies, Journal of
  Computational Physics 240 (2013) 268--283.

\bibitem{heister2015primal}
T.~Heister, M.~F. Wheeler, T.~Wick, A primal-dual active set method and
  predictor-corrector mesh adaptivity for computing fracture propagation using
  a phase-field approach, Computer Methods in Applied Mechanics and Engineering
  290 (2015) 466--495.

\bibitem{pop2004mixed}
I.~S. Pop, F.~A. Radu, P.~Knabner, Mixed finite elements for the {R}ichards’
  equation: linearization procedure, Journal of computational and applied
  mathematics 168~(1-2) (2004) 365--373.

\bibitem{list2016study}
F.~List, F.~A. Radu, A study on iterative methods for solving {R}ichards’
  equation, Computational Geosciences 20~(2) (2016) 341--353.

\bibitem{schulzdegenerate}
R.~Schulz, Degenerate equations in a diffusion–precipitation model for
  clogging porous media, European Journal of Applied Mathematics (2019) 1–20.

\bibitem{schulz2020degenerate}
R.~Schulz, Degenerate equations for flow and transport in clogging porous
  media, Journal of Mathematical Analysis and Applications 483~(2) (2020)
  123613.

\bibitem{bringedal2017effective}
C.~Bringedal, K.~Kumar, Effective behavior near clogging in upscaled equations
  for non-isothermal reactive porous media flow, Transport in Porous Media
  120~(3) (2017) 553--577.

\bibitem{Chen}
X.~Chen, D.~Hilhorst, E.~Logak, Mass conserving allen–cahn equation and
  volume preserving mean curvature flow, Interfaces and Free Boundaries 12
  (2010) 527–--549.

\bibitem{Carina_proced}
C.~Bringedal, A conservative phase-field model for reactive transport, in:
  R.~Kl{\"o}fkorn, E.~Keilegavlen, A.~F. Radu, J.~Fuhrmann (Eds.), Finite
  Volumes for Complex Applications IX - Methods, Theoretical Aspects, Examples,
  Vol. 323 of Springer Proceedings in Mathematics \& Statistics, Springer
  International Publishing, 2020, pp. 537--545.

\bibitem{garcke2015numerical}
H.~Garcke, C.~Hecht, M.~Hinze, C.~Kahle, Numerical approximation of phase field
  based shape and topology optimization for fluids, SIAM Journal on Scientific
  Computing 37~(4) (2015) A1846--A1871.

\bibitem{mikelic2013convergence}
A.~Mikeli{\'c}, M.~F. Wheeler, Convergence of iterative coupling for coupled
  flow and geomechanics, Computational Geosciences 17~(3) (2013) 455--461.

\bibitem{frank2018finite}
F.~Frank, C.~Liu, F.~O. Alpak, B.~Riviere, A finite volume/discontinuous
  galerkin method for the advective cahn--hilliard equation with degenerate
  mobility on porous domains stemming from micro-ct imaging, Computational
  Geosciences 22~(2) (2018) 543--563.

\bibitem{friedman1988combustion}
A.~Friedman, A.~E. Tzavaras, Combustion in a porous medium, SIAM Journal on
  Mathematical Analysis 19~(3) (1988) 509--519.

\bibitem{friedman1992transport}
A.~Friedman, P.~Knabner, A transport model with micro and macro-structure,
  Journal of differential equations 98~(2) (1992) 328--354.

\bibitem{muntean2010multiscale}
A.~Muntean, M.~Neuss-Radu, A multiscale galerkin approach for a class of
  nonlinear coupled reaction--diffusion systems in complex media, Journal of
  Mathematical Analysis and Applications 371~(2) (2010) 705--718.

\bibitem{cioranescu1999introduction}
D.~Cioranescu, P.~Donato, An introduction to homogenization, Vol.~17, Oxford
  university press Oxford, 1999.

\bibitem{Ladyzhenskaya}
O.~A. Ladyzhenskaya, N.~N. Ural'tseva, Linear and Quasilinear elliptic
  equations, Vol.~46 of Mathematics in Science and Engineering, Academic Press
  New York and London, 1968.

\bibitem{ray2018old}
N.~Ray, A.~Rupp, R.~Schulz, P.~Knabner, Old and new approaches predicting the
  diffusion in porous media, Transport in Porous Media 124~(3) (2018) 803--824.

\bibitem{storvik2019optimization}
E.~Storvik, J.~W. Both, K.~Kumar, J.~M. Nordbotten, F.~A. Radu, On the
  optimization of the fixed-stress splitting for biot's equations,
  International Journal for Numerical Methods in Engineering 120~(2) (2019)
  179--194.

\bibitem{bahriawati_three_2005}
C.~Bahriawati, C.~Carstensen, Three {MATLAB} implementations of the
  lowest-order {R}aviart-{T}homas {MFEM} with a posteriori error control,
  Computational Methods in Applied Mathematics 5~(4) (2005) 333--361.

\bibitem{boffi2013mixed}
D.~Boffi, F.~Brezzi, M.~Fortin, Mixed finite element methods and applications,
  Vol.~44, Springer, 2013.

\end{thebibliography}

\end{document}